\documentclass{article}
\usepackage{hyperref}    
\usepackage{graphicx}    
\usepackage{color}    
\usepackage{multirow}    
\usepackage{amsmath,amsfonts,amssymb,latexsym,amsthm,multicol}
\usepackage[english]{babel}

\usepackage{bm}
\providecommand{\B}{}
\renewcommand{\B}{\bm}
\newcommand{\D}{\partial}
\renewcommand{\le}{\leqslant}
\renewcommand{\ge}{\geqslant}

\newtheorem{lemma}{Lemma}

\newtheorem{proposition}{Proposition}
\theoremstyle{definition}
\newtheorem{definition}{Definition}

\title{Segregated Runge~-- Kutta schemes\\ for the time integration\\ of the incompressible Navier~-- Stokes equations in presence of pressure stabilization}
\author{P.A.~Bakhvalov}

\date{June 11, 2025}
\begin{document}

\numberwithin{equation}{section}
\maketitle

\sloppy 

\begin{abstract}
Segregated Runge~-- Kutta (SRK) schemes are time integration methods for the incompressible Navier~-- Stokes equations. In this approach, convection and diffusion can be independently treated either explicitly or implicitly, which in particular allows to construct implicit-explicit (IMEX) methods. 
Original SRK schemes (Colom\'{e}s, Badia, IJNME, 2015) are designed for finite-element methods that satisfy the inf-sup condition. In this paper, the idea of SRK schemes is generalized to spatial discretizations with pressure stabilization. In the numerical experiments, SRK schemes are demonstrated with both finite-difference and finite element spatial discretizations. Numerical results show that one of the SRK schemes outperforms the third-order multistep projection-based method in terms of accuracy while preserving the computational costs.
\end{abstract}


\section{Introduction}

A spatial discretization of the unsteady incompressible Navier~-- Stokes equations (NSE) usually yields a differential-algebraic equation (DAE). Starting from Chorin \cite{Chorin1968}, it is usually solved by fractional-step projection-type methods. These methods have lots of modifications, see an overview in \cite{Guermond2006}. When the aim is to achieve a high-order accuracy in time, there are several approaches: multistep methods \cite{Guermond2006, Trias2011, Trias2014, Franco2020}, Runge~-- Kutta (RK) methods \cite{John2006, Sanderse2012, Sanderse2013, Bassi2015, Noventa2016, Komen2021, Cai2025}, deferred correction methods \cite{Guesmi2023}, space-time finite element methods \cite{Tavelli2016}, see also references therein.


Runge~-- Kutta (RK) schemes are known as methods for solving ordinary differential equations (ODEs). When appiled to stiff systems ($\dot{x} = f(x,y)$, $\varepsilon\dot{y}  = g(x,y)$), some of them have a valid limit as $\varepsilon \to +0$. Therefore, they are capable to solve DAEs as well. Such schemes are called stiffly accurate \cite{Hairer1996}. 

Stiffly accurate schemes are never explicit. Application of an implicit scheme to the Navier~-- Stokes equations leads to nonlinear systems 
where pressure and velocity are coupled. Solving such systems at each timestep is impractical in terms of CPU cost. Nevertheless, this approach is sometimes used because it allows to demonstrate a high order convergence \cite{John2006, Sanderse2013, Bassi2015, Noventa2016}. The idea leading to the segregated Runge~-- Kutta schemes is to use an implicit-explicit (IMEX) RK scheme with a stiffly accurate implicit counterpart and thus to avoid coupled linear systems.

For finite element methods satisfying the inf-sup condition \cite{Elman2005}, pressure is uniquely defined for a given velocity field. This allows to use the space state method (see \cite{Hairer1996}), i. e. to reduce a differential-algebraic equation (for velocity and pressure) to an ordinary differential equation for velocity. There are several approaches to solve this ODE using Runge~-- Kutta schemes.

\begin{itemize}
\item Apply an implicit RK scheme. This returns us back to the need of solving coupled systems. See a recent review \cite{Cai2025} of these methods. 

\item Apply an explicit RK scheme. Then one needs to solve an elliptic equation for pressure at each RK stage, and everything else is explicit. The resulting schemes are called half-explicit Runge~-- Kutta (HERK) methods \cite{Hairer1996, Sanderse2012}.

\item Apply an implicit-explicit (IMEX) RK scheme. Pressure gradient is taken explicitly to avoid coupled systems. The resulting schemes are called segregated Runge~-- Kutta (SRK) methods \cite{Colomes2015}. Convection and diffusion terms can be taken either to the explicit part, or to the implicit part, or to different parts. This leads, respectively, to explicit SRK, implicit SRK, and IMEX SRK schemes. 
\end{itemize}

In finite-difference schemes and in finite-element methods that do not satisfy the inf-sup condition (for instance, methods with equal-order interpolation), pressure is no longer a function of velocity, and a pressure stabilization technique is necessary. Therefore, the SRK schemes as defined in \cite{Colomes2015} are not possible. The authors explicitly state this in their subsequent paper (\cite{Colomes2016}, Remark~3.2).


In this paper we overcome this limitation and construct SRK schemes for spatial discretizations with pressure stabilization. Although the main idea remains the same, the new schemes do not reduce to the original SRK schemes \cite{Colomes2015} as the pressure stabilization terms vanish. If the the underlying IMEX method is of type CK (see the classification in \cite{Boscarino2007}), for an IMEX method with $s$ explicit stages, the original SRK scheme calls the pressure solver $s$ times per timestep, while our scheme does this only $s-1$ times. Compared to the RK-based approach from \cite{Komen2021}, SRK schemes (both original  \cite{Colomes2015} and presented in the current paper) do not require an inner pressure-velocity coupling loop. 

The approach we use allows several options regarding underlying IMEX methods and stabilization terms. Not all choices lead to robust and accurate time integration schemes. We study different possibilities and choose three IMEX methods that show a good behavior and may be used for scale-resolving simulations. The SRK scheme based on the 4-th order adaptive Runge~-- Kutta method \cite{Kennedy2003} outperforms the third-order multistep method for all tests we used for comparison.

The rest of the paper is organized as follows. In Section~\ref{sect:stabilization} we introduce a general form of a semidiscrete scheme with pressure stabilization. 
In Section~\ref{sect:spatial} we present examples of spatial discretizations with a pressure correction, from both finite-difference and finite-element families. In Section~\ref{sect:IMEX} we describe IMEX RK methods. In Section~\ref{sect:main} we present new segregated Runge~-- Kutta schemes. Section~\ref{sect:verification} contains verification results, including a comparison with the third-order multistep projection-based scheme. For completeness, we consider the outcome of another choice of IMEX methods in Section~\ref{sect:IMEXv2}. In Conclusion we summarize the results and select three SRK schemes that perform the best.

\section{Pressure stabilization}
\label{sect:stabilization}

\subsection{Navier~-- Stokes equations}

In this paper, we consider the incompressible Navier~-- Stokes equations
\begin{equation}
\nabla \cdot \tilde{\B{u}} = 0,
\label{eq_NS_1}
\end{equation}
\begin{equation}
\frac{\D \tilde{\B{u}}}{\D t} + (\tilde{\B{u}} \cdot \nabla)\tilde{\B{u}} + \nabla \tilde{p} = \nu \triangle \tilde{\B{u}} + \B{f}_{source}(t, \B{r})
\label{eq_NS_2}
\end{equation}
in a bounded domain $\Omega \subset \mathbb{R}^d$ with initial data $\tilde{\B{u}}(0,\B{r}) = \tilde{\B{u}}_0(\B{r})$ and the Dirichlet boundary conditions $\tilde{\B{u}}(t,\B{r}) = \tilde{\B{u}}_b(t,\B{r})$, $\B{r} \in \D \Omega$. Here $\tilde{\B{u}}$ is the velocity vector, $\tilde{p}$ is the pressure, $\nu \ge 0$ is the viscosity coefficient, $\B{f}_{source}$ is a source term.  We also consider \eqref{eq_NS_1}--\eqref{eq_NS_2} in a cube with the periodic conditions.

\subsection{General semidiscrete form}

For \eqref{eq_NS_1}--\eqref{eq_NS_2}, consider a discretization of the following form.
Let $V_h$ and $Q_h$ be finite-dimensional spaces for the velocity and pressure, correspondingly, equipped with scalar products. Let $G \ : \ Q_h \to V_h$ be a discrete gradient operator, $D \ :\ V_h \to Q_h$ be a discrete divergence operator, and $L\ :\ Q_h \to Q_h$ be a discrete Laplace operator. If the Dirichlet boudnary conditions are set for the velocity, then $L$ approximates the Laplace operator with the Neumann boundary conditions. If $u(t) \in V_h$ is a velocity field at a time $t$, then its divergence is approximated as $D u(t) - H(t)$. The term $H(t) \in Q_h$ is due to the non-homogeneous boundary conditions; $H(t) \equiv 0$ if $\tilde{\B{u}}_b \equiv 0$. 

Assume that
\begin{equation}
D = -G^*,\quad L = L^* < 0, \quad S := (-L) - (-DG) \ge 0.
\label{eq_gen_assumptions}
\end{equation}
Here the stars denote the conjugation. Operators satisfying \eqref{eq_gen_assumptions} appear naturally in both finite-difference and finite-element methods. We show examples in Section~\ref{sect:spatial}.

An approximation of \eqref{eq_NS_2} may be written in the general form
\begin{equation}
\frac{du}{dt} + \B{F}_{conv}(u) + Gp = \B{F}_{diff}(u) + \B{F}_{source}(t)
\label{eq_momentun_base}
\end{equation}
where $\B{F}_{conv}$, $\B{F}_{diff}$, $\B{F}_{source}$ are some discretizations of the convection, diffusion, and source terms, respectively. 

If $\mathrm{Ker}\,G \ne \{0\}$, then equation \eqref{eq_momentun_base} ignores pressure fields from $\mathrm{Ker}\,G$, and this is why the pressure stabilization is necessary. So we consider the following approximation of \eqref{eq_NS_2}:
\begin{equation}
D u + \sigma_0 S p + \sigma_1 S \frac{dp}{dt} = H(t),
\label{eq_continuity_base}
\end{equation}
where $\sigma_0, \sigma_1 \ge 0$, $\sigma_0 + \sigma_1 > 0$. By the dimensional considerations, $\sigma_0 = O(\tau)$ and $\sigma_1 = O(\tau^2)$ where $\tau$ is the timestep. For the discussion on the initial data for \eqref{eq_momentun_base}--\eqref{eq_continuity_base}, see Appendix~\ref{sect:ID}.

As revealed in \cite{Zienkiewicz1995, Codina1997}, stabilization terms in the continuity equation do appear in projection-type methods (unless one uses inf-sup-stable finite elements or non-FE methods with similar properties). It is not a common practice to write these terms explicitly, but we believe this clarifies the analysis.

{\it Remark.} If $H = 0$, then \eqref{eq_continuity_base} is equivalent to 
$$
\sigma_0 p + \sigma_1 \frac{dp}{dt} = L^{-1} D \left(u + \sigma_0 G p + \sigma_1 G \frac{dp}{dt}\right).
$$
So we actually use the combination $L^{-1}D$ and not operators $L$ and $D$ independently. As a consequence, changing the scalar product in $Q_h$ does not change the scheme.

{\it Remark.} When solving ordinary differential equations using Runge~-- Kutta methods, changing the timestep size from one step to another does not require any adjustment of the algorithm, and does not cause any loss of accuracy. This is valid no more when considering  \eqref{eq_momentun_base}--\eqref{eq_continuity_base}. The reason is that \eqref{eq_continuity_base} depends on the timestep via $\sigma_0$ and $\sigma_1$. So a change of the timestep generates a non-physical perturbation. All the computations in this paper were conducted with a fixed timestep. The possibility of the use of a variable timestep size requires further analysis.

{\it Remark.} Writing \eqref{eq_momentun_base} we assume that the discretization of the convective fluxes does not involve pressure. The converse (which takes place in the Rhie~-- Chow method \cite{Rhie1983, Pascau2011} and is necessary to provide a kinetic-energy preserving discretization in some approaches) complicates the construction of high-order methods for the time integration. We believe that one can bypass the high-order treatment of pressure in the convection term without compromising the overall accuracy, but we have not studied this issue in details.

\subsection{Discussion on two types of stabilization}
\label{sect:stabtypes}

As we show in Section~\ref{sect:firstorder}, first-order segregated Runge~-- Kutta schemes belong to the family fractional-step projection-type methods. We will see that the case $\sigma_0 > 0$, $\sigma_1 = 0$ corresponds to the non-incremental pressure-correction scheme (in terms of \cite{Guermond2006}) or to the total pressure method (in terms of \cite{Trias2014}). The case $\sigma_0 = 0$, $\sigma_1 > 0$ corresponds to the standard incremental pressure-correction schemes (in terms of \cite{Guermond2006}) or to the pressure correction method (in terms of \cite{Trias2014}).


In this section, we study the behavior of potential velocity modes depending on $\sigma_0$ and $\sigma_1$. Assume that the convective term satisfies $\B{u} \cdot \B{F}_{conv}(\B{u}) \equiv 0$, and there is no diffusion and source terms: $\B{F}_{diff} \equiv 0$, $\B{F}_{source} \equiv 0$, $H \equiv 0$. Then
$$
\frac{d}{dt} \frac{(u, u)}{2} = - (u, G p) = (p, D u) = - \sigma_0 (p, Sp) - \sigma_1 \frac{d}{dt}\frac{(p,Sp)}{2}.
$$

If $\sigma_1 = 0$, then we have
$$
\frac{d}{dt} \frac{(u, u)}{2} = - \sigma_0 (p, Sp) \le 0,
$$
so $S$ is responsible for the numerical dissipation of the kinetic energy. 

If $\sigma_0 = 0$, then we have
$$
\frac{d}{dt} \left(\frac{(u,u)}{2} + \frac{(p,Sp)}{2}\right) = 0,
$$
which makes an impression that there is a new conservation law and, as a consequence, potential velocity modes do not decay. However, in this case we also have a dissipation mechanism. Let us demonstrate this using the spectral analysis assuming $\B{F}_{conv} = 0$. Then \eqref{eq_momentun_base}, \eqref{eq_continuity_base} reduce to
$$
\frac{du}{dt} + Gp = 0, \quad Du + S \left(\sigma_0 p + \sigma_1 \frac{dp}{dt}\right) = 0
$$
Taking $du/dt = -Gp$ to the time derivative of the second equation we get
$$
- DGp + S \left(\sigma_0 \frac{dp}{dt} + \sigma_1 \frac{d^2p}{dt^2}\right) = 0.
$$
Let $\hat{p}$ be a common eigenvector of $DG$ and $L$. Then for some $\beta \in [0, 1]$ there holds
\begin{equation}
DG\hat{p} = (1-\beta) L\hat{p}, \quad  S\hat{p} = - \beta L\hat{p}
\label{eq_def_Omega}
\end{equation}
Looking for a solution of the form $p(t) = \hat{p} \exp(\mu t)$ we come to the equation 
\begin{equation}
(1-\beta) + \beta (\sigma_0 \mu + \sigma_1 \mu^2) = 0.
\label{eq_def_mu}
\end{equation}

If $\sigma_0 > 0$ and $\sigma_1 = 0$, then $\mu = - (\beta^{-1}-1) \sigma_0^{-1} \le 0$, which corresponds to the dissipation. 

If $\sigma_0 = 0$ and $\sigma_1 > 0$, then we have $\mu = \pm i \omega$, $\omega = (\beta^{-1}-1)^{1/2} \sigma_1^{-1/2}$. By the dimensional considerations, $\sigma_1 \sim \tau^2$ where $\tau$ is the timestep. For long waves we have $\beta \ll 1$ and therefore $\omega \tau \gg 1$, so the frequencies of these modes cannot be resolved by the timestep. If the stability function of time integration method satisfies \mbox{$|R(\infty)|<1$} (for instance, if the method is L-stable), then the time integration is responsible for the damping of potential velocity modes, except for the shortest ones. The modes corresponding to $\beta \approx 1$ are indeed problematic, as we will see in Section~\ref{sect:spectral2} when analyzing the stability of SRK methods.

If $S \ge 0$ is not satisfied, then for some eigenvectors we have $\beta > 1$, which causes the exponential growth of the solution.

A remarkable fact is that we cannot use pressure derivatives of order greater than one  in \eqref{eq_continuity_base}. In that case, instead of \eqref{eq_def_mu} we would have
$$
\sigma_0 \mu + \ldots + \sigma_r \mu^{r+1} = -(\beta^{-1} - 1)
$$
with $r \ge 2$, $\sigma_r \ne 0$. As $\beta \rightarrow 0$, the terms at $\sigma_0, \ldots, \sigma_{r-1}$ may be dropped. The $r+1$ roots of the resulting equation are equidistantly distributed on a circumference, so at least one of them satisfies $\mathrm{Re}\,\mu > 0$. Thus, a higher order stabilization is not possible. 

In practical applications, an optimal choice of $\sigma_0$ and $\sigma_1$ will likely depend on a turbulence model. For large eddy simulation, the use of $\sigma_0 > 0$, $\sigma_1 = 0$ may add too much dissipation, so it is tempting to take $\sigma_0$ as small as possible (preferably $\sigma_0 = 0$). In contrast, in the steady case, the use of $\sigma_0 = 0$, $\sigma_1 > 0$ does not provide any stabilization. So for URANS simulation, one may use $\sigma_0 \ge 0$, $\sigma_1 = 0$.  Hybrid RANS/LES models will likely require to make the constants $\sigma_0$ and $\sigma_1$ space-dependent, but we leave this beyond our paper.

\section{Spatial approximations}
\label{sect:spatial}

In this section, we show four examples of discrete gradient/divergence/Laplace tuples satisfying \eqref{eq_gen_assumptions}.

\subsection{High order finite-difference methods}
\label{sect:FD}

We begin with finite-difference methods of order $2m$, $m \in \mathbb{N}$, in the unit square with the periodic boundary conditions. Let the computational mesh be uniform with step $h = 1/N$, $N \in \mathbb{N}$.

Let $V_h$ be the space of $N$-periodic sequences $\{\B{u}_{j,k} \in \mathbb{R}^2\}$ and $Q_h$ be the space of $N$-periodic sequences $p = \{p_{j,k} \in \mathbb{R}\}$ with zero average. They are equipped with the scalar products
$$
(\B{u}, \B{v}) = \sum\limits_{j,k = 0}^{N-1} h^2\, \B{u}_{j,k} \cdot \B{v}_{j,k},
\quad
(p, q) = \sum\limits_{j,k = 0}^{N-1} h^2  p_{j,k}\, q_{j,k}.
$$

Let $a_l/h$ and $b_l/h^2$, $l = -m, \ldots, m$, be the coefficients of the $2m$-th order finite difference approximations of the first and the second derivative, correspondingly. Obviously, $a_{-l} = -a_l$ and $b_{-l} = b_l$. Let $G$, $D$, and $L$ be defined as
$$
(G f)_{jk} = \left(\begin{array}{c} \frac{1}{h}\sum\limits_{l=1}^m a_l (f_{j+l,k} - f_{j-l,k})  \\ \frac{1}{h}\sum\limits_{l=1}^m a_l (f_{j,k+l} - f_{j,k-l}) \end{array}\right),
$$
$$
(D \B{f})_{jk} = \frac{1}{h} \sum\limits_{l=1}^m a_l (\B{f}_{j+l,k}-\B{f}_{j-l,k}) \cdot \left(\begin{array}{c} 1 \\ 0 \end{array}\right)  + \frac{1}{h} \sum\limits_{l=1}^m a_l (\B{f}_{j,k+l}-\B{f}_{j,k-l}) \cdot \left(\begin{array}{c} 0 \\ 1 \end{array}\right),
$$
$$
(L f)_{jk} = \frac{1}{h^2} \sum\limits_{l=1}^m b_l (f_{j+l,k} + f_{j-l,k} + f_{j,k+l} + f_{j,k-l} - 4f_{j,k}).
$$

To prove that operators $G$, $D$, $L$ satisfy \eqref{eq_gen_assumptions}, we need two auxiliary lemmas.

\begin{lemma}
There holds
\begin{equation}
a_k = \frac{(-1)^{k+1}}{k} \frac{(m!)^2}{(m+k)! (m-k)!}, \quad b_k = \frac{2}{k}a_k,\quad k \ne 0,
\label{eq_def_ak}
\end{equation}
$a_0 = 0$, and $b_0$ is defined by $\sum_k b_k = 0$.
\end{lemma}
\begin{proof}
The polynomial interpolating a function $u$ on the mesh with step $h$ is given by the Lagrange formula:
$$
\phi(x) = \sum\limits_{k=-m}^m u_k \prod\limits_{r=-m, r\ne k}^m \frac{x/h-r}{k-r}.
$$
Taking the first and the second derivatives at $x=0$ we get $\phi'(0) = h^{-1} \sum_k a_k u_k$ and $\phi''(0) = h^{-2} \sum_k b_k u_k$ where $a_k$ and $b_k$ are given by \eqref{eq_def_ak}.
\end{proof}

\begin{lemma}\label{th:convex}
Let
\begin{equation}
a(\phi) = \frac{1}{i}\sum\limits_{k=-m}^m a_k \exp(ik\phi), \quad b(\phi) = \sum\limits_{k=-m}^m b_k \exp(ik\phi).
\label{eq_def_aphi}
\end{equation}
Then $a(\phi)>0$, $a''(\phi)>0$, and $b(\phi) + (a(\phi))^2 < 0$ for $\phi \in (0,\pi)$.
\end{lemma}
\begin{proof}
Taking  \eqref{eq_def_ak} to \eqref{eq_def_aphi} we get
$$
a'(\phi) = 1 + \sum\limits_{k=-m}^m (-1)^{k+1} \frac{(m!)^2}{(m+k)! (m-k)!} \exp(ik\phi).
$$
The extra unit term annihilates with the term with $k=0$. Substituting $k = m-l$ we obtain
$$
a'(\phi) = 1 - (-1)^{m}\exp(im\phi)\frac{(m!)^2}{(2m)!}\sum\limits_{l=0}^{2m} \frac{(2m)!}{l! (2m-l)!} (-\exp(-i\phi))^l =
$$
$$
= 1 - (-1)^{m}\exp(im\phi)\frac{(m!)^2}{(2m)!} (1 - \exp(-i\phi))^{2m} = 
1 - \frac{(m!)^2 2^{2m}}{(2m)!} \sin^{2m}(\phi/2).
$$
Therefore,
$$
a''(\phi) = - m \frac{(m!)^2 2^{2m}}{(2m)!} \sin^{2m-1}(\phi/2) \cos(\phi/2) < 0.
$$
Since $a''(\phi) < 0$ and $a(0) = a(\pi) > 0$, then $a(\phi) > 0$ on $(0,\pi)$. Since $a''(\phi) < 0$ and $a'(0) = 1$, then $a'(\phi) < 1$ on $(0, \phi)$.

Since $b_k = 2a_k/k$, then $b'(\phi) = -2a(\phi) < -2a(\phi)a'(\phi)$. By integration, $b(\phi) + (a(\phi))^2 < 0$.
\end{proof}

\begin{proposition}
The operators $G$, $D$, $L$ defined above satisfy \eqref{eq_gen_assumptions}.
\end{proposition}
\begin{proof}
The proof of $D = -G^*$ and $L = L^*$ is straightforward. Therefore, $S = S^*$.
Operators $DG$ and $L$ have the same set of eigenvectors $p^{\alpha,\beta, \pm}$: 
$$
(p^{\alpha,\beta,+})_{jk} = \cos(\alpha j + \beta k), \quad
(p^{\alpha,\beta,-})_{jk} = \sin(\alpha j + \beta k), 
$$
where $\alpha, \beta = 0, 2\pi/N, \ldots, 2\pi(N-1)/N$, $\alpha+\beta \ne 0$. Let
$a(\phi)$ and $b(\phi)$ be defined by \eqref{eq_def_aphi}.
Note that $a(\phi), b(\phi) \in \mathbb{R}$ for $\phi \in \mathbb{R}$ and $a(0)=b(0) = 0$. By construction, 
$$
-DGp^{\alpha,\beta,\pm}=((a(\alpha))^2+(a(\beta))^2) p^{\alpha,\beta,\pm}, \quad Lp^{\alpha,\beta,\pm} = (b(\alpha)+b(\beta)) p^{\alpha,\beta,\pm}.
$$
By Lemma~\ref{th:convex}, $(a(\phi))^2 + b(\phi) < 0$ for $\phi \in (0,2\pi)$. Since $\alpha$ and $\beta$ cannot be equal to zero simultaneously, then
$$
(a(\alpha))^2 + b(\alpha) + (a(\beta))^2 + b(\beta) < 0.
$$
Hence, $S = DG-L > 0$.
\end{proof}

{\it Remark.} The use of a high-order gradient discretization as $G$ and a lower-order Laplace discretization as $L$ will likely result in the violation of $S \ge 0$.


\subsection{Finite differences, Dirichlet conditions}
\label{sect:FD2}

The basic 1-exact finite-difference method is applicable for non-uniform meshes and in the case of the Dirichlet boundary conditions. Let $\Omega$ be the unit square and the computational mesh be the Cartesian product of two one-dimensional meshes: 
$$
x_0 = 0 < x_1 < \ldots < x_{N_x} = 1, \quad 
y_0 = 0 < y_1 < \ldots < y_{N_y} = 1.
$$
The 1D cell volumes in $x$ corresponding to mesh nodes are $(\Delta x)_0 = (x_1 - x_0)/2$, $(\Delta x)_j = (x_{j+1}-x_{j-1})/2$, $j = 1, \ldots, N_x-1$, and $(\Delta x)_{N_x} = (x_{N_x} - x_{N_x-1})/2$. The 1D cell volumes in $y$ are defined the same way.

Let $\bar{\mathcal{N}}$ be the set of mesh nodes, $\mathcal{N}$ be the set of internal mesh nodes (i.~e. belonging to $(0,1)^2$), $\mathcal{N}_b$ be the set of boundary nodes (i. e. nodes belonging to $\D \Omega$), and $\mathcal{N}_c$ be the set of 4 corner nodes. Each node is identified by a pair of indices: $\langle j,k \rangle$, $j = 0, \ldots, N_x$, $k = 0, \ldots, N_y$. On $\mathbb{R}^{\bar{\mathcal{N}}}$, define the scalar product
$$
[f, g] = \sum\limits_{j=0}^{N_x}\sum\limits_{k=0}^{N_y} (\Delta x)_j (\Delta y)_k\  f_{j,k}\, g_{j,k}.
$$

Let
$$
V_h = \{\B{f} \in (\mathbb{R}^{\bar{\mathcal{N}}})^2\ :\  \B{f}_{\langle j,k \rangle} = 0\ \mathrm{for}\ \langle j,k \rangle \in \mathcal{N}_b\},
$$
$$
Q_h = \{f \in \mathbb{R}^{\bar{\mathcal{N}}}\ :\  [f,1]=0, \ \ f_{{\langle j,k \rangle}} = 0 \ \mathrm{for}\ {\langle j,k \rangle} \in \mathcal{N}_c\}.
$$
The scalar products on $V_h$ and $Q_h$ are inherited from $\mathbb{R}^{\bar{\mathcal{N}}}$.

An element of $V_h$ has zero values at boundary nodes. So if $u(t) \in V_h$, then the actual velocity field $u_{actual}(t) \in (\mathbb{R}^{\bar{\mathcal{N}}})^2$ is $u_{actual}(t) = u(t) + u_b(t)$ where $u_b(t) \in (\mathbb{R}^{\bar{\mathcal{N}}})^2$ depends on $\tilde{\B{u}}_b(t)$ only. For now, define $u_b(t)$ by
\begin{equation}
(u_b(t))_{\langle j,k \rangle} = (\Pi \tilde{\B{u}}_b(t))_{\langle j,k \rangle} := \left\{\begin{array}{ll} \tilde{\B{u}}_b(t, x_j, y_k), & {\langle j,k \rangle} \in \mathcal{N}_b; \\ 0 & {\langle j,k \rangle} \in \mathcal{N}.
\end{array}\right.
\label{eq_def_ub}
\end{equation}


Let $G$ be the operator taking  $f \in Q_h$ to $G f \in V_h$ as
\begin{equation}
\begin{gathered}
(G f)_{j,k} = \frac{1}{2}\left(\begin{array}{c} (f_{j+1,k} - f_{j-1,k})/(\Delta x)_j \\ (f_{j,k+1} - f_{j,k-1})/(\Delta y)_k \end{array}\right), \quad \langle j,k \rangle \in \mathcal{N};
\\
(G f)_{j,k} = 0, \quad \langle j,k \rangle \in \mathcal{N}_b.
\label{eq_FD2_G}
\end{gathered}
\end{equation}
Thus, if $f$ is the set of the point values of a smooth function $\tilde{f}$, then $Gf$ approximates the gradient of $\tilde{f}$ at internal nodes and is equal to zero at boundary nodes. The latter is obligatory because $Gf$ should belong to $V_h$.

Now we approximate the divergence of a vector field. For $\B{F} = (f, 0)^T$, $f \in \mathbb{R}^{\bar{\mathcal{N}}}$, and each $k = 0, \ldots, N_y$ we put 
$$
(Div \B{F})_{0,k} = \frac{f_{1,k} - f_{0,k}}{x_1 - x_0}, \quad
(Div \B{F})_{N_x,k} = \frac{f_{N_x,k} - f_{N_x-1,k}}{x_{N_x} - x_{N_x-1}},
$$
$$
(Div \B{F})_{j,k} = \frac{f_{j+1,k} - f_{j-1,k}}{x_{j+1} - x_{j-1}}, \quad j = 1, \ldots, N_x-1.
$$
Similarly, for $\B{F} = (0, g)^T$, $g \in \mathbb{R}^{\bar{\mathcal{N}}}$, and each $j = 0, \ldots, N_x$ we put 
$$
(Div \B{F})_{j,0} = \frac{g_{j,1} - g_{j,0}}{y_1 - y_0}, \quad
(Div \B{F})_{j,N_y} = \frac{g_{j,N_y} - g_{j,N_y-1}}{y_{N_y} - y_{N_y-1}},
$$
$$
(Div \B{F})_{j,k} = \frac{g_{j,k+1} - g_{j,k-1}}{y_{k+1} - y_{k-1}}, \quad k = 1, \ldots, N_x-1.
$$
For a general $\B{F} \in (\mathbb{R}^{\bar{\mathcal{N}}})^2$, define $Div \B{F}$ by linearity.

Let $D$ be the operator taking $\B{f} \in V_h$ to $D\B{f} = Div\B{f}$ and put $H(t) = -Div(\Pi \tilde{\B{u}}_b(t))$. The components of $D \B{f}$ may be written in the concise form
\begin{equation}
(D \B{f})_{j,k} = \frac{\B{f}_{j+1,k} - \B{f}_{j-1,k}}{2(\Delta x)_j} \cdot \B{e}_x + \frac{\B{f}_{j,k+1} - \B{f}_{j,k-1}}{2(\Delta y)_k} \cdot \B{e}_y,
\label{eq_FD2_D}
\end{equation}
where values of $\B{f}$ with out-of-range indices are assumed zero. It is easy to check that $D\B{f} \in Q_h$. We do not write an explicit expression for $H(t)$ because it is cumbersome and, furthermore, for a program implementation one needs the operator $Div$ and not $D$ and $H$ independently.

To this moment, we did not control that $H(t)$ belongs to $Q_h$, i. e. \mbox{$\int \tilde{\B{u}}_b \cdot \B{n} ds = 0$} is satisfied in the discrete sense. The easiest way to enforce this is to correct the mapping $\Pi$ in \eqref{eq_def_ub} so that $Div\, u_b(t) \equiv 0$.

Let $L$ be the operator taking $f \in Q_h$ to $Lf \in Q_h$ as
\begin{equation}
\begin{gathered}
(Lf)_{j,k} = \varphi_{\langle j,k \rangle, \langle j+1,k \rangle} (\Delta x)_j^{-1}\frac{f_{j+1,k} - f_{j,k}}{x_{j+1}-x_j} - \varphi_{\langle j,k \rangle, \langle j-1,k \rangle} (\Delta x)_j^{-1}\frac{f_{j,k} - f_{j-1,k}}{x_j-x_{j-1}} +
\\
+ \varphi_{\langle j,k \rangle, \langle j,k+1 \rangle} (\Delta y)_k^{-1}  \frac{f_{j,k+1} - f_{j,k}}{y_{j+1}-y_j} - \varphi_{\langle j,k \rangle, \langle j,k-1 \rangle} (\Delta y)_k^{-1} \frac{f_{j,k} - f_{j,k-1}}{y_{j}-y_{j-1}}
\end{gathered}
\label{eq_FD2_L1}
\end{equation}
where 
\begin{equation}
\varphi_{\langle j,k \rangle, \langle j',k' \rangle} = \left\{\begin{array}{ll}
1, & \langle j,k \rangle, \langle j',k' \rangle \in \mathcal{N};  \\
1/2, & \langle j,k \rangle \in \mathcal{N}, \langle j',k' \rangle \in \mathcal{N}_b;  \\
1/2, & \langle j,k \rangle \in \mathcal{N}_b, \langle j',k' \rangle \in \mathcal{N};  \\
0, & \langle j,k \rangle, \langle j',k' \rangle \in \mathcal{N}_b;
\\
0, & \langle j,k \rangle \not\in \bar{\mathcal{N}} \ \mathrm{or}\  \langle j',k' \rangle \not\in \bar{\mathcal{N}}.
\end{array}\right.
\label{eq_FD2_L2}
\end{equation}

\begin{lemma}
For $f \in Q_h$, $\B{g} \in V_h$, there holds $[f, D\B{g}] = -[\B{g}, Gf]$.
\label{th:lemma3}
\end{lemma}
\begin{proof}
First assume that $\B{g} = (g, 0)^T$. By definition,
\begin{equation*}
\begin{gathered}
  \left[f, D\B{g}\right] = 
\frac{1}{2}\sum\limits_{k=0}^{N_y}(\Delta y)_k \Bigl(g_{1,k} f_{0,k} + (g_{2,k} - g_{0,k}) f_{1,k} + \ldots
\\
 + (g_{N_x,k} - g_{N_x-2,k}) f_{N_x-1,k} - g_{N_x-1,k} f_{N_x,k}\Bigr).
\end{gathered}
\end{equation*}
Note that $g_{0,k} = g_{N_x,k}=0$ because $\B{g} \in V_h$. By reordering,
$$
\left[f, D\B{g}\right] =
-\frac{1}{2}\sum\limits_{k=0}^{N_y}(\Delta y)_k \Bigl(g_{1,k} (f_{2,k} - f_{0,k}) + \ldots 
 + g_{N_x-1,k} (f_{N_x,k}-f_{N_x-2,k})\Bigr) = [-Gf, \B{g}].
$$
The proof for $\B{g} = (0,g)^T$ is similar, and the general case is by linearity.
\end{proof}

\begin{proposition}
Operators $G$, $D$, $L$ defined by \eqref{eq_FD2_G}--\eqref{eq_FD2_L2} satisfy \eqref{eq_gen_assumptions}. Furthermore, $S f = 0$ holds if and only if $f$ is a linear function, i. e. $f_{j,k} = \alpha x_j + \beta y_k$ for some $\alpha, \beta \in \mathbb{R}$.
\end{proposition}
\begin{proof}
The identity $D = -G^*$ is by Lemma~\ref{th:lemma3}. For each $f, g \in Q_h$ there holds
\begin{equation}
\begin{gathered}
\ [Lf, g] = -\sum\limits_{k=1}^{N_y-1}  (\Delta y)_k  \sum\limits_{j=0}^{N_x-1} \varphi_{\langle j,k \rangle, \langle j+1,k \rangle} \frac{(f_{j+1,k} - f_{j,k})(g_{j+1,k} - g_{j,k})}{x_{j+1}-x_j}
\\
-\sum\limits_{j=1}^{N_x-1}  (\Delta x)_j  \sum\limits_{k=0}^{N_y-1} \varphi_{\langle j,k \rangle, \langle j,k+1 \rangle} \frac{(f_{j,k+1} - f_{j,k})(g_{j,k+1} - g_{j,k})}{y_{k+1}-y_k}.
\end{gathered}
\label{eq_aux_FD2_L}
\end{equation}
Thus, $L = L^* \le 0$. Furthermore, let $f \in Q_h$ and $[Lf, f] = 0$. Then for each pair of neighboring nodes with $\phi_{\langle j,k\rangle, \langle j',k'\rangle} \ne 0$ there holds $f_{j,k} = f_{j',k'}$. By the choice of $Q_h$, we have $f = 0$. This proves that $L < 0$.

Since $D = -G^*$, we have
$$
-[f, DGf] = \sum\limits_{\langle j, k \rangle \in \mathcal{N}} (\Delta x)_j (\Delta y)_k |(Gf)_{\langle j, k \rangle}|^2 = \Sigma_x + \Sigma_y
$$
where
$$
\Sigma_x = \frac{1}{4} \sum\limits_{k=1}^{N_y-1}  (\Delta y)_k  \sum\limits_{j=1}^{N_x-1}  (\Delta x)_j^{-1}(f_{j+1,k} - f_{j-1,k})^2,
$$
$$
\Sigma_y =\frac{1}{4}\sum\limits_{j=1}^{N_x-1}  (\Delta x)_j \sum\limits_{k=1}^{N_y-1}  (\Delta y)_k^{-1}  (f_{j,k+1}-f_{j,k-1})^2.
$$
After a simple transformation,
$$
\Sigma_x = \frac{1}{2} \sum\limits_{k=1}^{N_y-1}  (\Delta y)_k  \sum\limits_{j=1}^{N_x-1} \left[\frac{(f_{j+1,k}-f_{j,k})^2}{x_{j+1}-x_j} +   \frac{(f_{j,k}-f_{j-1,k})^2}{x_{j}-x_{j-1}}\right] - 
$$
$$
 - \frac{1}{2} \sum\limits_{k=1}^{N_y-1}  (\Delta y)_k  \sum\limits_{j=1}^{N_x-1}\frac{(x_{j+1}-x_j)(x_j - x_{j-1})}{x_{j+1}-x_{j-1}} \left(\frac{f_{j+1,k}-f_{j,k}}{x_{j+1}-x_j} - \frac{f_{j,k}-f_{j-1,k}}{x_j - x_{j-1}}\right)^2.
$$
By \eqref{eq_aux_FD2_L}, $[Lf,f] = \tilde{\Sigma}_x + \tilde{\Sigma}_y$ where
$$
\tilde{\Sigma}_x = -\sum\limits_{k=1}^{N_y-1}  (\Delta y)_k  \sum\limits_{j=0}^{N_x-1} \varphi_{\langle j,k \rangle, \langle j+1,k \rangle} \frac{(f_{j+1,k} - f_{j,k})^2}{x_{j+1}-x_j},
$$
$$
\tilde{\Sigma}_y = -\sum\limits_{j=1}^{N_x-1}  (\Delta x)_j  \sum\limits_{k=0}^{N_y-1} \varphi_{\langle j,k \rangle, \langle j,k+1 \rangle} \frac{(f_{j,k+1} - f_{j,k})^2}{y_{k+1}-y_k}.
$$
We see that $-\tilde{\Sigma}_x$ coincides with the first term on the right-hand side of the expression for $\Sigma_x$. Thus, $f^T S f = -(\Sigma_x + \Sigma_y + \tilde{\Sigma}_x + \tilde{\Sigma}_y)\ge 0$ and $S$ vanishes on linear functions.
\end{proof}

{\it Remark.} Taking the standard approximation of the Laplace operator ($\phi_{\langle j,k\rangle, \langle j',k'\rangle} = 1$ for $\langle j,k\rangle, \langle j',k'\rangle \in \bar{\mathcal{N}}$) would preserve  \eqref{eq_gen_assumptions} but the stabilization term would no longer be zero on linear functions, causing a large solution error near boundaries.

{\it Remark.} Due to the constraint $[f, 1]=0$ in the definition of $Q_h$, a direct matrix representation of operator $L$ with a local portrait is not possible. Instead, to efficiently solve a system of the form $Lx = y$, $x, y \in Q_h$, the following well-known procedure may be used:
\begin{enumerate}
\item Let $L' : \mathbb{R}^{\bar{\mathcal{N}}} \to \mathbb{R}^{\bar{\mathcal{N}}}$ be defined by \eqref{eq_FD2_L1}--\eqref{eq_FD2_L2}. Let $\mathrm{L}'$ be its matrix representation in the standard basis (which is defined by $(\phi_{\langle j,k \rangle})_{\langle l,m \rangle} = \delta_{jl}\delta_{km}$, $\langle j,k \rangle, \langle l,m \rangle \in \bar{\mathcal{N}}$). Consider the system $\mathrm{L}' x' = y$ for $x' \in \mathbb{R}^{\bar{\mathcal{N}}}$.
\item Pin the values at corner nodes: for each ${\langle j,k \rangle} \in \mathcal{N}_c$, replace the line corresponding to node $j$ by $x'_{\langle j,k \rangle} = 0$.
\item Pin the value at one non-corner node: replace the line corresponding to one (arbitrary chosen) node ${\langle j,k \rangle} \in \bar{\mathcal{N}} \setminus \mathcal{N}_c$ by $x'_{\langle j,k \rangle} = 0$.
\item The resulting system has a non-degenerate matrix. Solve it for $x' \in \mathbb{R}^{\bar{\mathcal{N}}}$ using an algebraic multigrid method. By construction, $x'_{\langle j,k \rangle} = 0$ for ${\langle j,k \rangle} \in \mathcal{N}_c$.
\item Put $x_{\langle j,k \rangle} = x'_{\langle j,k \rangle} - [x', 1]$ for each ${\langle j,k \rangle} \in \bar{\mathcal{N}}$. Now $x \in Q_h$.
\end{enumerate}

\subsection{Galerkin methods}
\label{sect:Gal}

We begin with the case of periodic boundary conditions. Let $\Omega$ be the unit cube. Let $L^2_{per}(\mathbb{R}^d)$ and $H^1_{per}(\mathbb{R}^d)$ be the spaces of 1-periodic functions from $L^2_{loc}(\mathbb{R}^d)$ and $H^1_{loc}(\mathbb{R}^d)$, correspondingly. Let $V_h$ be a finite-dimensional subspace of $(H^1_{per}(\mathbb{R}^d))^d$ equipped with the scalar product $(\B{f},\B{g}) = \int_{\Omega} \B{f}\cdot\B{g}dr$. Let $Q$ be the space of functions from $H^1_{per}(\mathbb{R}^d)$ with zero average. Let $Q_h$ be a finite-dimensional subspace of $Q$ equipped with the scalar product from $L^2(\Omega)$. Let $Gf$, $f \in Q_h$, be defined as the orthogonal projection of $\nabla f$ on $V_h$:
\begin{equation}
(G f, \B{g}) = \int\limits_{\Omega} \B{g}(\B{r}) \cdot \nabla f(\B{r}) dV, \quad f \in Q_h, \quad \B{g} \in V_h.
\label{eq_def_FE_G}
\end{equation}
Then $D = -G^*$ is given by
\begin{equation}
(D \B{f}, g) = \int\limits_{\Omega} g(\B{r}) \nabla \cdot \B{f}(\B{r}) dV, \quad \B{f} \in V_h, \quad g \in Q_h.
\label{eq_def_FE_D}
\end{equation}
Let $L$ be defined by
\begin{equation}
(L f, g) = -\int\limits_{\Omega} \nabla {f}(\B{r}) \cdot \nabla {g}(\B{r}) dV, \quad f, g \in Q_h.
\label{eq_def_FE_L}
\end{equation}

{\it Remark.} In bases $\{\phi_j \in V_h\}$ and $\{\psi_j \in Q_h\}$, the operators $G$, $D$, and $L$ take the form $M_{\psi}^{-1} \tilde{G}$, $M_{\phi}^{-1} \tilde{G}^T$, $M_{\phi}^{-1} \tilde{L}$ where $M_{\psi}$ and $M_{\phi}$ are the mass matrices, $\tilde{L}$ is the stiffness matrix, and $\tilde{G} = \{\int \phi_j \nabla \psi_k dV\}$. For $f = \sum f_j \phi_j$ and $g = \sum g_j \phi_j$ the scalar product is $(f, g) = (f_1, \ldots, f_{dim \,Q_h})^T M_{\phi} (g_1, \ldots, g_{dim\, Q_h})$.

\begin{proposition}
The operators $G$, $D$, $L$ defined above satisfy \eqref{eq_gen_assumptions}.
\end{proposition}
The proof may be found in \cite{Codina1998} (Sect. 4.1). We give its short form.

\begin{proof}
The conditions $D = -G^*$ and $L = L^* < 0$ are by construction. Let $\|\ \cdot\ \|$ be the square norm of the vector with the components from $L^2(\Omega)$. Then $(-Lf, f) = \|\nabla f\|^2$. We also have $(DGf, f) = -(Gf, Gf) = -\|Gf\|^2$.
Since $Gf$ is the orthogonal projection of $\nabla f$ on $V_h$, then $\|Gf\| \le \|\nabla f\|$, and $(S f, f) = \|\nabla f\|^2 - \|G f\|^2 \ge 0$.
\end{proof}

We utilized the Galerkin method to construct operators $G$, $D$, $L$ for use in \eqref{eq_momentun_base}, \eqref{eq_continuity_base}. What we obtained, can be transformed back to a weak formulation, which is conventional for finite-element methods. For the steady Stokes problem, this is done in \cite{Codina1997}, and accuracy estimates were proved.

{\it Remark.} For methods that do not satisfy the inf-sup condition, pressure stabilization is a stability requirement. If the inf-sup condition holds, instead of \eqref{eq_def_FE_L} one can put $L = DG$ and thus $S = 0$. However, for finite element methods, the matrix representation of $L = DG$ is nonlocal (unless the mass lumping is used), and one finds themselves in a position so solve a Darcy-type problem instead of a scalar elliptic equation (see Remark~3.1 in \cite{Colomes2015}). If one wants to have the stiffness matrix in the pressure equation, they naturally get $S \ne 0$ even for an inf-sup-stable method.\\

Now we move to the case that $\Omega$ is a bounded domain in $\mathbb{R}^d$, and the homogeneous Dirichet boundary conditions are applied on $\D \Omega$. Here the choice of finite element spaces is controversial. It seems natural to take $V_h$ as a subspace of $(H^1_0(\Omega))^d$. However, this has a major drawback. On the boundary, $\nabla p$ is generally nonzero while \mbox{$Gp \in V_h$} is always zero. Therefore, one can expect that $(S p, p) = \|\nabla p\|^2 - \|Gp\|^2$ is bounded from below by $O(h)$. For the Stokes problem, this restricts the  convergence rate in $L_2$ by $O(h^{3/2})$ for the velocity and $O(h^{1/2})$ for the pressure gradient (see Theorem~3 in \cite{Codina1997}). 

One may ignore the Dirichlet boundary conditions when choosing $V_h$ and enforce them by a proper discretization of the flux terms. However, this setup is counterintuitive.  First, if $G$ is given by \eqref{eq_def_FE_G}, then $D = -G^*$ is defined by
$$
(D \B{f}, g) = \int\limits_{\Omega} g(\B{r}) \nabla \cdot \B{f}(\B{r}) dV - \int\limits_{\D \Omega} g(\B{r}) \B{f}(\B{r}) \cdot \B{n} dS, \quad \B{f} \in V_h, \quad g \in Q_h.
$$
Second, the simplified system $\nabla \cdot \tilde{\B{u}}=0$, $\tilde{\B{u}}_t + \nabla \tilde{p} = 0$  requires a slip boundary condition, but there are no flux terms where these boundary conditions may be imposed. Therefore, it is hardly possible to construct a robust numerical method, especially for convection-dominant flows.


An intermediate approach is to take $V_h$ as a subspace of $(H^1(\Omega))^d$ such that for each $\B{f} \in V_h$ there holds
$$
\int\limits_{\D \Omega} \B{f}(\B{r}) \cdot \B{n}\ g(\B{r}) dS = 0, \quad g \in Q_h.
$$
Then $\tilde{\B{u}} \cdot \B{n} = 0$ is imposed by the choice of the finite element space, while \mbox{$\tilde{\B{u}} \cdot \B{n}^{\perp} = 0$} should be imposed by the discretization of the viscous term.  If \mbox{$\D p/\D n \ne 0$} on $\D \Omega$, then we get a numerical boundary layer and a low order convergence. However, if $\D p/\D n$ is small, then this error may be not so crucial. We use this approach in Section~\ref{sect:turbchannel} for the turbulent flow in a channel.

{\it Remark.} It is well-known that projection-based methods (except explicit ones) in the standard formulation generate a numerical boundary layer, see \cite{Guermond2006}. Here we have a numerical boundary layer of a different origin. It results from the difference between $-G^*G$ and $L$, and does not vanish for explicit methods.

\subsection{Mass-lumped P1-Galerkin method}

As we saw in Section~\ref{sect:FD2}, the basic 1-exact central difference scheme naturally admits the strong enforcement of the Dirichlet condition. Here we show that the same holds for the mass-lumped P1-Galerkin method, which may be considered as a generalization of the central difference scheme to unstructured meshes.

If the basis functions are nonnegative, then the mass-lumped Galerkin method gives operators that satisfy \eqref{eq_gen_assumptions}. For a general Galerkin method, mass lumping results in the loss of high order convergence. High-order mass-lumped methods are designed in \cite{Mulder2001, Cohen2001} (see also references therein) but the basis functions are not nonnegative. Thus, although we do not rely on a specific form of basis functions, what we discuss here is mainly designed for the mass-lumped P1-Galerkin method.

Let $\mathcal{N}$ and $\mathcal{N}_0 \subset \mathcal{N}$ be finite non-empty sets. Let $\{\phi_j \in H^1(\Omega), j \in \mathcal{N}\}$ be a linearly independent system of functions used to approximate one velocity component. Assume that $v_j := \int \phi_j dV > 0$ for each $j \in \mathcal{N}$, and $\phi_j(\B{r})=0$ on $\D \Omega$ for $j \in \mathcal{N} \setminus \mathcal{N}_0$. For 
$$
f = \sum\limits_{k \in \mathcal{N}} f_k \phi_k, \quad g = \sum\limits_{k \in \mathcal{N}} g_k \phi_k
$$
denote
$$
[f,g] = \sum\limits_{j \in \mathcal{N}} v_j f_j g_j.
$$

Define $\hat{V}_h = (span\{\phi_j, j \in \mathcal{N}\})^d$ and equip it with the scalar product $[\B{f}, \B{g}] = [f^x, g^x] + [f^y, g^y]$ (for $d=2$) or similar for other dimensions. Let $V_h \subset \hat{V}_h$ be the set of functions with zero coefficients for $j \in \mathcal{N} \setminus \mathcal{N}_0$. Let $Q_h$ be a finite-dimensional subspace of $Q$ where $Q$ is the space of functions from $H^1(\Omega)$ with zero average.  The choice of the scalar product in $Q_h$ does not matter because we only need the combination $L^{-1}D$.


For $f \in Q_h$, define $\hat{G} f \in V_h$ by
\begin{equation}
[\hat{G} f, \B{g}] = \int\limits_{\Omega} \B{g}(\B{r}) \cdot \nabla f(\B{r}) dV,  \quad \B{g} \in \hat{V}_h.
\label{eq_def_FE_G_lumped}
\end{equation}
Or, equivalently,
\begin{equation}
(\hat{G} f)(\B{r}) = \sum\limits_{j \in \mathcal{N}} (\hat{G} f)_j \phi_j(\B{r}), 
\quad 
(\hat{G} f)_j = v_j^{-1} \int\limits_{\Omega} \phi_j(\B{r}) \nabla f(\B{r}) dV, \quad f \in Q_h.
\label{eq_def_FE_G_lumped2}
\end{equation}
Then $Gf$ can be obtained from $\hat{G}f$ by nullifying the coefficients at $\phi_j$, $j \in \mathcal{N} \setminus \mathcal{N}_0$, and $D = -G^*$ is defined by \eqref{eq_def_FE_D}. Finally, put
\begin{equation}
(L f, g) = -\int\limits_{\Omega} \nabla f \cdot \nabla g  dV + \sum\limits_{j \in \mathcal{N} \setminus \mathcal{N}_0} v_j (\hat{G}f)_j \cdot (\hat{G}g)_j.
\label{eq_def_FE_L_lumped}
\end{equation}

\begin{proposition}
Assume that $\phi_j(\B{r}) \ge 0$ and $\sum_k \phi_k(\B{r}) \le 1$ for each $j \in \mathcal{N}$ and $\B{r} \in \Omega$. Then operators $G$, $D$, $L$ defined above satisfy \eqref{eq_gen_assumptions}.
\end{proposition}

\begin{proof}
First we claim that for each $w \in \hat{V}_h$ there holds $\int |w|^2 dV \le [w,w]$. Indeed,
\begin{equation*}
\begin{gathered}
\int |w|^2 dV = \int \sum\limits_{j,k} w_j \cdot w_k \phi_j \phi_k dV \le 
\int \sum\limits_{j,k} \frac{|w_j|^2 + |w_k|^2}{2} \phi_j \phi_k dV 
= \\
=
\int \sum\limits_{j} |w_j|^2 \phi_j \sum\limits_{k} \phi_k dV 
\le 
\sum\limits_{j} |w_j|^2 \int \phi_j dV = [w,w].
\end{gathered}
\end{equation*}
Using this fact for $w = \hat{G} f$, $f \in Q_h$, and by the Schwarz inequality we obtain
\begin{equation}
\begin{gathered}
\left[\hat{G} f, \hat{G}f\right] = \int (\hat{G} f) \cdot \nabla f dV \le
\left(\int |\hat{G} f|^2 dV\right)^{1/2}
\left(\int |\nabla f|^2 dV\right)^{1/2} 
\le 
\\
\le
\left([\hat{G} f, \hat{G} f]\right)^{1/2}
\left(\int |\nabla f|^2 dV\right)^{1/2}.
\end{gathered}
\label{eq_Schrartz}
\end{equation}
Thus, $[\hat{G} f, \hat{G}f] \le \int |\nabla f|^2 dV$. The condition $D = -G^*$ is by construction. Therefore,
$$
(S f, f) = - [G f, G f] - (L f, f) = \int |\nabla f|^2 dV - [\hat{G} f, \hat{G}f] \ge 0.
$$
Thus, $S \ge 0$ and $L \le 0$. 

It remains to prove that $(Lf, f) < 0$ for each $f \in Q_h$, $f \ne 0$. If $(S f, f) > 0$, then $(Lf, f) \le -(S f, f) < 0$. If $(S f, f)=0$, then both inequalities in \eqref{eq_Schrartz} turn to equalities. The first one does if $(\hat{G}f)_j=(\hat{G}f)_k$ for each $j, k \in \mathcal{N}$. Therefore, $f$ is a sum of a linear function and a vector from $\mathrm{Ker}\,\hat{G}$. The second one does if $\hat{G}f = c\nabla f$, hence, $f$ can be either a linear function ($c=1$) or a vector from $\mathrm{Ker}\,\hat{G}$ ($c=0$). In the first case, we have $(Lf, f) = - |\nabla f|^2 \sum v_j$ where the sum is over $j \in \mathcal{N}_0$. In the second case, we have $(Lf, f) = -\int |\nabla f|^2 dV < 0$. Thus, $L < 0$.
\end{proof}

\section{IMEX Runge~-- Kutta methods}
\label{sect:IMEX}



Consider an ODE system of the form
\begin{equation}
\frac{dy}{dt} = \xi(t,y) + \eta(t,y).
\label{eq_dydt}
\end{equation}
In this section, we describe implicit-explicit (IMEX) Runge~-- Kutta methods for \eqref{eq_dydt}. Here $\xi$ is to be taken explicitly, and $\eta$ is to be taken implicitly. We shall use superscript for the timestep index and subscript for the stage index; $t^n$ always stand for time moments and not for powers of $t$. Let $y^{n-1}$ be the numerical solution of \eqref{eq_dydt} at $t^{n-1}$, and $\tau = t^{n} - t^{n-1}$. 

There are two groups of IMEX methods the literature, the difference between them being not in the coefficients but in the form of the methods themselves. We primarily use the form from \cite{Ascher1997, Kennedy2003, Boscarino2007a, Boscarino2009} because it better suits our purpose. The other form, which is suggested in \cite{Boscarino2021, Boscarino2024}, will be considered in Section~\ref{sect:IMEXv2}. We intentionally separate them to make the text easier to read.

\subsection{General form of IMEX RK methods}

Following \cite{Ascher1997, Kennedy2003, Boscarino2007a, Boscarino2009}, we consider implicit-explicit (IMEX) Runge~-- Kutta methods for \eqref{eq_dydt} of the following form. Let $s$ be the number of stages. For each stage $j = 1, \ldots, s$, the intermediate value $y_j$ is defined as the solution of
\begin{equation}
y_j = y_{j,*} + \tau a_{jj} K_j, \quad K_j = \eta(t_j,y_j),
\label{eq_IMEX_1}
\end{equation}
with $t_j = t^{n-1} + c_j \tau$ and
\begin{equation}
y_{j,*} = y^{n-1} + \tau \sum\limits_{k=1}^{j-1} a_{jk} K_k + \tau \sum\limits_{k=1}^{j-1} \hat{a}_{jk} \hat{K}_k,
\label{eq_IMEX_1a}
\end{equation}
and an explicit part follows:
\begin{equation}
\hat{K}_{j} = \xi(t_j,y_j).
\label{eq_IMEX_2}
\end{equation}
Note that in \eqref{eq_IMEX_1}, $a_{jj}$ may be zero for $j=1$. Finally, the value at $t^{n}$ is defined as
\begin{equation}
y^n = y^{n-1} + \tau \sum\limits_{k=1}^s b_k K_k + \tau \sum\limits_{k=1}^{s} \hat{b}_{k} \hat{K}_k.
\label{eq_IMEX_3}
\end{equation}

\subsection{Some notation}

IMEX RK methods from both groups are characterized by the coefficient matrices $A = (a_{jk})$, $\hat{A} = (\hat{a}_{jk})$ and vectors $b = (b_1, \ldots, b_s)^T$, $\hat{b} = (\hat{b}_1, \ldots, \hat{b}_s)^T$. They can be represented by a double tableau in the usual Butcher notation:
$$
\begin{array}{c|c} c & A \\ \hline & b^T \end{array} \quad
\begin{array}{c|c} \hat{c} & \hat{A} \\ \hline & \hat{b}^T \end{array},
$$
and $c$ and $\hat{c}$ are defined as 
$$
c_j = \sum\limits_{k=1}^j a_{jk}, \quad \hat{c}_j = \sum\limits_{k=1}^{j-1} \hat{a}_{jk}.
$$

For $\xi = 0$, an IMEX RK scheme reduces to a diagonally implicit (DIRK) scheme, and for $\eta = 0$, it reduces to an explicit scheme. An IMEX RK method is called {\it stiffly accurate} if its implicit counterpart is stiffly accurate, i. e. $a_{sk} = b_k$ for each $k = 1, \ldots, s$. This property is crucial for integrating stiff ODEs.

IMEX RK schemes are usually designated by triplet $(\sigma_{im}, \sigma_{ex}, p)$ where $\sigma_{im}$ is the number of implicit stages, $\sigma_{ex}$ is the number of explicit stages, and $p$ is the combined order of the method. A simple example of an IMEX RK scheme is the forward-backward Euler method $(1,2,1)$:
\begin{equation}
\begin{gathered}
y_1 = y^{n-1} + \tau \xi(t^{n-1}, y^{n-1}) + \tau \eta(t^{n}, y_1),
\\
y^n = y^{n-1} + \tau \xi(t^{n}, y_1) + \tau \eta(t^{n}, y_1) = y_1 + \tau(\xi(t^{n}, y_1) - \xi(t^{n-1}, y^{n-1})).
\end{gathered}
\label{eq_Butcher_Euler}
\end{equation}
It is of the form \eqref{eq_IMEX_1}--\eqref{eq_IMEX_3} with the matrix representation
$$
\begin{array}{c|cc} 0 & 0 & 0 \\ 1 & 0 & 1 \\ \hline & 0 & 1 \end{array}, \quad
\begin{array}{c|cc} 0 & 0 & 0 \\ 1 & 1 & 0 \\ \hline & 0 & 1 \end{array}.
$$

Following \cite{Boscarino2007}, we say that IMEX methods present in the literature can be classified in three different types. An IMEX RK method is of {\it type A} if the matrix $A$ is invertible. An IMEX RK method if of {\it type CK} if the matrix $A$ has the form
$$
A = \left(\begin{array}{cc} 0 & 0 \\ a & \hat{A}\end{array}\right),
$$
where $\hat{A}$ is an invertible matrix of size $s-1$. An IMEX RK method is of {\it type ARS} \cite{Ascher1997} if it is of type CK with $a = 0$. Or, in terms of coefficients, a method is of type ARS iff $a_{j1} = 0$ for each $j = 1, \ldots, s$. The forward-backward Euler method is of type ARS and denoted below as ARS(1,2,1).

If an IMEX method is of type A, then we assume that its implicit counterpart is an SDIRK method, i. e. the diagonal coefficients satisfy $a_{11} = \ldots = a_{ss}$. If it is of type CK, then we assume that its implicit counterpart is an ESDIRK method, i. e. $a_{11}=0$, $a_{22} = \ldots = a_{ss}$. This assumption holds for most of DIRK methods used in practice, but some methods \cite{Caflisch1997, Liotta2000} violate it. Denote
$$
\breve{\tau} = a_{ss} \tau. 
$$

Among the family \eqref{eq_IMEX_1}--\eqref{eq_IMEX_3}, we primarily consider methods of type CK satisfying $c = \hat{c}$. The equality $c = \hat{c}$ is impossible for methods of type A because $c_1 = a_{ss} \ne 0$ and $\hat{c}_1 = 0$. This is a reason for the poor behavior of methods \eqref{eq_IMEX_1}--\eqref{eq_IMEX_3} of type A: first order convergence for pressure was reported in \cite{Boscarino2007}.

\subsection{List of methods}

The list of IMEX Runge~-- Kutta methods we consider in this paper is given in Table~\ref{table:IMEX0}.

\begin{table}[t]
\caption{\label{table:IMEX0}List of IMEX RK methods considered in this paper}
\begin{center}
\begin{tabular}{|c|c|c|c|c|c|c|c|}
\hline 
Class & Notation & Ref. & \multicolumn{2}{|c|}{$\mathrm{CFL}_{\max}$} \\
& & & & $/ \sigma_{p}$ \\
\hline
ARS & ARS(1,2,1) & \cite{Ascher1997} & 1 & 1 \\
ARS & ARS(2,3,2) & \cite{Ascher1997} & 1.73 & 0.86 \\
ARS & ARS(3,4,3) & \cite{Ascher1997} & 2.82 & 0.94 \\
\hline
ARS & ARS(1,1,1) & \cite{Ascher1997} & -- & -- \\
ARS & ARS(2,2,2) & \cite{Ascher1997} & -- & -- \\
ARS & ARS(4,4,3) & \cite{Ascher1997} & 1.57 & 0.39 \\
\hline
ARS & MARS(3,4,3) & \cite{Boscarino2007a} & 2.82 & 0.94 \\
\hline
CK & ARK3(2)4L[2]SA & \cite{Kennedy2003} & 2.48 & 0.82 \\
CK & ARK4(3)6L[2]SA & \cite{Kennedy2003} & 4 & 0.8 \\
CK & ARK5(4)8L[2]SA & \cite{Kennedy2003} & 0.79 & 0.11 \\
\hline
CK & MARK3(2)4L[2]SA & \cite{Boscarino2007a} & 2.82 & 0.94 \\
CK & BHR(5,5,3) & \cite{Boscarino2009} & 2.25 & 0.56 \\
\hline
A & SSP2(3,2,2) & \cite{Pareschi2003} & -- & --  \\
A & SI-IMEX(3,3,2) & \cite{Boscarino2021} & 1.73 & 0.57 \\
A & SI-IMEX(4,4,3) & \cite{Boscarino2021} & 1.74 & 0.43 \\
A & SI-IMEX(4,3,3) & \cite{Boscarino2024} & 2.82 & 0.7 \\
\hline
\end{tabular}
\end{center}
\end{table}

For each method, we present
\begin{itemize}
\item a reference to a paper where the double Butcher tableu may be found;
\item the maximal value of $\lambda$ such that the segment $[0, i\lambda]$ belongs to the stability domain of the explicit counterpart of the method. This value is denoted as $\mathrm{CFL}_{\max}$.
\item $\mathrm{CFL}_{\max} / \sigma_{p}$ where $\sigma_{p}$ is the number of pressure solver calls per timestep (see Remark in Section~\ref{sect:stab}).
\end{itemize}

All values are rounded down to a hundredth. Dash means that the corresponding value is zero.

To compare, the third-order projection-type method in \cite{Franco2020} has $\mathrm{CFL}_{\max} \approx 0.72$ (and one pressure solver call per timestep).

\section{Segregated RK schemes}
\label{sect:main}

Here and below, $\B{C}(t,u)$ is a discretization of terms that should be treated explicitly, and $\B{D}(t,u)$ is a discretization of terms that should be treated implicitly. For IMEX methods, which are of a primary concern in this paper, $\B{C} \equiv -\B{F}_{conv} + \B{F}_{source}$, and $\B{D} \equiv \B{F}_{diff}$.
Recall that $D$ stands for a divergence approximation while $\B{D}(t,u)$ stands for the implicit term.

Introduce the notation for the momentum residual:
$$
\B{R}(t,u,p) = -\B{F}_{conv}(u) + \B{F}_{diff}(u) + \B{F}_{source}(t) - Gp.
$$
Or, which is the same,
$$
\B{R}(t,u,p) = \B{C}(t,u) + \B{D}(t,u) - Gp.
$$

\subsection{Original SRK schemes}
\label{sect:originalSRK}

The segregated Runge~-- Kutta schemes were introduced in \cite{Colomes2015} for the case $S = 0$. The system \eqref{eq_momentun_base}, $D u = H$, is a DAE of index 2. 
Taking the derivative of $Du = H$ and substituting $du/dt$ from \eqref{eq_momentun_base} we get $D\B{R}(t,u,p) = \dot{H}$, which expands to
$$
p = (DG)^{-1} \left[ D(\B{C}(t,u) + \B{D}(t,u)) - \dot{H} \right].
$$
Here and below $\dot{H} \equiv dH/dt$. Taking this back to \eqref{eq_momentun_base} we obtain
\begin{equation}
\frac{du}{dt} = \B{C}(t,u) + \B{D}(t,u) - G (DG)^{-1} \left[ D(\B{C}(t,u) + \B{D}(t,u)) - \dot{H} \right].
\label{eq_momentum_LBB}
\end{equation}

A segregated Runge~-- Kutta scheme for \eqref{eq_momentun_base}, $D u = H$ is an IMEX RK method \eqref{eq_IMEX_1}--\eqref{eq_IMEX_3} with the substitution $y \equiv u$, 
$$
\xi(t,y) \equiv \B{C}(t,u) - G (DG)^{-1} \left[ D(\B{C}(t,u) + \B{D}(t,u)) - \dot{H} \right], \quad \eta(t,y) \equiv \B{D}(t,u).
$$

{\it Remark.} The term $\dot{H}$ appears here naturally. Consider, for instance, the finite-difference method described in Section~\ref{sect:FD2}. The values $\B{u}_j(t)$ of $\B{u}(t) \in V_h$ at boundary nodes are zero. The actual velocity field is $\B{u}_j(t)$ at the internal nodes and $\B{u}_b(t, \B{r}_j)$ at the boundary nodes. Its divergence is $D \B{u}(t) - H(t)$. Similarly, $\B{C}(t,u) + \B{D}(t,u)$ is an approximation of the time derivative of the velocity assuming that there is no pressure gradient; its value belongs to $V_h$ and therefore is zero for boundary nodes. At a boundary node $j$, the time derivative is $\D\B{u}_b/\D t(t, \B{r}_j)$. Then the divergence of the velocity derivative is approximated by $D(\B{C}+\B{D}) - \dot{H}$.

~\\

A shortcoming of this method is that the divergence-free condition is enforced in the incremental way and therefore the divergence error may accumulate over time. In \cite{Colomes2015} (see Remark 3.4), it is suggested to project the velocity onto the divergence-free space when the divergence residual exceeds some threshold. However, there is a more delicate way to solve this issue. Instead of applying $d/dt$ to $Du = H$, we can apply $\alpha + d/dt$, $\alpha > 0$. This is known as the Baumgarte method \cite{Baumgarte1972}. This yields $D(\B{R}(t,u,p) + \alpha u) = \dot{H} + \alpha H$, which expands to
\begin{equation}
p(t, u) = (DG)^{-1}   \left[D(\B{C}(t,u) + \B{D}(t,u)) - \dot{H}(t) + \alpha (Du - H(t))\right].
\label{eq_p_bad}
\end{equation}

The use of $\xi(t,y) \equiv \B{C}(t,u) - G p(t, u)$ with $p(t,u)$ given by \eqref{eq_p_bad} would ruin the accuracy, because the value of $u$ taken to $\xi$ satisfies the continuity equation with the error $O(\tau)$. Instead, we take the term $Du - H$ at the beginning of the timestep. Then a corrected SRK scheme for \eqref{eq_momentun_base}, $D u = H$ may be defined as an IMEX RK method \eqref{eq_IMEX_1}--\eqref{eq_IMEX_3} with the substitution $y \equiv u$, 
\begin{gather}
\xi(t,y) \equiv \B{C}(t,u) - G (DG)^{-1} \left[D(\B{C}(t,u) + \B{D}(t,u)) - \dot{H}(t) + \alpha \Xi^{n-1}\right], 
\label{eq_SRK_xi}
\\
\eta(t,y) \equiv \B{D}(t,u),
\label{eq_SRK_eta}
\end{gather}
with $\Xi^{n-1} = Du^{n-1} - H(t^{n-1})$.

To choose $\alpha$, consider the case $\B{C} \equiv \B{D} \equiv 0$, $H \equiv 0$, and assume that the solution is a potential velocity field: $u(t) = G\phi(t)$. Then the implicit term vanishes, and an SRK scheme reduces to its explicit part for the equation $d\phi/dt = - \alpha \phi^{n-1}$, which has the solution $\phi^n = (1-\alpha \tau) \phi^{n-1}$. For stability, we need to take $0 \le \alpha \le 2 \tau^{-1}$. A natural choice is $\alpha = \tau^{-1}$.

The use of the Baumgarte method is not necessary for the original SRK schemes. 
Its importance grows in presence of pressure stabilization because a projection operator on a divergence-free space is undefined.

\subsection{SRK methods in case of pressure stabilization}
\label{sect:stab}

In this section, for simplicity we assume that $S$ is invertible and $\sigma_1 > 0$. The final form of the scheme does not use $S^{-1}$. The case $\sigma_1 = 0$ is discussed in the next section.

The system \eqref{eq_momentun_base}, \eqref{eq_continuity_base} is a singularly perturbed ODE system. Standard methods for systems of the form
$y' = f(y,z)$, $\varepsilon z' = g(y,z)$ require $(\D g/\D z)^{-1}$ to exist and be bounded as $\varepsilon \rightarrow 0$. This requirement is not met in our case.

Instead, we apply the operator $\alpha + d/dt$, $\alpha > 0$, to \eqref{eq_continuity_base} and subsitute $du/dt$ from \eqref{eq_momentun_base}:
\begin{equation}
\frac{d{u}}{dt} = \B{R}(t,u,p),
\label{eq_bb1}
\end{equation}
\begin{equation}
\sigma_1 S\frac{d^2p}{dt^2} = -\sigma_0 S \frac{dp}{dt} -D\B{R}(t,{u},p) + \dot{H} - \alpha\!\left(D{u} + \sigma_0 S p + \sigma_1 S\frac{dp}{dt}-H\right).
\label{eq_bb2}
\end{equation}

Denote $q = dp/dt$. Then \eqref{eq_bb1}--\eqref{eq_bb2} may be transformed to the form 
\begin{equation}
\begin{gathered}
\frac{dy}{dt} = \xi(t,y) + \eta(t,y), \quad y = (u, p, q)^T,
\\
\xi(t,y) = \left(\begin{array}{c} \B{C}(t,{u}) - G p \\ 0 \\ 0 \end{array}\right),
\\
\eta(t,y) = \left(\begin{array}{c} \B{D}(t,{u}) \\ q \\ (\sigma_1 S)^{-1}(-\sigma_0 S q -D\B{R}(t,{u},p) + \dot{H}(t) - \alpha \Xi^{n-1}) \end{array}\right),
\\
\Xi = Du + \sigma_0 Sp + \sigma_1 S q - H(t).
\end{gathered}
\label{eq_xi_eta_for_IMEX}
\end{equation}
\begin{definition}
We define a segregated RK method for \eqref{eq_momentun_base}, \eqref{eq_continuity_base} as an IMEX RK method of the form \eqref{eq_IMEX_1}--\eqref{eq_IMEX_3} with the substitution \eqref{eq_xi_eta_for_IMEX} while taking $\Xi$ from the $(n-1)$-th layer.
\end{definition}

An implicit stage \eqref{eq_IMEX_1} with $a_{jj} \ne 0$ needs a further elaboration. Let $u_j, p_j, q_j$ be the components of $y_j$ and $u_{j,*}, p_{j,*}, q_{j,*}$ be the components of $y_{j,*}$. Then \eqref{eq_IMEX_1} expands to
\begin{gather}
u_j = u_{j,*} + \breve{\tau} \B{D}(t_j, u_j),
\label{eq_impl_stage1_bbb}
\\
p_j = p_{j,*} + \breve{\tau} q_j,
\label{eq_impl_stage3_bbb}
\\
\breve{\tau}^{-1} \sigma_1 S(q_j - q_{j,*}) 
= -\sigma_0 S q_j-D\B{R}(t_j, u_j, p_j) + \dot{H}(t_j) - \alpha \Xi^{n-1}.
\label{eq_impl_stage2_bbb}
\end{gather}
Equation \eqref{eq_impl_stage1_bbb} does not involve pressure, so each implcit stage with $a_{jj} \ne 0$ begins with solving \eqref{eq_impl_stage1_bbb} for $u_j$.

Taking $q_j$ from \eqref{eq_impl_stage3_bbb} to  \eqref{eq_impl_stage2_bbb} we get
\begin{equation*}
\begin{gathered}
(\breve{\tau}^{-2} \sigma_1 + \breve{\tau}^{-1} \sigma_0) S(p_j - p_{j,*})
= -D\B{R}(t_j, u_j, p_j) + \dot{H}(t_j) - \alpha \Xi^{n-1}  + \breve{\tau} \sigma_1 S q_{j,*}.
\end{gathered}
\end{equation*}
This transforms to
\begin{equation}
\begin{gathered}
(DG - (\breve{\tau}^{-2}\sigma_1 + \breve{\tau}^{-1} \sigma_0) S)(p_j - p_{j,*}) = 
\\
= D\B{R}(t_j, u_j, p_{j,*}) - \dot{H}(t_j) + \alpha \Xi^{n-1} - \breve{\tau}^{-1} \sigma_1 Sq_{j,*}.
\end{gathered}
\label{eq_impl_stage2_bbb3}
\end{equation}
After $p_j$ is evaluated, one finds $q_j \equiv K_j^p$ from  \eqref{eq_impl_stage3_bbb} and $K_j^q = (q_j - q_{j,*})/\breve{\tau}$ where $K_j^u$, $K_j^p$, $K_j^q$ are the components of $K_j$.

Recall that $S = DG - L$, and $L$ is the operator we want in the pressure equation. Equating the left-hand side of \eqref{eq_impl_stage2_bbb3} to $L (p_j - p_{j,*})$ we obtain
\begin{equation}
\breve{\tau}^{-2}\sigma_1 + \breve{\tau}^{-1}\sigma_0 = 1.
\label{eq_sigma0sigma1}
\end{equation}

In order to use one parameter instead of two ($\sigma_0$ and $\sigma_1$), we put
\begin{equation}
\sigma_0 = (1-r_{\sigma}) \breve{\tau}, \quad \sigma_1 = r_{\sigma} \breve{\tau}^2
\label{eq_sigma0sigma1_2}
\end{equation}
with $r_{\sigma} \in (0, 1]$. The case $r_{\sigma} = 0$ is special because the orders of differential equations \eqref{eq_continuity_base} and \eqref{eq_bb2} decrease by one; we discuss it in Section~\ref{sect:sigma10}.

For methods of type CK that are not of type ARS, the expression for $q_{j,*}$, $j \ge 2$, contains $K_1^q$. Since $a_{11}=0$ and $K_1 = \eta(t^{n-1}, y^{n-1})$ (see \eqref{eq_IMEX_1}, \eqref{eq_IMEX_2}), we cannot evaluate $q_{j,*}$ and $K_j^q$ without the use of $S^{-1}$. 

To mitigate this issue (and for methods of other types -- to reduce the computational workload), we construct expressions for $Sq_{j,*}$ and $SK_j^q$ without the use of $S^{-1}$. Write 
\begin{equation}
\breve{\tau}^{-1} \sigma_1 Sq_{j,*} = S \mu_j + \nu_j, \quad
\sigma_1 SK_j^q = S\tilde{\mu}_j - \tilde{\nu}_j.
\label{eq_def_muj_1}
\end{equation}
We look for recursive formulas for $\mu_j$, $\tilde{\mu}_j$, $\nu_j$, $\tilde{\nu}_j$. Applying $\breve{\tau}^{-1} \sigma_1 S$ to the last component of \eqref{eq_IMEX_1a} we get
$$
\breve{\tau}^{-1} \sigma_1 S q_{j,*} = \breve{\tau}^{-1} \sigma_1 S q^{n-1} + \tau \sum\limits_{k=1}^{j-1} a_{jk} \breve{\tau}^{-1} \sigma_1 S K_k^q,
$$
(there are no explicit terms in the equation for the pressure derivative). Equating the expressions at $S$ and without $S$ we get
\begin{equation}
\mu_j = \breve{\tau}^{-1} \sigma_1 q^{n-1} + \tau \breve{\tau}^{-1} \sigma_1 \sum\limits_{k=1}^{j-1} a_{jk} \tilde{\mu}_j, \quad
\nu_j = - \tau \breve{\tau}^{-1} \sum\limits_{k=1}^{j-1} a_{jk} \tilde{\nu}_j.
\label{eq_expr_muj}
\end{equation}

Consider a method of type CK. Since
$$
\sigma_1 SK_1^q = - \sigma_0 Sq -D\B{R}(t^{n-1}, u^{n-1}, p^{n-1})  + \dot{H}(t^{n-1}) - \alpha \Xi^{n-1},
$$
$$
SK_j^q = (Sq_j - Sq_{j,*})/\breve{\tau}, \quad j \ge 2,
$$
we write
\begin{equation}
\begin{gathered}
\tilde{\mu}_1 = - \sigma_0 q^{n-1} - \alpha (\sigma_0 p^{n-1} + \sigma_1 q^{n-1}),
\\
\tilde{\nu}_1 = D\B{R}(t^{n-1},u^{n-1},p^{n-1}) - \dot{H}(t^{n-1}) + \alpha (Du^{n-1} - H(t^{n-1})),
\\
\tilde{\mu}_j = \breve{\tau}^{-1} \sigma_1 q_j-\mu_j, \quad j \ge 2,
\\
\tilde{\nu}_j = \nu_j, \quad j \ge 2.
\end{gathered}
\label{eq_expr_muj_tilde}
\end{equation}

Then pressure equation \eqref{eq_impl_stage2_bbb3} takes the form
\begin{equation}
L(p_j - \tilde{p}_j) = D\B{R}(t_j, u_j, \tilde{p}_j) -
\dot{H}(t_j) + \alpha (Du^{n-1} - H(t^{n-1})) - \nu_j.
\label{eq_impl_stage2_bbb4}
\end{equation}
with
$$
\tilde{p}_j = p_{j,*} + {\mu}_j - \alpha (\sigma_0 p^{n-1} + \sigma_1 q^{n-1}).
$$

From \eqref{eq_expr_muj} and \eqref{eq_expr_muj_tilde}, it is easy to see that $\nu_j = d_j \tilde{\nu}_1$, $j \ge 2$, where $d_j$ are the solution-independent constants defined by
\begin{equation}
d_1 = 1, \quad d_j = - (a_{ss})^{-1} \sum\limits_{k=1}^{j-1} a_{jk} d_k.
\label{eq_def_d}
\end{equation}
Equation \eqref{eq_def_d} is equivalent to $\sum a_{jk} d_k = 0$. Therefore, the vector $(d_1, \ldots, d_s)^T$ is the vector from $\mathrm{Ker}\, A$ normalized by $d_1 = 1$. For all methods from Table~\ref{table:IMEX0}, $d_s = 0$ (which means that in the case $\B{C}(t,u) \equiv \B{D}(t,u) \equiv 0$, $H \equiv 0$, $\alpha = 0$, $S = 0$, the pressure at the end of timestep should be $p^n = 0$).

For the methods of type A, all conditions $j \ge 2$ should be replaced by $j \ge 1$, the special expressions for $\tilde{\mu}_1$ should be dropped, and $\nu_j = 0$ for all $j$. For methods of type ARS, $d_j  = 0$ for $j \ge 2$, so $\nu_j = 0$ and there is no need to evaluate $\tilde{\nu}_1$.~\\

In the majority of situations, the solution by a segregated Runga~-- Kutta method satisfies the discrete continuity equation exactly. This feature is inherited from the original SRK methods, see Proposition~3.2 in \cite{Colomes2015}. We adapt the statement and the proof of this proposition for our case.

\begin{proposition}\label{th:continuity}
Let $\Xi^n$ is defined by \eqref{eq_xi_eta_for_IMEX}. Assume that every component of $H(t)$ is a $p$-th order polynomial in time, the initial condition satisfies $\Xi^0 = 0$, the RK integrator integrates exactly polynomials of order $p-1$, and $b = \hat{b}$. Then the SRK method preserves the exact discrete divergence constraint at all time steps, i. e. $\Xi^n = 0$ for each $n$. 
\end{proposition}
\begin{proof}
Assume that $\Xi^{n-1} = 0$. The velocity component of \eqref{eq_IMEX_3} is
$$
u^n = u^{n-1} + \tau \sum\limits_{k=1}^s b_k K_k^u + \tau \sum\limits_{k=1}^s \hat{b}_k \hat{K}_k^u.
$$
By assumption $b_k = \hat{b}_k$, $k = 1, \ldots, s$. By definition, $K_k^u + \hat{K}_k^u = \B{R}(t_k, u_k, p_k)$. Thus,
$$
u^n = u^{n-1} + \tau \sum\limits_{k=1}^s b_k \B{R}(t_k, u_k, p_k).
$$
Apply $D$ to this equation. Take $Du^{n-1}$ from \eqref{eq_xi_eta_for_IMEX} using $\Xi^{n-1} = 0$, $Du^n$ also from \eqref{eq_xi_eta_for_IMEX}, and $D\B{R}(t_k, u_k, p_k)$ from \eqref{eq_impl_stage2_bbb}. After reordering terms we obtain
$$
\Xi^n = \left(H(t^n) - H(t^{n-1}) - \tau \sum\limits_{k=1}^s b_k \dot{H}(t_k)\right) + \sigma_0 S \left(p^n - p^{n-1} - \tau \sum\limits_{k=1}^s b_k q_k\right) + 
$$
$$
+ \sigma_1 S\left(q^n - q^{n-1} - \tau \breve{\tau}^{-1} \sum\limits_{k=1}^s b_k (q_k - q_{k,*})\right).
$$
The first term is zero by assumption ($H(t)$ is a $p$-th orer polynomial, which is integrated exactly). The last two terms are zero by construction (equating them to zero we get the definitions of $p^n$ and $q^n$). Thus, $\Xi^n = 0$.
\end{proof}

{\it Remark.} An SRK method calls the pressure solver $\sigma_p = s$ times per timestep if the underlying IMEX RK scheme is of type A and $\sigma_p = s-1$ times if it is of type CK. This may cause confusion regarding the notation. For instance, BHR(5,5,3) calls the pressure solver $\sigma_p = 4$ times per timestep while $\sigma_{im} = \sigma_{ex} = 5$. We preserve the original naming of the methods. 

{\it Remark.} Let us discuss the case $S = 0$. If the underlying IMEX method is of type A, then we recover the original SRK schemes \cite{Colomes2015} (for $\alpha=0$) or their versions with the Baumgarte correction (for $\alpha \ne 0$, see Section~\ref{sect:originalSRK}). However, if the IMEX method is of type CK, taking $S = 0$ does not yield the original SRK scheme. The difference is in the first stage ($j=1$). Since $a_{11} = 0$, then $a_{11}^{-1}S$ is a zero-by-zero uncertainty. Original SRK schemes resolve it by assuming $a_{11}^{-1}S = 0$. Our methods do this by assuming $a_{11}^{-1}S = \infty$, i. e. puts $p_1 = p_{1,*} \equiv p^{n-1}$. Therefore, the pressure solver is skipped for $j=1$.

\subsection{The case $\sigma_1 = 0$}
\label{sect:sigma10}

Let us discuss the case $r_{\sigma} = 0$, i. e. $\sigma_0 = \breve{\tau}$, $\sigma_1= 0$. It is special because \eqref{eq_bb2} is now a first order ODE and not a second order ODE. Therefore, we do not need to consider the pressure derivative as a separate variable. 

For methods of type A and ARS, following the way described in the Section~\ref{sect:stab} but with $y = (u, p)^T$, we come to the same formulas with the substitution $\sigma_0 = \breve{\tau}$, $\sigma_1= 0$. Note that we do not need to specify the value of $q^{n-1}$.

For methods of type CK that are not of type ARS, the situation is different. In the expression of $\tilde{\mu}_1$ in \eqref{eq_expr_muj_tilde}, the term $q^{n-1}$ appears. In constrast, going with $y = (u, p)^T$, one may find that $p_{j,*}$ and $q_j$ are not well-defined if $S$ is degenerate. To mitigate this issue, we may put
$$
\breve{\tau}^{-1} \sigma_0 Sp_{j,*} = S \mu'_j + \nu_j, \quad
\sigma_0 Sq_j = S\tilde{\mu}'_j - \tilde{\nu}_j.
$$
We omit the detailed derivation of the method because is resembles the one in Section~\ref{sect:stab}. The result is
$$
\mu'_j = p^{n-1} + \tau \sum\limits_{k=1}^{j-1} a_{jk} \tilde{\mu}'_j, 
\quad
\tilde{\mu}'_1 = -\alpha \breve{\tau} p^{n-1},
\quad
\tilde{\mu}'_j = p_j - \mu'_j, \quad j \ge 2,
$$
and $\tilde{p}_j = \mu'_j - \alpha\breve{\tau} p^{n-1}$. All other formulas the same as in Section~\ref{sect:stab}.

Proposition~\ref{th:continuity} also extends to this case.

\subsection{Final form of an SRK scheme}

We come to the following form of segregated Runge~-- Kutta schemes based on IMEX methods \eqref{eq_IMEX_1}--\eqref{eq_IMEX_3} for the time integration of \eqref{eq_momentun_base}, \eqref{eq_continuity_base} with operators $D$, $G$, $L$ satisfying \eqref{eq_gen_assumptions}. Let $r_{\sigma}$ be the stabilization type parameter. Here we consider two extreme cases, $r_{\sigma}=0$ and $r_{\sigma}=1$. 

Let $u^{n-1}$, $p^{n-1}$, $q^{n-1}$ be the solution at $t^{n-1}$, and $\tau = t^{n} - t^{n-1}$ be the timestep. We assume that the underlying IMEX method is (1) stiffly accurate and (2) either SDIRK or ESDIRK. Let $j_0 = 1$ for methods of type A and $j_0 = 2$ otherwise. Denote $\breve{\tau} = a_{ss} \tau$. Let $d_k$ be defined by
$$
d_1 = 1, \quad d_j = -  (a_{ss})^{-1} \sum\limits_{k=1}^{j-1} a_{jk} d_k, \quad j = 2, \ldots, s.
$$

1. For methods of type CK (including ones of type ARS), evaluate
$$
\hat{K}_1^u = \B{C}(t^{n-1}, u^{n-1}) - Gp^{n-1}.
$$
For methods of type CK that are not of type ARS, also put $q_1 = q^{n-1}$ (for $r_{\sigma}=1$) and evaluate
$$
K_1^u = \B{D}(t^{n-1}, u^{n-1}),
$$
$$
\tilde{\nu}_1 = D(\hat{K}_1^u + K_1^u) - \dot{H}(t^{n-1}) + \alpha (D u^{n-1} - H(t^{n-1})).
$$

2. For each $j = j_0, \ldots, s$:

2.1. Put $t_j = t^{n-1} + \tau c_j$, $c_j = \sum_{k=1}^j a_{jk}$.

2.2. Implicit velocity step: define
$$
u_{j,*} = u^{n-1} + \tau \sum\limits_{k=1}^{j-1} (a_{jk} K_k^u + \hat{a}_{jk} \hat{K}_k^u),
$$
solve
$$
u_j = u_{j,*} + \breve{\tau} \B{D}(t_j, u_j),
$$
and put $K_j^u = \breve{\tau}^{-1}(u_j - u_{j,*})$.

2.3. Evaluate the explicit velocity fluxes: 
$$
C_j^u = \B{C}(t_j, u_j).
$$

2.4 Pressure step: define
$$
\mu_j = (1 - a_{j1} \alpha \tau)^{j_0-1} w^{n-1} + (a_{ss})^{-1} \sum\limits_{k=j_0}^{j-1} a_{jk} \tilde{\mu}_k
$$
where $w^{n-1} = p^{n-1}$ for $r_{\sigma}=0$ and $w^{n-1} = \breve{\tau} q^{n-1}$ for $r_{\sigma}=1$. 
For methods of type CK (excluding ARS), put $\nu_j = d_j \tilde{\nu}_1$, otherwise put $\nu_j = 0$. For $r_{\sigma}=0$, put $\hat{p}_{j,*} = 0$, and for $r_{\sigma}=1$, put
$$
\tilde{p}_{j,*} = p^{n-1} + \tau \sum\limits_{k=1}^{j-1} a_{jk} q_k.
$$
Put $\tilde{p}_j = \tilde{p}_{j,*}+\mu_j - \alpha \breve{\tau}w^{n-1}$.
Solve
$$
L (p_j - \tilde{p}_j) = D(K_j^u + C_j^u - G \tilde{p}_j) - \dot{H}(t_j) + \alpha (D u^{n-1} - H(t^{n-1})) - \nu_j.
$$
Put
$$
\tilde{\mu}_j = p_j - (\tilde{p}_{j,*}+\mu_j).
$$
For $r_{\sigma}=1$, put $q_j = (p_j - \tilde{p}_{j,*})/\breve{\tau}$.

2.5. Evaluate
$$
\hat{K}_j^u = C_j^u - Gp_j.
$$

3. Put $p^n = p_s$, $q^n = q_s$ (for $r_{\sigma}=1$ only), and evaluate
$$
u^n = u^{n-1} + \tau \sum\limits_{k=1}^{s} (b_k K_k^u + \hat{b}_{k} \hat{K}_k^u).
$$

\subsection{Example: forward-backward Euler method}
\label{sect:firstorder}

As an example, we consider an SRK method based on ARS(1,2,1) (i. e. on the forward-backward Euler method \eqref{eq_Butcher_Euler}). For simplicity, assume \mbox{$H(t) \equiv 0$}.

First we evaluate
\begin{equation}
\hat{K}_1^u = \B{C}(t^{n-1}, u^{n-1}) - G p^{n-1}.
\label{eq_FOM0}
\end{equation}

In the implicit velocity step, we write
$$
u_{2,*} = u^{n-1} + \tau \hat{K}_1^u,
$$
$$
u_2 = u_{2,*} + \tau \B{D}(t^{n}, u_2), \quad K_2^u = (u_2 - u_{2,*})/\tau.
$$
This expands to
\begin{equation}
\frac{u_2 - u^{n-1}}{\tau} = \B{C}(t^{n-1}, u^{n-1}) + \B{D}(t^{n}, u_2) - G p^{n-1},
\label{eq_FOM1}
\end{equation}
which means that $u_2$ coincides with the predictor velocity in a conventional first order projection-type method.

The second part of the implicit stage is the pressure update. Since the backward Euler method is stiffly accurate, $p_2$ coincides with $p^n$ and (for $r_{\sigma} = 1$) $q_2$ coincides with $q^n$. The pressure equation has the form
\begin{equation}
L(p^n-\tilde{p}) = D\B{R}(t^{n}, u_2, \tilde{p}) + \alpha D u^{n-1}
\label{eq_FOM2a}
\end{equation}
where $\tilde{p} = (1 - \alpha\tau) p^{n-1}$ if $r_{\sigma} = 0$ and $\tilde{p} = p^{n-1}+\tau (1 - \alpha\tau) q^{n-1}$ if $r_{\sigma} = 1$. 


At the second explicit stage we evaluate
$$
\hat{K}_2^u = \B{C}(t^{n}, u_2) -  G p^n.
$$

Finally, the velocity at the new time layer is
\begin{equation}
u^n = u^{n-1} + \tau (K_2^u + \hat{K}_2^u).
\label{eq_FOM3}
\end{equation}

The most natural choice is $\alpha = \tau^{-1}$. Then we have $\tilde{p} = 0$ or $\tilde{p} = p^{n-1}$ depending on the case. Then \eqref{eq_FOM2a} and \eqref{eq_FOM3} reduce to
$$
L(p^n-\tilde{p}) = \tau^{-1} D(\tilde{u}^{n}), \quad u^n = \tilde{u}^{n} - G (p^n-\tilde{p})
$$
with $\tilde{u}^{n}$ defined by
$$
\frac{\tilde{u}^{n} - u^{n-1}}{\tau} = \B{C}(t^{n}, u_2) + \B{D}(t^n, u_2) - G \tilde{p}
$$
and $u_2$ defined by \eqref{eq_FOM1}. Thus, SRK schemes based on ARS(1,2,1) with the choice $\alpha = \tau^{-1}$ belong to the class of pressure-correction schemes. Taking $r_{\sigma}=0$ we obtain a total pressure method (also called a non-incremental pressure correction method) and $r_{\sigma}=1$ we obtain an incremental pressure correction method.

\subsection{Spectral analysis}
\label{sect:spectral2}

To conclude the description of SRK schemes, it remains to specify the range of parameter $\alpha$. By the dimensional consideration, $\alpha \sim \tau^{-1}$. Therefore, if the timestep is small enough, then the suppression of potential velocity modes are quicklier than the convection and diffusion. Therefore, as a necessary stability condition, we may check the case $\B{C}(t, u) \equiv 0$, $\B{D}(t, u) \equiv 0$, $H(t) \equiv 0$. Under these assumptions, the divergence-free velocity modes are steady and correspond to zero pressure, so they are of no interest. Therefore, we consider only potential velocity modes, $u^n = \tau G \phi^n$.

Assume that $w$ is a common eigenvector of $DG$ and $L$, and let $\beta \in [0,1]$ be defined by \eqref{eq_def_Omega}: $DG w = (1-\beta) L w$. Then $Sw = -\beta Lw$. We are looking for solutions of the form $q^{n} = \hat{q}^{n} w$, $p^{n} = \hat{p}^{n} w$, $u^{n} = \tau G w \hat{\phi}^{n}$. Then
$$
\left(\begin{array}{c} \hat{q}^n \\ \hat{p}^n \\ \hat{\phi}^n \end{array}\right) = M_A \left(\begin{array}{c} \hat{q}^{n-1} \\ \hat{p}^{n-1} \\ \hat{\phi}^{n-1} \end{array}\right)
\quad \mathrm{or} \quad
\left(\begin{array}{c} \hat{p}^n \\ \hat{\phi}^n \end{array}\right) = M_A \left(\begin{array}{c} \hat{p}^{n-1} \\ \hat{\phi}^{n-1} \end{array}\right)
$$
(the latter is for the case $\sigma_1 = 0$). We want to find the range of $\alpha$ such that all eigenvalues $\lambda$ of the amplification matrix $M_A$ satisfy $|\lambda| \le 1$.

Applying our assumptions is equivalent to the following changes: $\B{C}, \B{D} \to 0$, $H \to 0$, $\tau \to 1$, $u \to \hat{\phi}$, $p \to \hat{p}$, $q \to \hat{q}$, $G \to 1$, $L^{-1}D \to (1-\beta)$.

Doing the analysis of the general case by hand is difficult. We study some particular cases below. Before doing this, we analyze the methods listed in Table~\ref{table:IMEX0} numerically. The results are as follows:
\begin{itemize}
\item for each method from Table~\ref{table:IMEX0} that satisfies $b = \hat{b}$, there holds $(\alpha \tau)_{\max} = 2$ for both $r_{\sigma} = 0$ and $r_{\sigma} = 1$;
\item for ARS(1,1,1) we have $(\alpha \tau)_{\max} = 1$ for both $r_{\sigma} = 0$ and $r_{\sigma} = 1$;
\item for ARS(2,2,2) we have $(\alpha \tau)_{\max} \approx 0.82$ for $r_{\sigma} = 0$, and for $r_{\sigma} = 1$ the amplification matrix has an eigenvalue $\lambda$, $|\lambda|>1$, for any $\alpha \ne 0$;
\item for ARS(4,4,3) we have $(\alpha \tau)_{\max} = 2$ for $r_{\sigma} = 0$ and $(\alpha \tau)_{\max} \approx 1.43$ for $r_{\sigma} = 1$.
\end{itemize}

As an example, consider methods of type CK at two extreme cases, $\beta = 0$ and $\beta = 1$. We begin with the no-stabilization limit ($\beta = 0$). There holds
\begin{equation}
\hat{p}_j = \alpha \hat{\phi}^{n-1} - d_j (\alpha \hat{\phi}^{n-1} - \hat{p}^{n-1}), \quad j \ge 2.
\label{eq_hat_pj}
\end{equation}
Since $d_s = 0$, we have $\hat{p}^n = \hat{p}_s = \alpha \hat{\phi}^{n-1}$. Furthermore,
$$
\hat{\phi}^n = \hat{\phi}^{n-1} - \hat{b}_1 \hat{p}^{n-1} - \sum\limits_{k=2}^s \hat{b}_k \hat{p}_k.
$$
Taking $\hat{p}_j$ from \eqref{eq_hat_pj}, 
$$
\hat{\phi}^n = \hat{\phi}^{n-1} - \hat{b}_1 \hat{p}^{n-1} - \sum\limits_{k=2}^s \hat{b}_k (\alpha \hat{\phi}^{n-1} - d_k (\alpha \hat{\phi}^{n-1} - \hat{p}^{n-1})).
$$
There holds $\sum_{k=1}^s \hat{b}_k = 1$ (first order condition). Denote $\theta = \sum_{k=1}^s \hat{b}_k d_k$. If $b = \hat{b}$, then $\theta = 0$ ($b_k = a_{sk}$ and $d$ is orthogonal to each row of $A$ by construction). Then we have
$$
\hat{\phi}^n = (1-\alpha +\alpha\theta)\hat{\phi}^{n-1} - \theta \hat{p}^{n-1}.
$$
It can also be shown that $\hat{q}^n$ does not depend on $\hat{q}^{n-1}$ (provided that $d_s = 0$). Therefore, the amplification matrix has the form
$$
M_A = \left(\begin{array}{ccc} 0 & * & * \\ 0 & 0 & \alpha \\ 0 & -\theta & 1-\alpha +\alpha\theta \end{array}\right),
$$
and we can write a necessary stability condition:
$$
\alpha \le \frac{1}{\max\{\theta, 1/2 - \theta\}}.
$$
In particular, if $b = \hat{b}$, we have $\theta=0$, and the condition reduces to $\alpha \le 2$.

Now consider the case $\beta = 1$, which corresponds to checkerboard oscillations, for the methods of type CK. In this case, stage pressure values do not depend on $\hat{\phi}^{n-1}$. We have
$$
\hat{\tilde{\mu}}_j = - a_{ss} \alpha \hat{w}^{n-1},
\quad
\hat{\mu}_j = (1 - (c_j - a_{ss}) \alpha) \hat{w}^{n-1},
\quad
\hat{q}_j = a_{ss}^{-1} (1 - c_j \alpha) \hat{w}^{n-1}
$$
where $\hat{w}^{n-1} = \hat{p}^{n-1}$ for $r_{\sigma} = 0$ and $\hat{w}^{n-1} = a_{ss} \hat{q}^{n-1}$ for $r_{\sigma} = 1$. Thus, for the case $r_{\sigma} = 0$, we have $\hat{p}_s = (1 - \alpha) \hat{p}^{n-1}$ and the amplification matrix
$$
M_A = \left(\begin{array}{cc} 1 - \alpha  & 0 \\ * & 1 \end{array}\right).
$$
The stability condition is $\alpha \le 2$. For the case $r_{\sigma} = 1$, we have
$$
\hat{q}_j = (1 - c_j \alpha) \hat{q}^{n-1}, \quad \hat{p}_j = \hat{p}^{n-1} + \left(\ldots\right) \hat{q}^{n-1}.
$$
So the amplification matrix has the form
$$
M_A = \left(\begin{array}{ccc} 1 - \alpha  & 0 & 0 \\ * & 1 & 0 \\ * & * & 1 \end{array}\right).
$$
The eigenvalues satisfy $|\lambda| \le 1$ if $0 \le \alpha \le 2$. However, there is an eigenvalue $\lambda = 1$ of multiplicity 2. Numerical evaluation shows that the values marked by stars are indeed nonzero. Thus, $\|M_A^n\|$ grows linearly with $n$, and, from the formal point of view, these methods are unstable. However, this instability is only linear (and not exponential), and we expect that a small amount of viscosity (which always exists in practical cases) is enough to stabilize the computations.

\section{Verification}
\label{sect:verification}

In all numerical experiments, we put $\alpha \tau = 0.5 (\alpha \tau)_{max}$ where $(\alpha \tau)_{max}$ is discussed in Section~\ref{sect:spectral2}. 


\subsection{Manufactured analytical solution}

The verification of SRK methods in presense of pressure stabilization is tricky. Dealing with methods with no pressure stabilization, one can start with the non-stiff case (for instance, by choosing a coarse mesh and a small timestep) and verify that the convergence order coincides with the expected one. In our case, as we discussed in Section~\ref{sect:stabtypes}, reducing the timestep changes coefficients in front of the stabilization operator, so long-wave potential velocity modes can never be resolved. There is no non-stiff limit.

We begin with a small test proposed in \cite{Colomes2015}. The idea is to use a solution where all spatial approximations are exact, so that the solution error originates from the time integration only.

Let $\Omega = (0,1)^2$, $\nu = 0.1$, and the exact solution be
\begin{equation}
\tilde{\B{u}}(t,x,y) = \left(\begin{array}{c} x \\ -y \end{array}\right) g(t), \quad \tilde{p}(t,x,y) = x+y
\label{eq_ctest_exactsol}
\end{equation}
with $g(t) = \exp(t/25) \sin(\pi t/10)$. The Dirichlet boundary conditions $\B{u}_b(t,x,y) = \tilde{\B{u}}(t,x,y)$ are prescribed on $\D \Omega$, and the initial data are prescribed for $t=0$. The source term $\B{f}_{source}(t,\B{r})$ is defined so that \eqref{eq_ctest_exactsol} satisfies \eqref{eq_NS_2}. The computational mesh is uniform, with 10 elements per each direction. At $t = t_{\max} = 0.1$, we evalutate the solution error
\begin{equation}
e_u = \max\limits_{j \in \bar{\mathcal{N}}} |\B{u}_j(t) - \tilde{\B{u}}(t,\B{r}_j)|, \quad
e_p = \max\limits_{j \in \bar{\mathcal{N}}} |p_j(t) - \tilde{p}(t,\B{r}_j)|
\label{eq_def_eu_ep}
\end{equation}
where $\bar{\mathcal{N}}$ is the set of mesh nodes and $\B{r}_j$ is the raduis-vector of node $j \in \bar{\mathcal{N}}$.

We use the finite-difference discretization described in Section~\ref{sect:FD2}. Convective fluxes are approximated with the basic (``second-order'') central differences based on the skew-symmetric splitting. For the viscous term, we use the standard 5-point Laplace approximation. The timestep size is $\tau = 0.1 \cdot 2^{-n}$, $n = 0, \ldots, 4$.

We consider explicit, IMEX, and implicit SRK methods. In the explicit and IMEX methods, we take the source term to the explicit part. In the implicit SRK methods, we take the source term to the implcit part. Although the source term does not depend on the numerical solution, this treatment affects the overall accuracy.

\begin{figure}[p]
\centering
\includegraphics[width=0.3\textwidth]{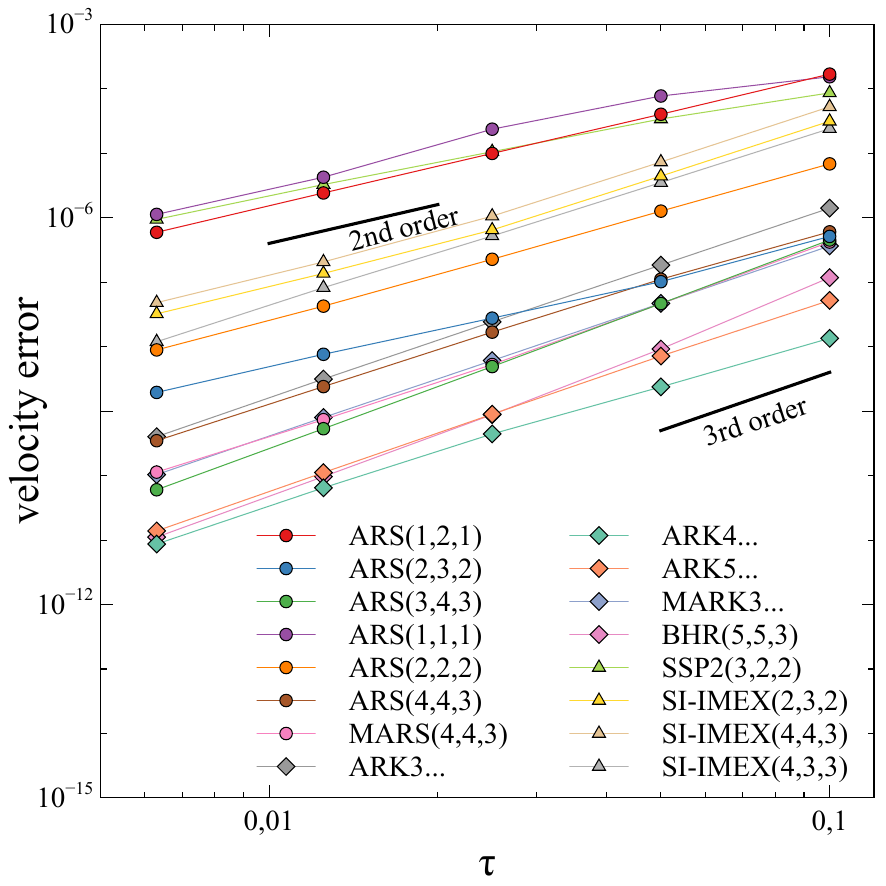}
\includegraphics[width=0.3\textwidth]{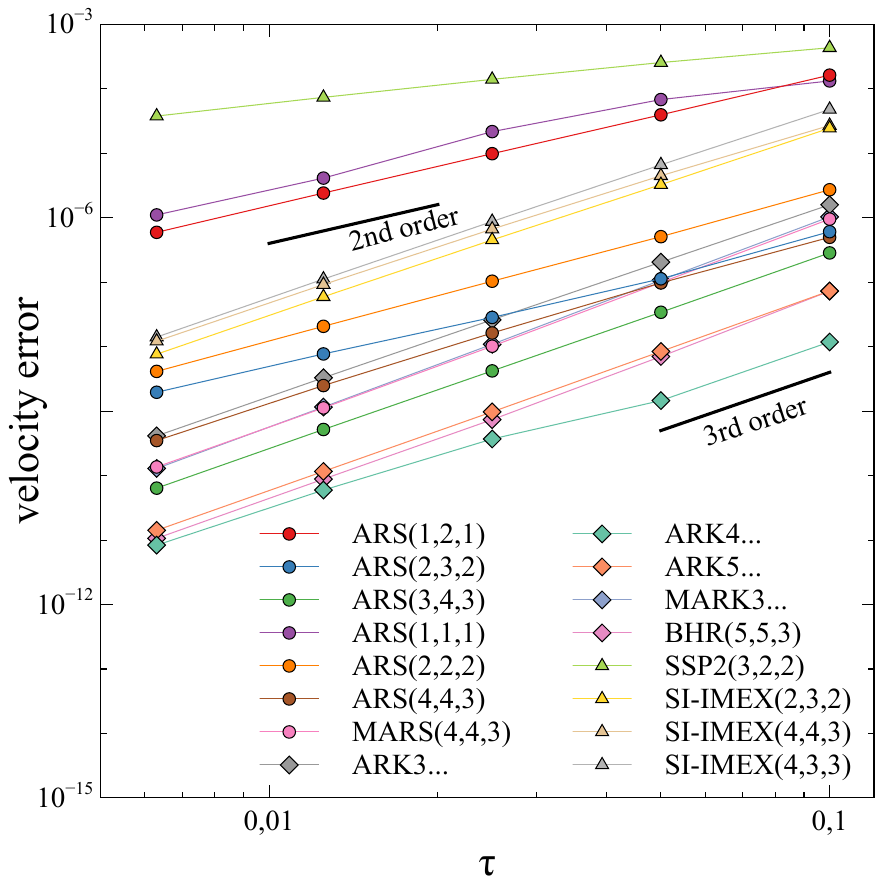}
\includegraphics[width=0.3\textwidth]{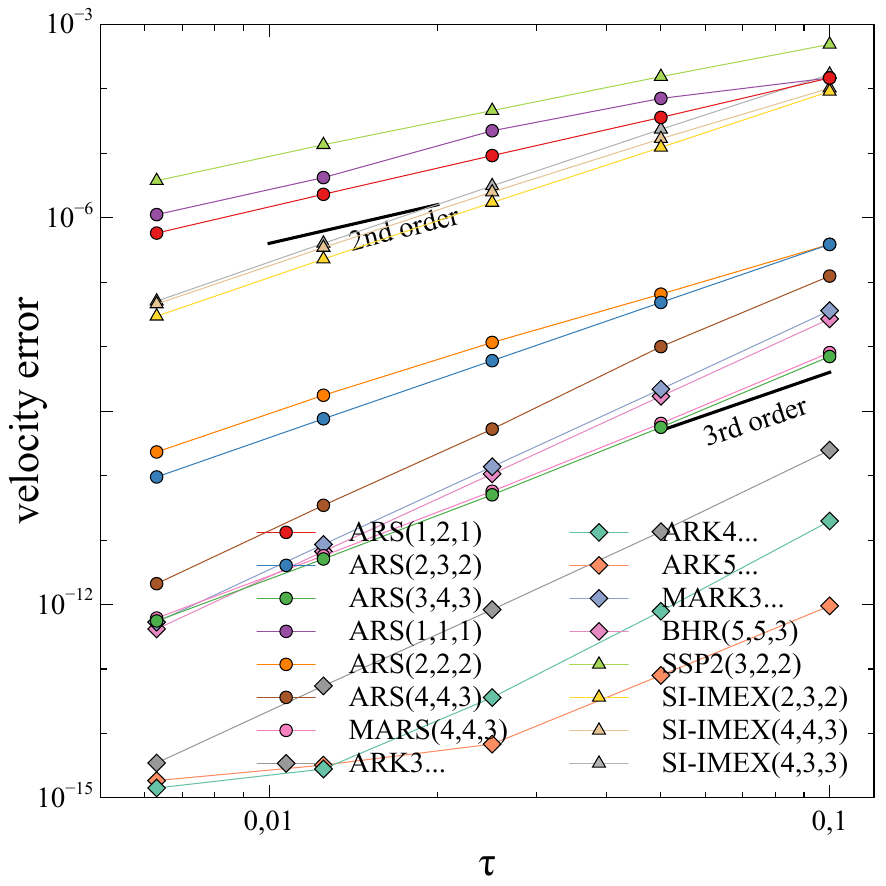}
\includegraphics[width=0.3\textwidth]{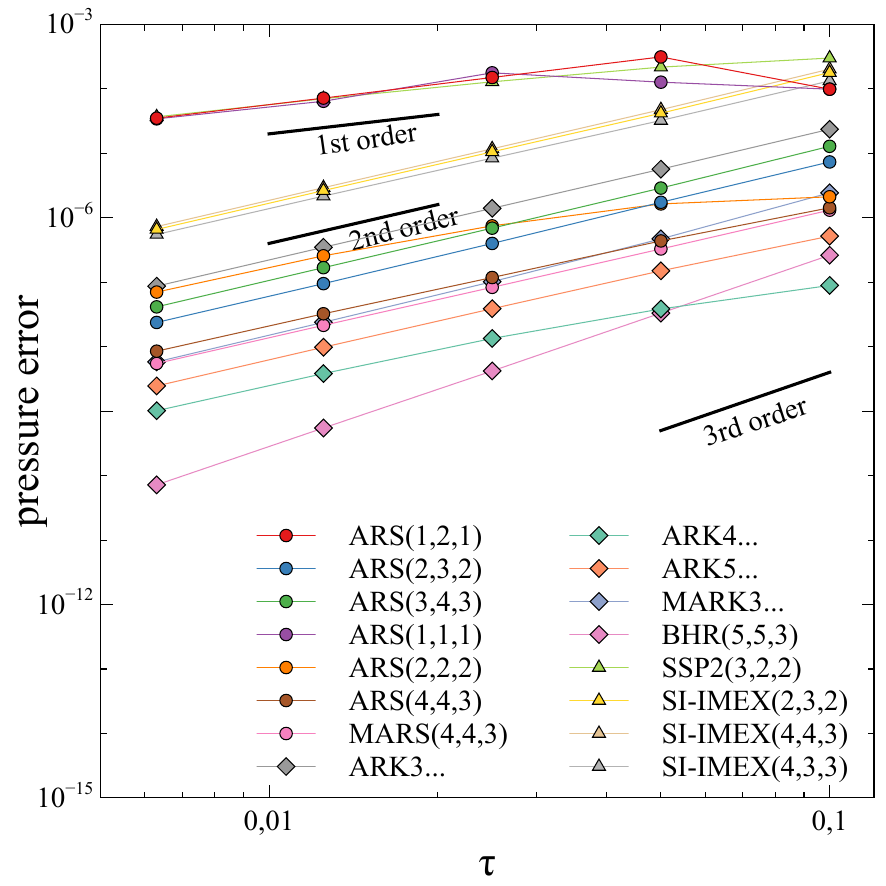}
\includegraphics[width=0.3\textwidth]{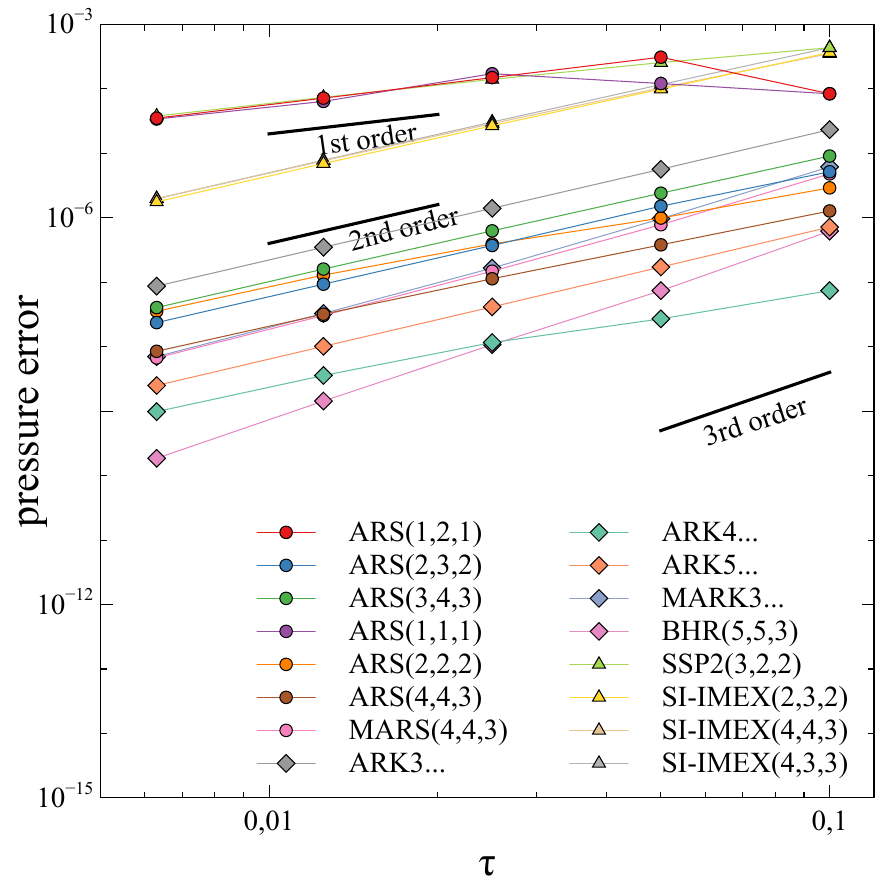}
\includegraphics[width=0.3\textwidth]{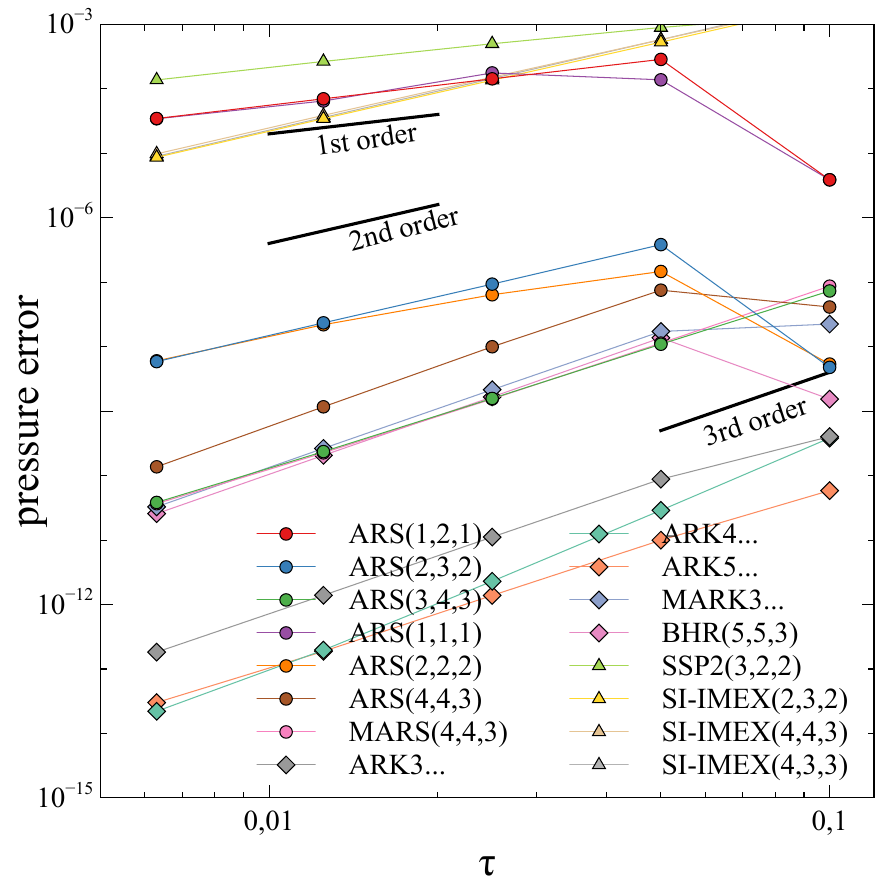}
\caption{Solution error for the test with a manufactured solution. $r_{\sigma} = 0$. Left to right: explicit SRK schemes, IMEX SRK schemes, implicit SRK schemes. Top: $e_u$; bottom: $e_p$}\label{fig:results_ctest1}
~\\~\\

\centering
\includegraphics[width=0.3\textwidth]{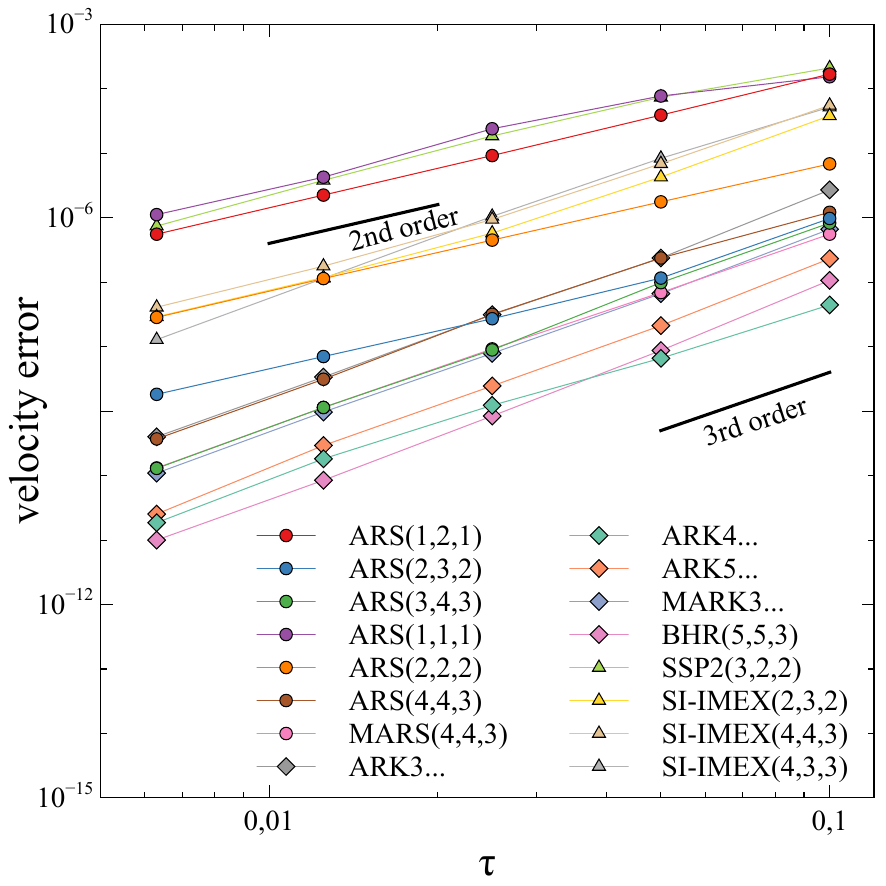}
\includegraphics[width=0.3\textwidth]{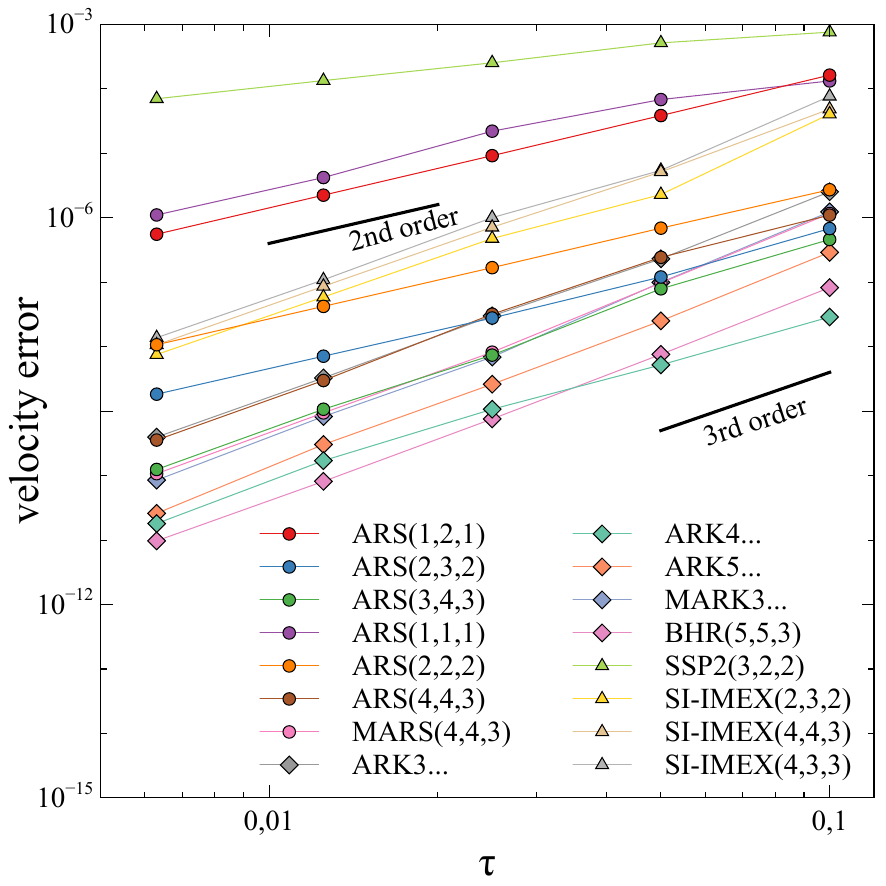}
\includegraphics[width=0.3\textwidth]{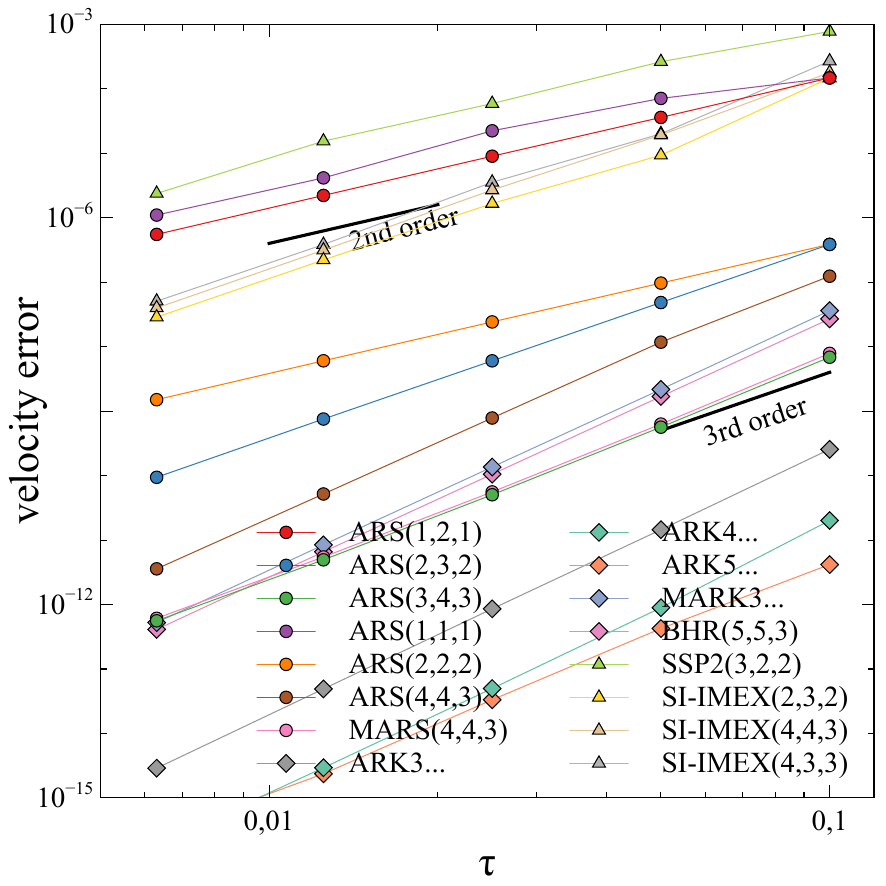}
\includegraphics[width=0.3\textwidth]{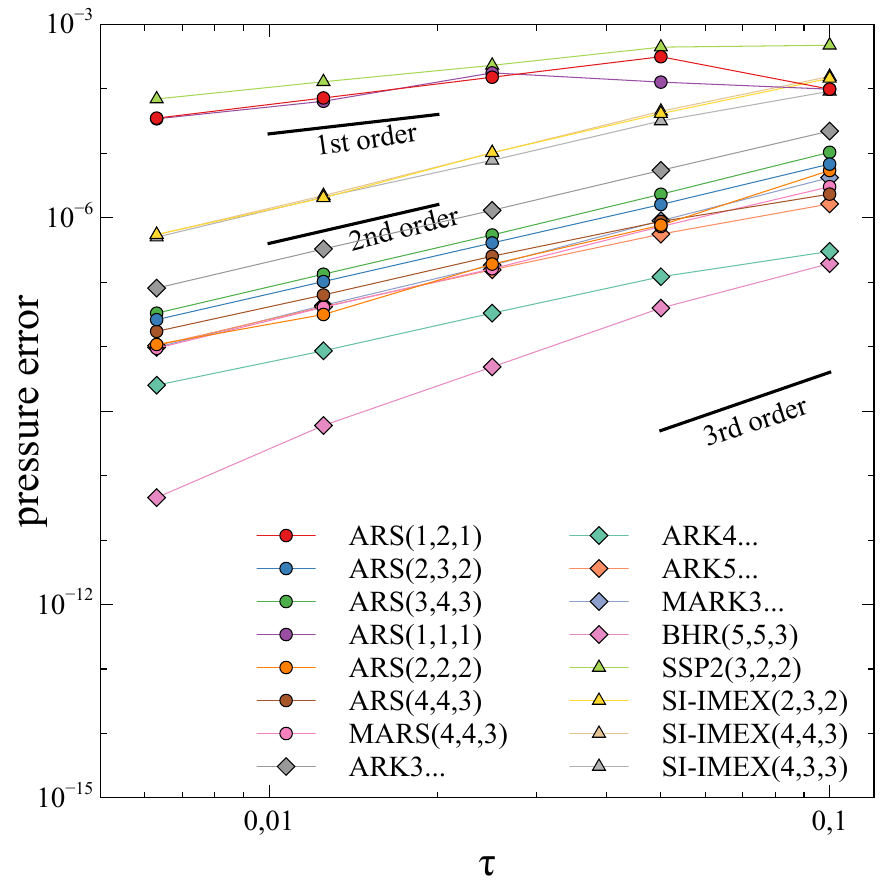}
\includegraphics[width=0.3\textwidth]{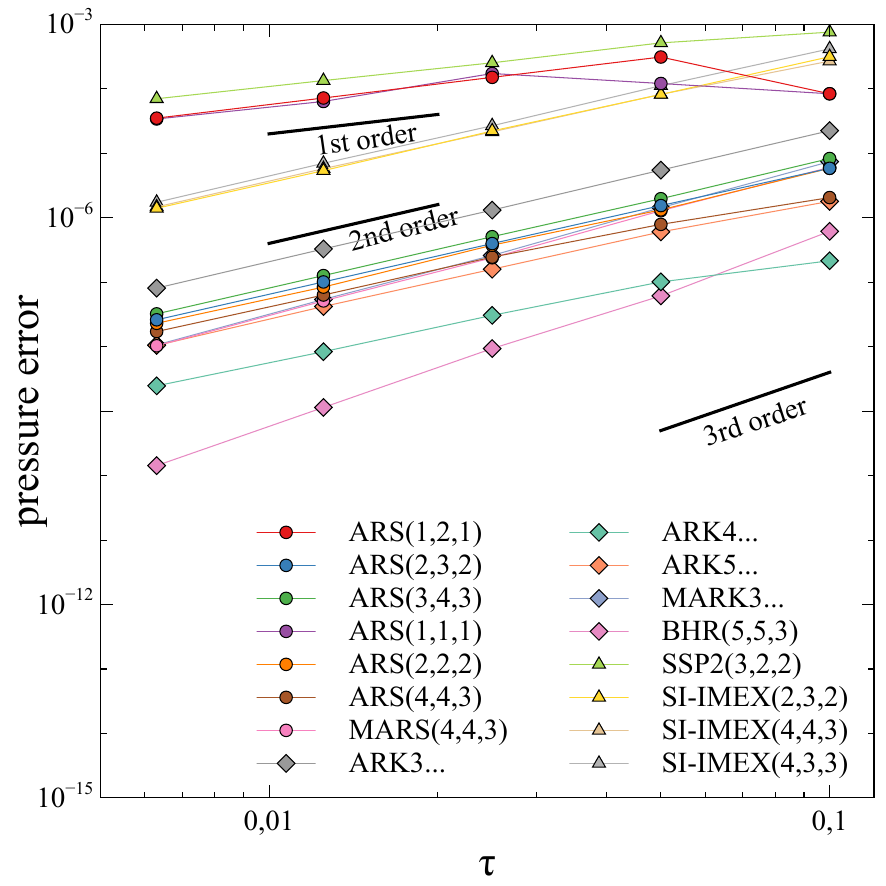}
\includegraphics[width=0.3\textwidth]{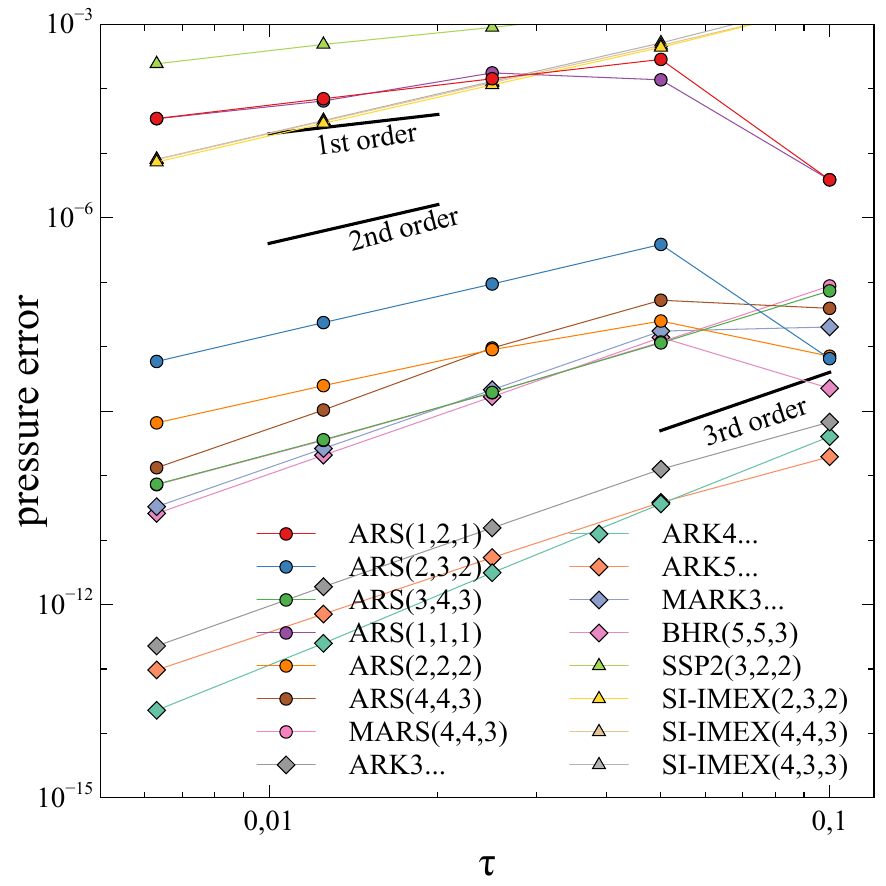}
\caption{Solution error for the test with a manufactured solution. $r_{\sigma} = 1$. Left to right: explicit SRK schemes, IMEX SRK schemes, implicit SRK schemes. Top: $e_u$; bottom: $e_p$}\label{fig:results_ctest2}
\end{figure}

The results for the temporal convergence are presented in Fig.~\ref{fig:results_ctest1} for the case $r_{\sigma} = 0$ and in Fig.~\ref{fig:results_ctest2} for the case $r_{\sigma} = 1$. They reveal several effects.
\begin{itemize}
\item Methods of type A of the form \eqref{eq_IMEX_1}--\eqref{eq_IMEX_3} show a poor accuracy. Although the numerical convergence rate is close to 2 (and not to 1 as we expected basing on \cite{Boscarino2007}), the absolute value of the solution error is big. This motivates to use a different IMEX formulation for this method, see Section~\ref{sect:IMEXv2}.
\item Most of the methods show the second order convergence for pressure. This is known as the order reduction. The only method showing the third order convergence is BHR(5,5,3) \cite{Boscarino2009}. This matches the results from \cite{Colomes2015}.
\item In our experiments, implicit SRK schemes usually show better accuracy than explicit and IMEX SRK schemes. The reason is that the term $\dot{H}(t)$ naturally appears at the implicit stage. So the use of an implicit SRK scheme reduces the discrepancy between the internal and boundary nodes. This effect is peculiar for this model problem.
\end{itemize}

In the following sections, we do not consider methods of type A, first order methods (ARS(1,2,1) and ARS(1,1,1)), and the method ARS(2,2,2) because it is not stable for $r_{\sigma}=1$ and $\alpha>0$.

\subsection{2D vortex, viscous-dominant case}
\label{sect:2Dvisc}

Let $\Omega = (0,2\pi)^2$ with the periodic conditions, viscosity be $\nu = 0.5$. The Navier~-- Stokes equations admit the following solution (travelling Taylor~-- Green vortex):
$$
\tilde{\B{u}}(t, x, y) = \left(\begin{array}{c} \bar{u} + \sin (x - \bar{u} t) \cos (y - \bar{v} t) \\ \bar{v} -\cos (x - \bar{u} t) \sin (y - \bar{v} t) \\ 0 \end{array}\right) \exp(-2\nu t),
$$
$$
\tilde{p}(t, x, y) = \frac{1}{4}(\cos(2(x - \bar{u} t)) + \cos(2(y - \bar{v}t))) \exp(-2\nu t).
$$
We put $\bar{u} = 1$, $\bar{v} = 0$. At $t = t_{\max} = 2$, we evaluate the solution error $e_u$ and $e_p$ defined by \eqref{eq_def_eu_ep}. Since $h \ll 1$, we have $\nu/h \gg \max|\tilde{\B{u}}|$, so the viscosity dominates.

For this case, we use uniform Cartesian meshes with steps $h = 2\pi/N$, $N = 32, 64, 128$. For the gradient and Laplace operators, we use the 6-th order approximations defined in Section~\ref{sect:FD}. For the convection, we use 6-th order central difference scheme based on the skew-symmetric splitting, and for the viscous term, we use the 6-th order approximation on the same stencil. The spatial approximation error is very small, so the time integration error dominates.

For explicit SRK schemes, we put $\tau = (2h^{-1} + 6\nu h^{-2})^{-1}$, which may be considered as the use of CFL $=2$ (the biggest eigenvalue of the discrete $-\nu \triangle$  operator is $12\nu h^{-2}$). For IMEX and implicit SRK schemes, we drop the timestep restriction given by viscosity, and put $\tau = (2h^{-1})^{-1} = h/2$, which corresponds to CFL $=1$.

The results with $r_{\sigma} = 0$ and $r_{\sigma} = 1$ are visually identical. They are shown in Fig.~\ref{fig:TGV2DVisc}. Horizontal axis corresponds to the mesh step $h$.

\begin{figure}
\centering
\includegraphics[width=0.3\textwidth]{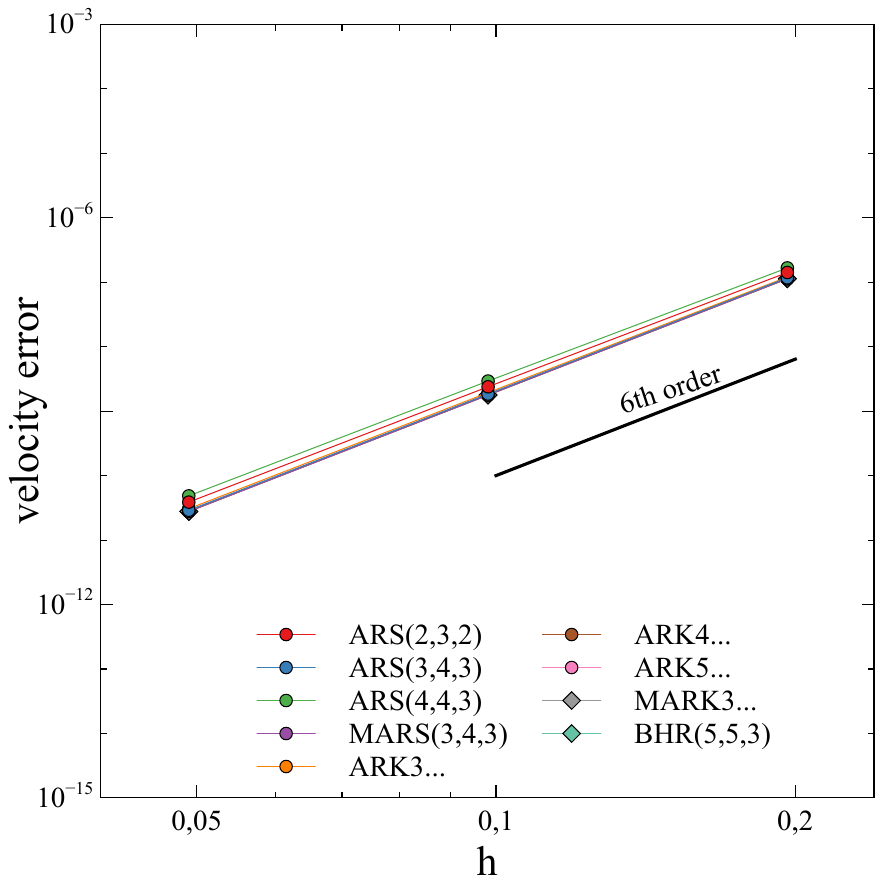}
\includegraphics[width=0.3\textwidth]{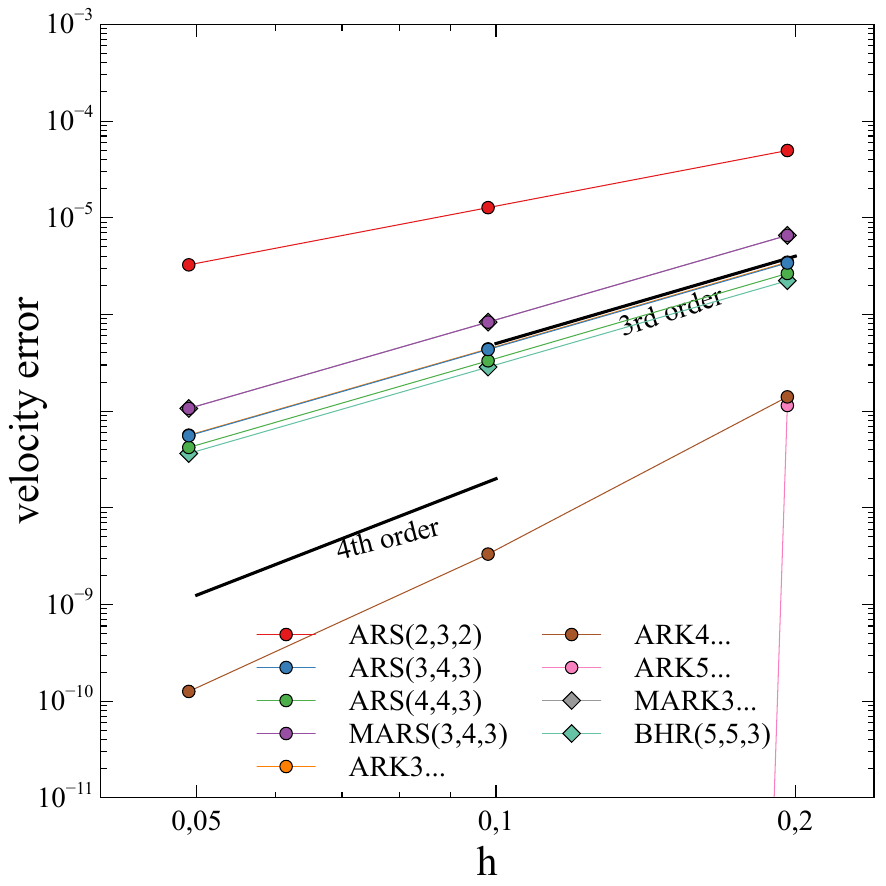}
\includegraphics[width=0.3\textwidth]{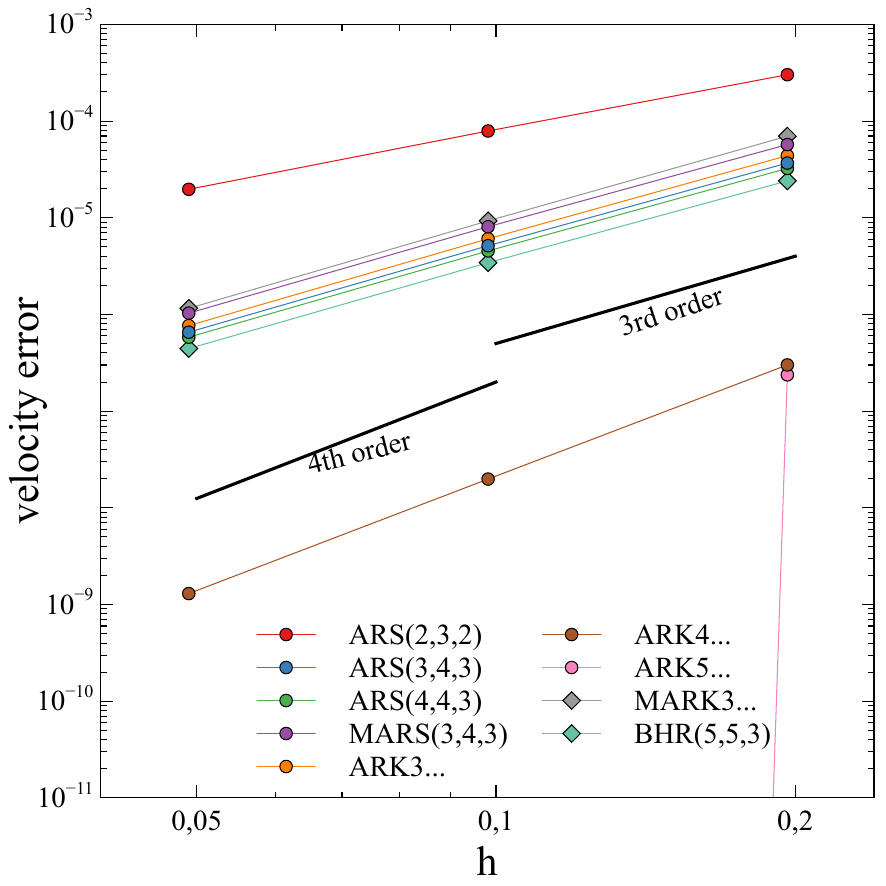}
\includegraphics[width=0.3\textwidth]{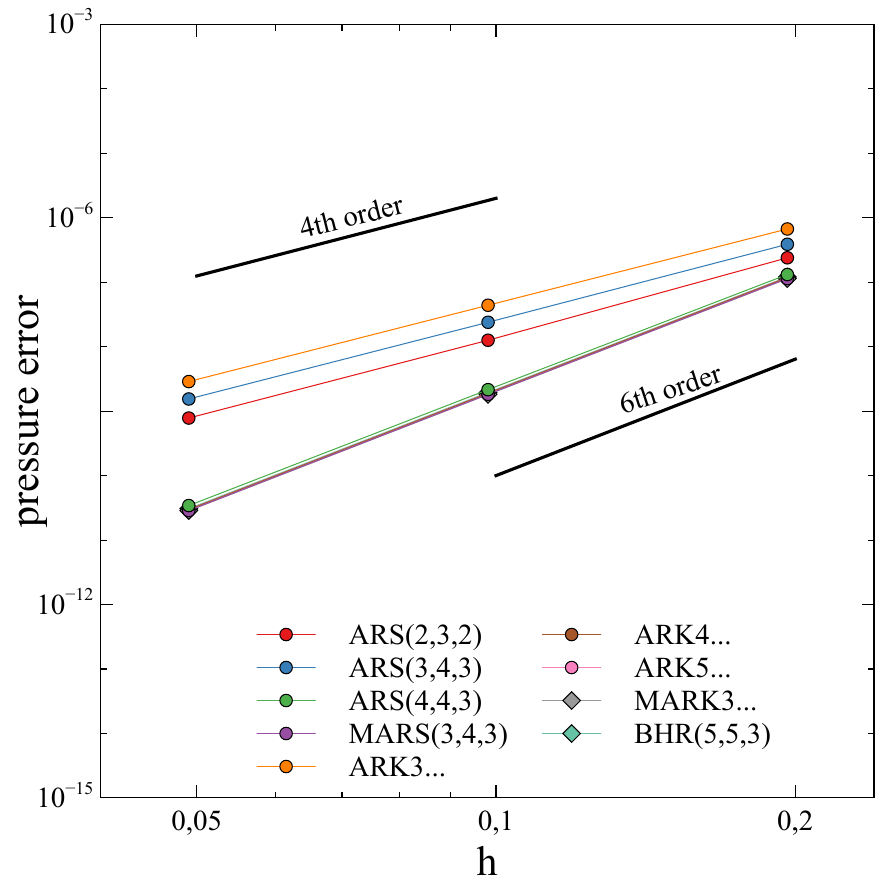}
\includegraphics[width=0.3\textwidth]{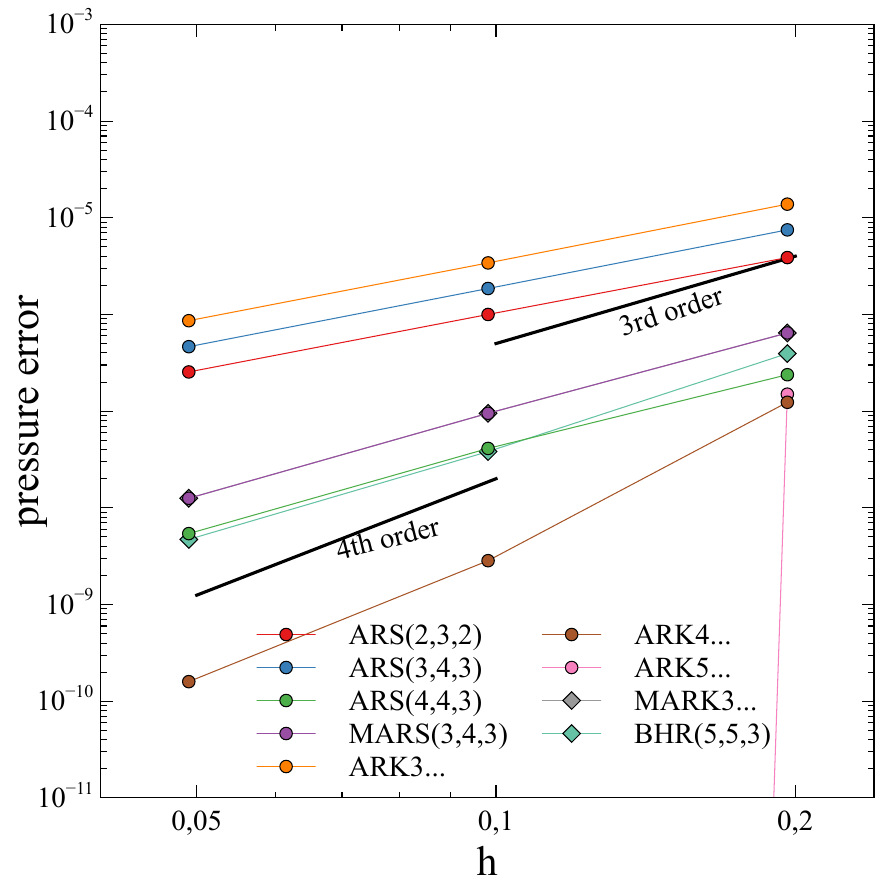}
\includegraphics[width=0.3\textwidth]{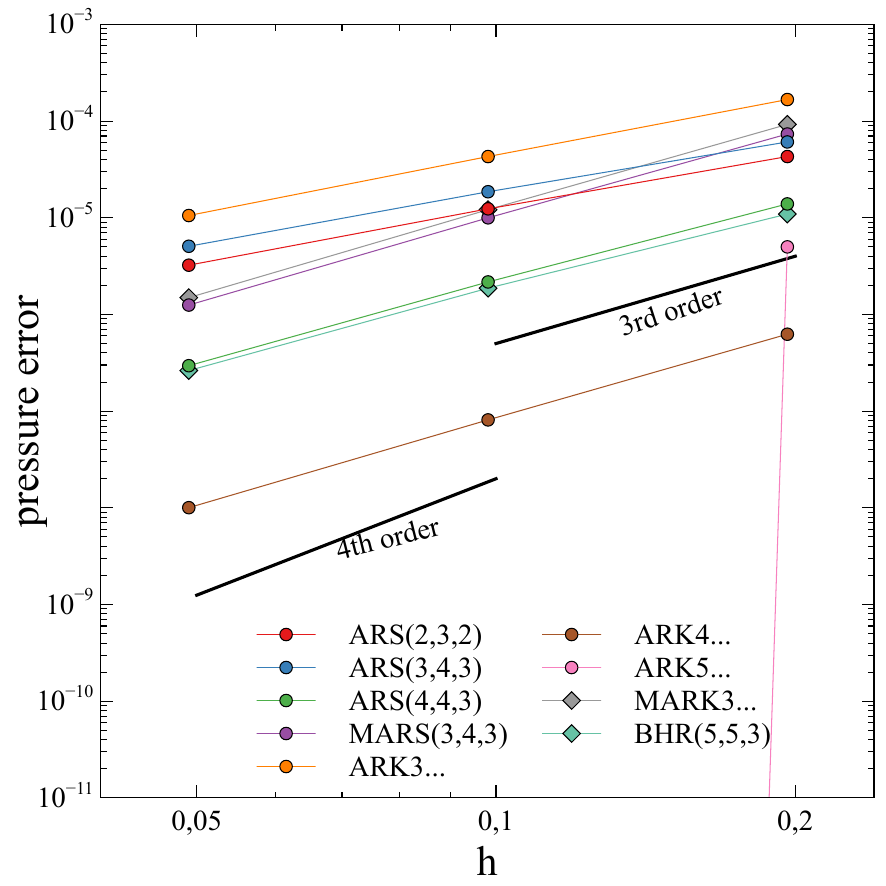}
\caption{Solution error for the Taylor~-- Green vortex with $\nu = 0.5$. Left to right: explicit SRK schemes, IMEX SRK schemes, implicit SRK schemes. Top: $e_u$; bottom: $e_p$}\label{fig:TGV2DVisc}
\end{figure}

For explicit schemes, we see the slopes corresponding to the 4th and 6th orders. The reason is $\tau \sim h^2$, so the orders in time are actually 2 and 3. For the velocity, time integration error is comparable to the spatial error. For the pressure, the third order convergence with nearly the same values is shown by ARS(4,4,3), MARS(3,4,3), ARK4(3)6L[2]SA, ARK5(4)8L[2]SA, MARK3(2)4L[2]SA, and BHR(5,5,3).

The situation is different for IMEX and in implicit SRK schemes. Here the timestep is much bigger, $\tau = h/2$ instead of $\tau \sim h^2$. Most of the schemes show second or third order depending on the variable, specific scheme, and whether an IMEX or an implicit method is used. The scheme ARK4(3)6L[2]SA shows a drastic advantage over other methods. The computations with ARK5(4)8L[2]SA blow up. Although this scheme has a more restrictive condition (see Table~\ref{table:IMEX0}), a reasonable reduction of the timestep size does not improve the situation, so the exact reason of its failure is not clear.


In our computations, we used a FFT-based solver for pressure but an AMG solver for velocity. We noticed that in the IMEX case, some methods (namely, MARK3(2)4L[2]SA, ARK4(3)6L[2]SA, especially for $r_{\sigma} = 1$) are more sensitive than other methods to the tolerance passed to the AMG solver. This might be a drawback of using a well-tuned time integration method. This effect is still to be studied.

\subsection{3D Taylor~-- Green vortex, inviscid}
\label{sect:TGV_inviscid}

Now we assess how well the zero-divergence condition is preserved as well as the leak of the kinetic energy do to the time integration error. In doing so, we follow \cite{Lehmkuhl2019} and consider an inviscid flow in $\Omega = (0, 2\pi)^3$ with the periodic boundary conditions. The initial data are prescribed by the 3D Taylor~-- Green vortex:
\begin{equation}
\begin{gathered}
\tilde{\B{u}}(t, x, y, z) = \left(\begin{array}{c} \cos x \sin y \sin z \\ - \sin x \cos y \sin z \\ 0 \end{array}\right),
\\
\tilde{p}(t,x,y,z) = (\cos 2x + \cos 2y)(2 + \cos 2z) / 16.
\end{gathered}
\label{eq_TGV}
\end{equation}

We restrict our analysis to the schemes ARS(3,4,3) and ARK4(3)6L[2]SA. We put $\tau = h$ for $r_{\sigma} = 0$ and $\tau = 0.7h$ for $r_{\sigma} = 1$. The spatial discretization is the same as in Section~\ref{sect:2Dvisc}.

Define the norm of the velocity divergence as $\|R_0\|^2 = |\mathcal{N}|^{-1}\sum_j (R_0)_j^2$ with $R_0 = L^{-1}D\B{u}$. The values of $\|R_0(t)\|$, $0 \le  t \le 6$, are shown in Fig.~\ref{fig:TGV3DNoVisc} (top). The continuity equation is satisfied with the accuracy at most $10^{-16}$ (in the same norm) after each timestep (which verifies Proposition~\ref{th:continuity}), so we do not show it on the graph. For small $t$, the velocity divergence is very small and quickly decreases as \mbox{$h \rightarrow 0$}, which corresponds to the 6-th order discretization (for a smooth enough pressure field $p$, $\|Sp\| = O(h^6)$). As the small-scale pulsations arise, the pressure stabilization terms become more significant.

\begin{figure}
\centering
\includegraphics[width=0.3\textwidth]{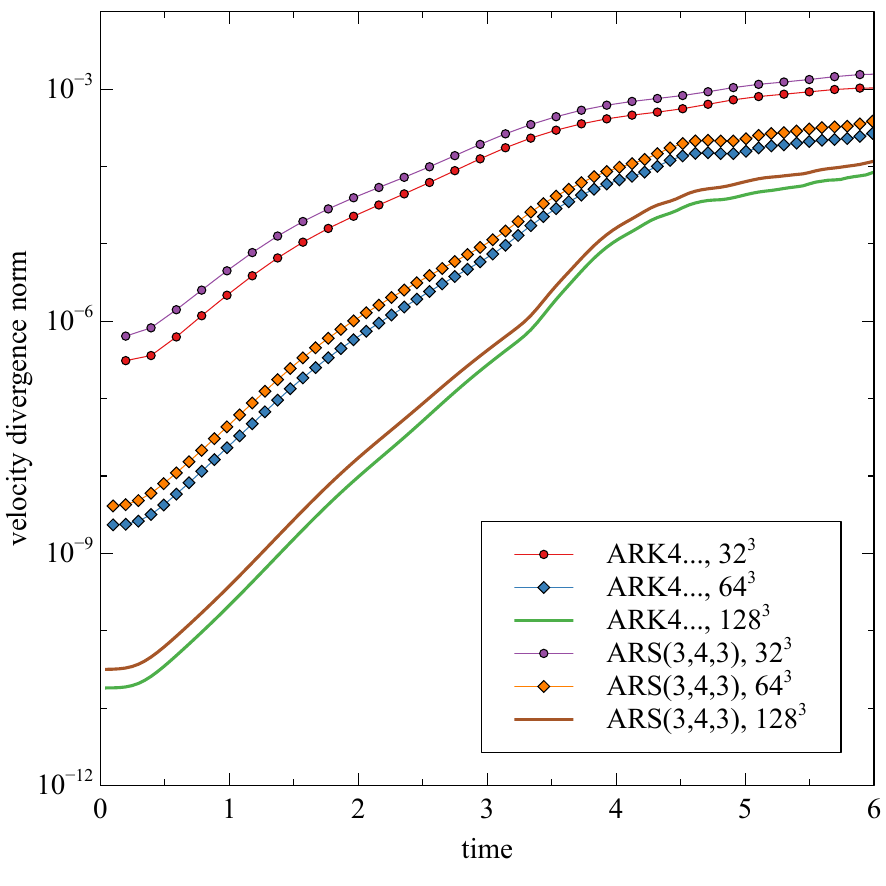}
\includegraphics[width=0.3\textwidth]{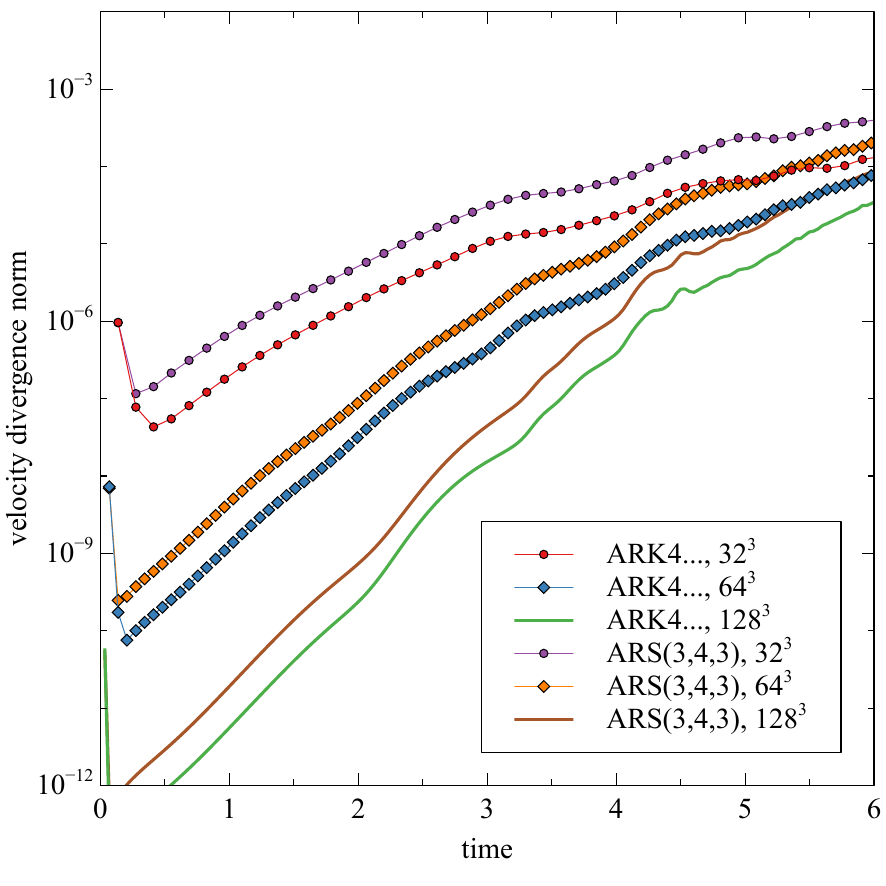}
~\\~\\
\includegraphics[width=0.35\textwidth]{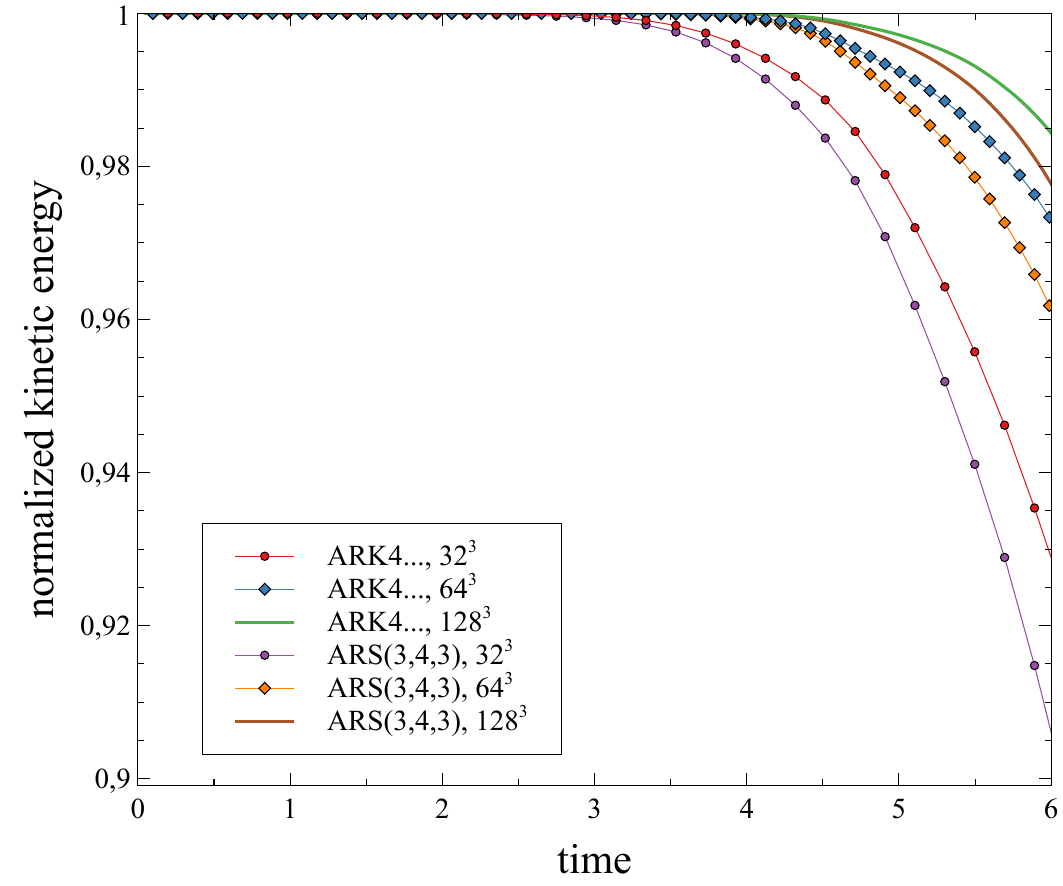}
\includegraphics[width=0.35\textwidth]{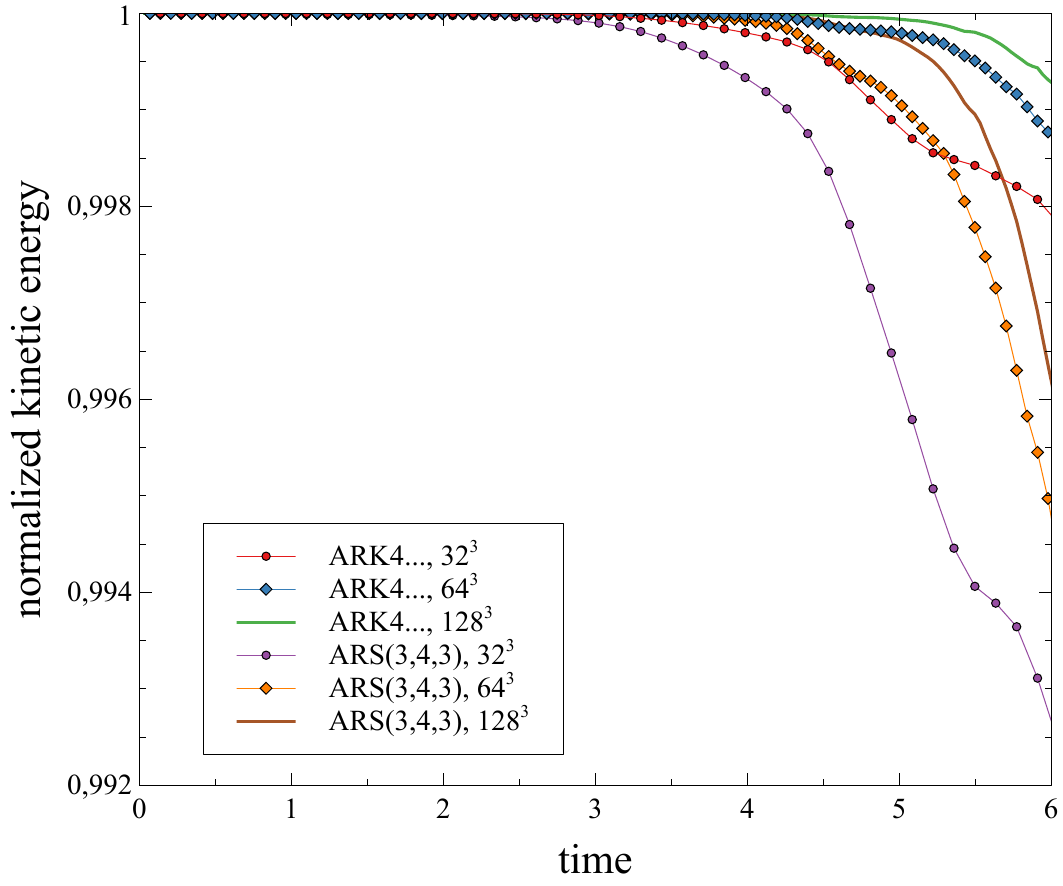}
\caption{Velocity divergence norm (top) and kinetic energy (bottom) for the inviscid Taylor~-- Green vortex. Left: $r_{\sigma} = 0$. Right: $r_{\sigma} = 1$}\label{fig:TGV3DNoVisc}
~\\~\\
\includegraphics[width=0.5\textwidth]{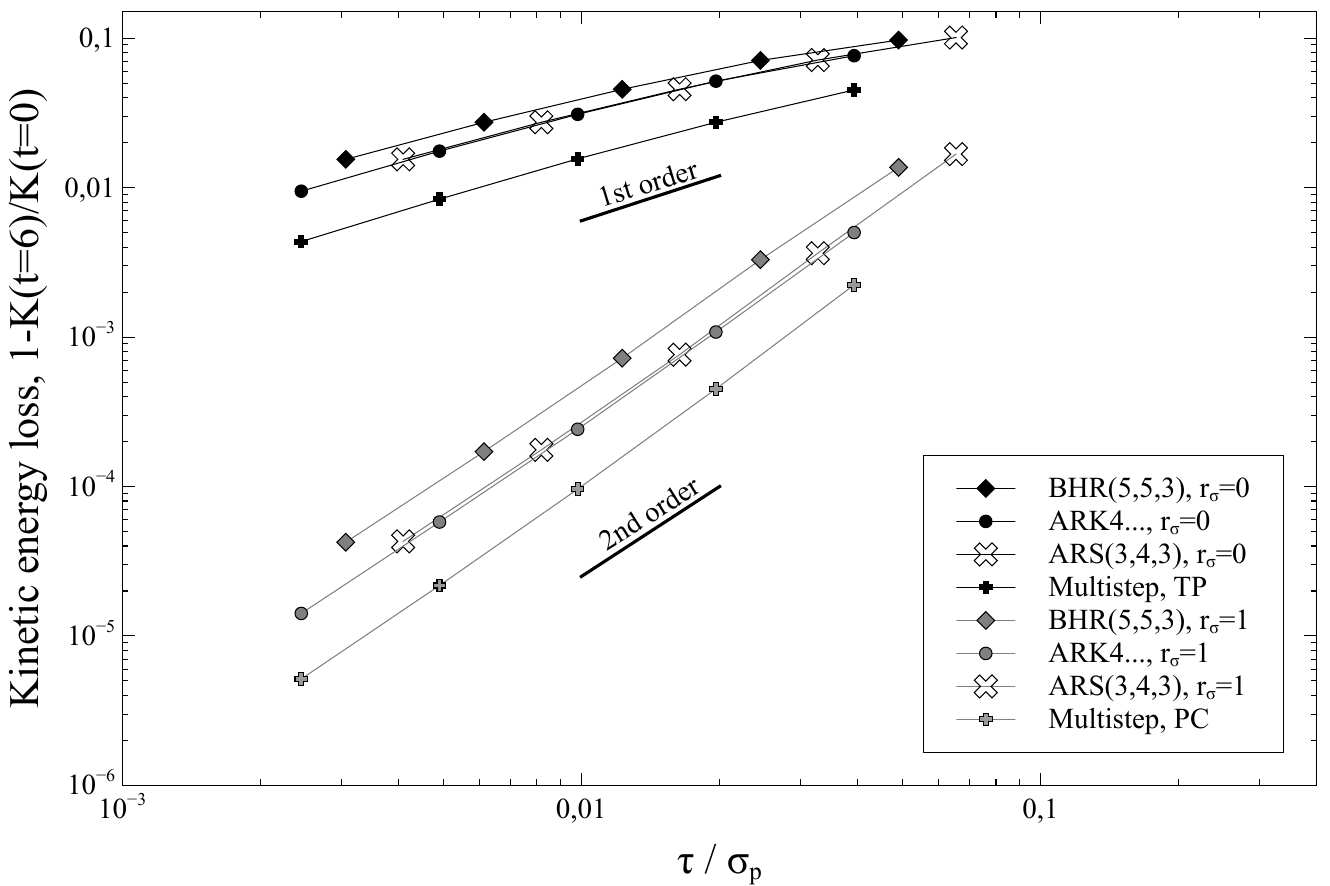}
\caption{Kinetic energy loss for the inviscid Taylor~-- Green vortex depending on the timestep. Mesh $32^3$}\label{fig:TGV3DNoVisc2}
\end{figure}

The dependency of the kinetic energy in time is shown in Fig.~\ref{fig:TGV3DNoVisc} (bottom). The kinetic energy is normalized by its value after two timesteps. For all schemes, the kinetic energy balance is much more accurate if $r_{\sigma} = 1$ (note the different scales on the left and right of Fig.~\ref{fig:TGV3DNoVisc}). This is an expected behavior, see Section~\ref{sect:stabtypes}. If $r_{\sigma} = 1$, then the kinetic energy loss is at most 1\% for each scheme even on the $32^3$ mesh. To compare, the results from \cite{Lehmkuhl2019} show the 10\% energy loss on the $64^3$ mesh.

Additionally, we fix the mesh $32^3$ and plot the kinetic energy loss, \mbox{$e_k = 1 - K(6)/K(0)$} where $K(t)$ is the kinetic energy, as a function of $\tau/\sigma_{p}$ in Fig.~\ref{fig:TGV3DNoVisc2}. Recall that $\sigma_p$ is the number of pressure solver calls per timestep. The parameter $\tau/\sigma_{p}$ instead of $\tau$ is used for the fair comparison. We see that the kinetic energy loss is approximately $O(\tau)$ (numerical convergence orders are 0.82 and 0.89) for $r_{\sigma} = 0$ and $O(\tau^2)$ for $r_{\sigma} = 1$. This result is in accordance with the dependency on $\tau$ of the constants in front of the stabilization terms. The comparison with the 3-rd order multistep method confirms that the case $r_{\sigma}=0$ corresponds to the total pressure (or non-incremental pressure-correction) method, and $r_{\sigma}=1$ corresponds to the standard (incremental) pressure-correction method.

\subsection{2D Taylor~-- Green vortex, inviscid}
\label{sect:TGV2D_inviscid}

Let us repeat the previous case in the the 2D setup. We pick the mesh $32^2$, put $\nu = 0$, prescribe the initial data by the 2D Taylor~-- Green vortex (see Section~\ref{sect:2Dvisc}). The difference to the 3D case is that the flow is laminar.  We show $|e_k| = |1 - K(6)/K(0)|$ as a function of $\tau$ in Fig.~\ref{fig:TGV2DNoVisc}.
\begin{figure}
\centering
\includegraphics[width=0.4\textwidth]{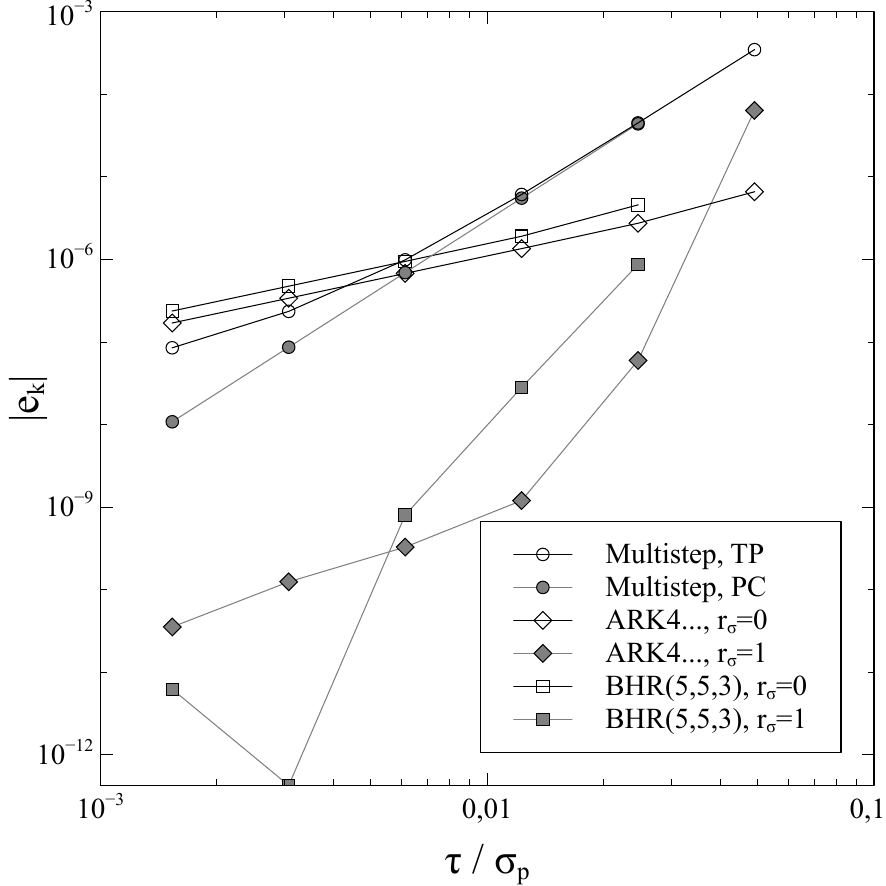}
\caption{Kinetic energy loss for the 2D inviscid Taylor~-- Green vortex depending on the timestep. Mesh $32^2$}\label{fig:TGV2DNoVisc}
\end{figure}

Here the SRK schemes show much more accurate kinetic energy balance than the 3-rd order multistep method. The reason the slopes change with $\tau$ is that for large $\tau$ the time integration error dominates, and for small $\tau$ the kinetic energy continue to decrease while the integral velocity error tends to a nonzero constant defined by the spatial error.

\subsection{3D Taylor~-- Green vortex, Re = 800}

Now consider the Taylor~-- Green vortex  with $\mathrm{Re} = 800$ in the 3D setup. Let $\Omega = (0, 2\pi)^3$ and $\nu = \mathrm{Re}^{-1}$. The initial data are given by \eqref{eq_TGV}. We take $h = 2\pi/128$ and $\tau = 0.8 h$. The values of $dK/dt$, $0 < t < 10$, are compared to the results from \cite{Brachet1984}, where the mesh with $h = 2\pi/256$ was used.

The results for BHR(5,5,3), ARS(2,3,2), ARS(3,4,3), ARK4(3)6L[2]SA are visually identical. They are also match the reference data for $t \le 8.5$, as this is shown in Fig.~\ref{fig:tgv3d}. A choice of $r_{\sigma}$ results in a 0.5\% difference in the final value of $dK/dt$ and in a 0.1\% in the final value of $K$. The use of a second-order finite-element discretization of viscous terms results in a visually identical data for $dK/dt$ in $t \in [0, 8.5]$ and in a 0.5\% difference as $t = 10$. All the results were obtained with the 6-th order FD spatial discretization. We also conducted the computations with the P1-Galerkin method (without the mass lumping) on the same mesh, and results match as well.

\begin{figure}[t]
\centering
\includegraphics[width=0.4\textwidth]{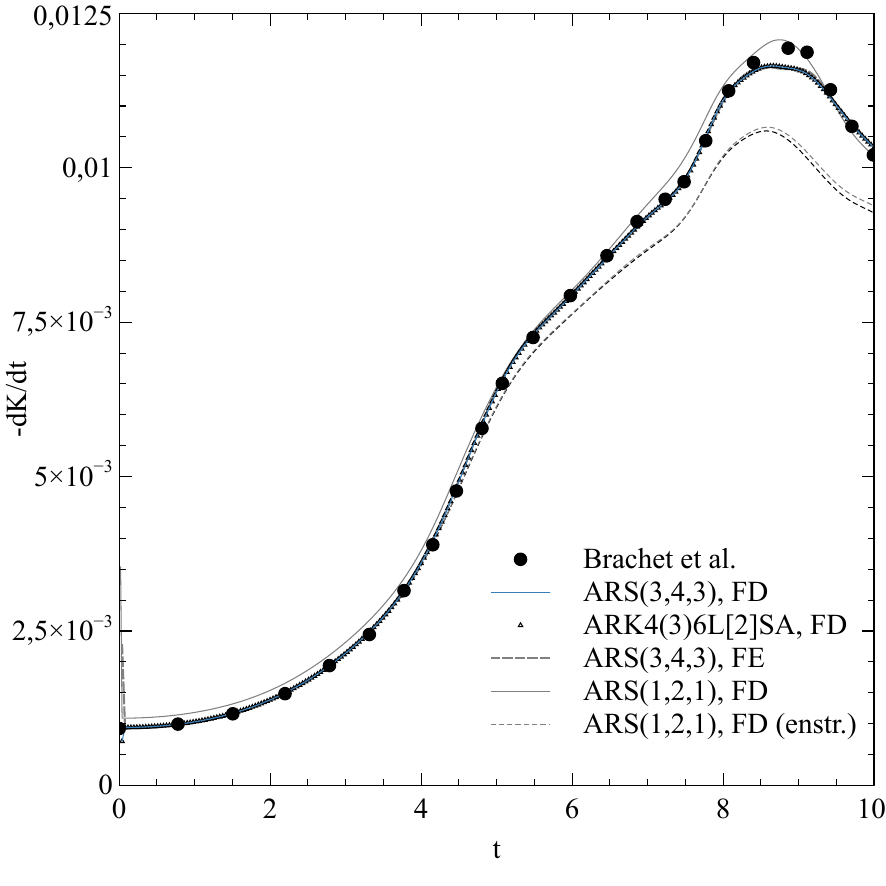}
\caption{$-dK/dt$ for the 3D Taylor~-- Green vortex with Re = 800. Mesh $128^3$}\label{fig:tgv3d}
\end{figure}

To compare, we consider the scheme ARS(1,2,1) with $\tau = 0.4 h$. The results for $-dK/dt$, which go above the true value, differ from the results for the enstrophy, wihch go below. Two lines correspond to $r_{\sigma} = 0$ and $r_{\sigma} = 1$.

\subsection{Channel flow at $\mathrm{Re}_{\tau} = 180$}
\label{sect:turbchannel}

Now we demonstrate our scheme on the turbulent channel flow (TCF) test with $\mathrm{Re}_{\tau} = 180$. Here $\mathrm{Re}_{\tau}$ is the Reynolds number based on the friction velocity and the half-width of the channel. We take 
$\Omega = [0, 4\pi] \times [0, 2] \times [0, 4\pi/3]$, $\nu = 1/180$. Then $u_{\tau}$ should be equal to 1 if the scheme conserves the momentum and the time averaging interval is sufficiently long. 

The computational mesh is Cartesian. The 1D meshes in streamwise and spanwise directions are uniform, with 256 and 128 nodes, correspondingly. The mesh in the normal direction has $y_+ = 0.5$ for the first node, then follows the progression with $q = 1.05$ until the maximal step $(\Delta y)_+ = 2.2$ is reached, 197 nodes in total.

We use the basic finite-difference scheme (see Section~\ref{sect:FD2}) as well as the P1-Galerkin method (see Section~\ref{sect:Gal}). Pressure solver is based on the fast Fourier transform in periodic directions. We take the convective terms, streamwise, and spanwise stresses explicitly; the normal stresses are taken implicitly. This approach was initially designed for computations with a turbulence model (which do not allow FFT-based solvers for the implicit velocity step because of the variable viscosity). For the time integration, we use ARS(2,3,2) with $r_{\sigma} = 1$ and the timestep $\tau = 8 \cdot 10^{-4}$.

The mean streamwise velocity $\bar{u}$ and the root-mean-square pulsations $\overline{(u - \bar{u})^2}$, $\overline{v^2}$, $\overline{w^2}$ are shown in Fig.~\ref{fig:turbch}, together with the reference data from \cite{Vreman2014}. The time averaging interval was of length $t_{\max} = 10$ and $t_{\max} = 20$ (for the finite-difference and Galerkin methods, correspondingly). We do not average over the symmetry plane of the channel, so two lines corresponds to the two halfs of the channel. A minor discrepancy between these lines mean that the time averaging intervals is slightly insufficient.

\begin{figure}[t]
\centering
\includegraphics[width=0.9\textwidth]{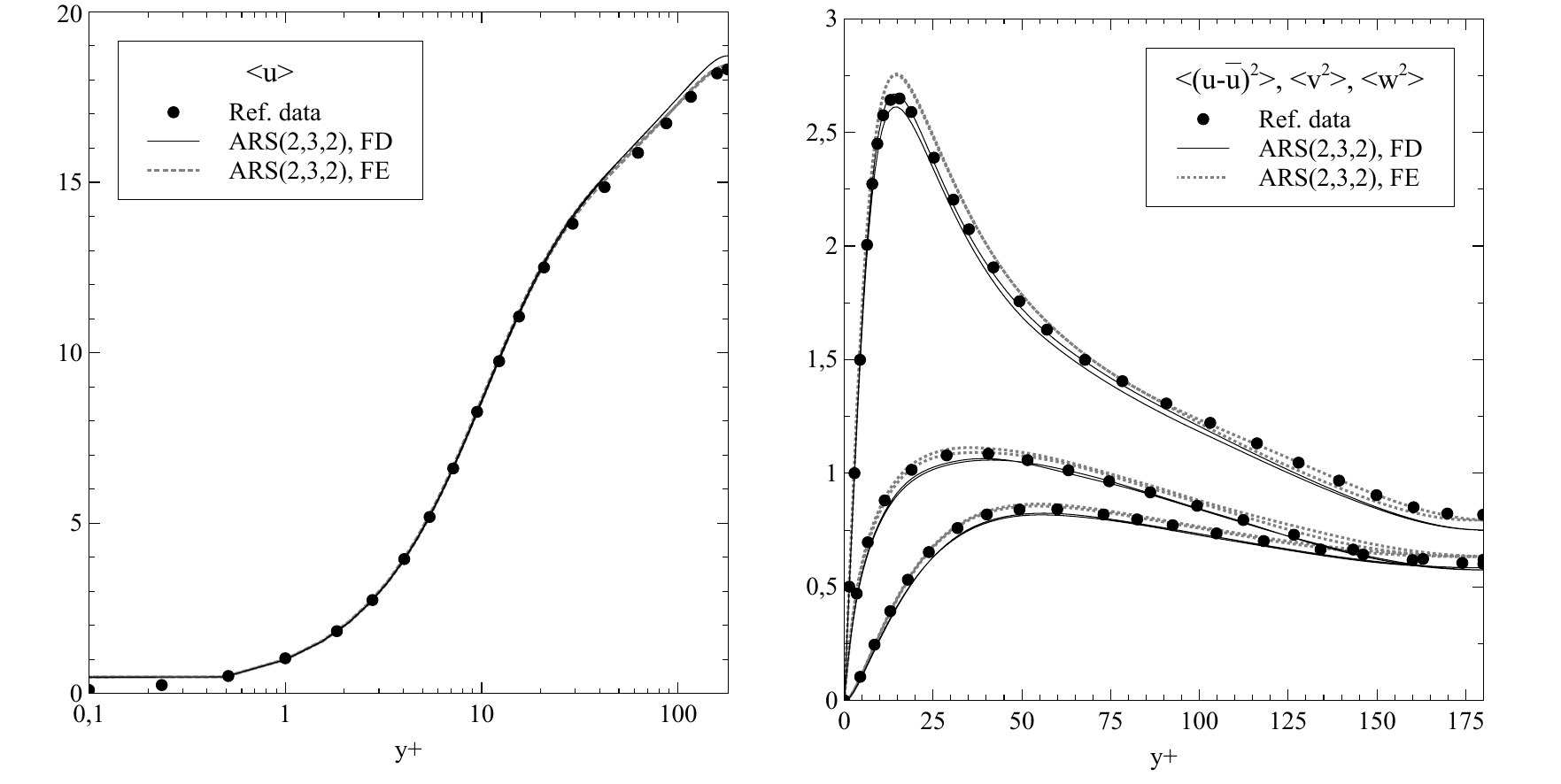}
\caption{Mean velocity (left) and root-mean-square pulsations (right) in the turbulent channel with $\mathrm{Re}_{\tau} = 180$}\label{fig:turbch}
\end{figure}

The numerical results generally comply with the reference solution, considering the fact that the turbulence structures in the streamwise and spanwise directions are slightly underresolved ($(\Delta x)_+ \approx 8.8$, $(\Delta z)_+ \approx 5.9$). The Galerkin method yields slightly better results (except for the RMS streamwise pulsations, the reason being unknown). This is expectable, because it is based on a fourth order spatial discretization on uniform meshes, compared to the second order in the finite-difference case.

\subsection{2D vortex, unstructured mesh}

Now we consider high-order Galerkin discretizations with equal-order piecewise-polynomial interpolation on unstructured triangular meshes. As their implementation, we use the MFEM library \cite{mfem}. MFEM has its own Navier~-- Stokes solver based on
the third-order multistep time integration \cite{Franco2020}, so we compare our results with it. As a multistep solver, it has a drawback that the first two timesteps use lower-order time discretizations. For most of practical applications, the accuracy of the first steps is not important. So for a fair comparison, we present two values: 1) with lower-order timestepping for the first two steps (labeled 1xID in figures below) and 2) with replacing the first two timesteps by the use of the exact solution. Also the results by the multistep method with a timestep $\tau$ are compared to the SRK methods with timestep $\sigma_{p} \tau$, because one stage of an SRK scheme (except for $j=1$ for methods of type CK) has approximately the same the computational costs as one timestep of the multistep method.

We begin with the viscous-dominant case described in Section~\ref{sect:2Dvisc}. The mesh is triangular with average edge length $2\pi/8$. The polynomial order is $p = 4$. We take the convection term explicitly and the viscous term implicitly, as this is done in the MFEM Navier~-- Stokes solver. The results are collected in Fig.~\ref{fig:tgv2dviscunstr_r0} and Fig.~\ref{fig:tgv2dviscunstr_r1}. The time integration error in this case dominates over the spatial error ($\approx 10^{-5}$ for the velocity), so we can see the convergence. 

\begin{figure}
\centering
\includegraphics[width=0.45\textwidth]{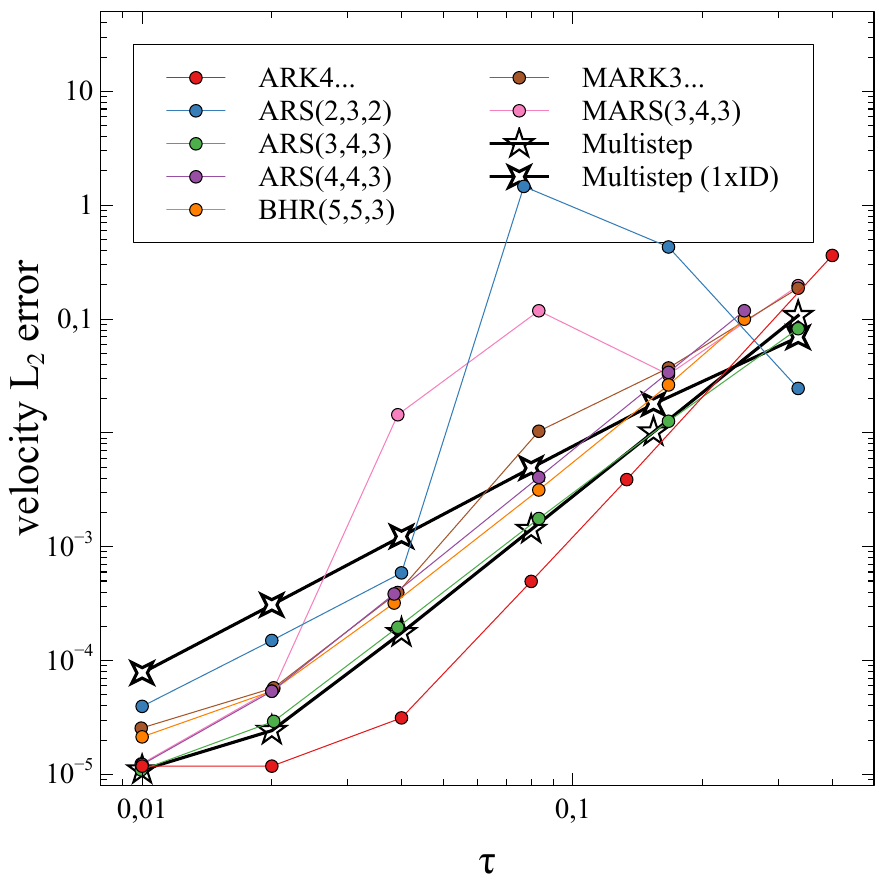}
\includegraphics[width=0.45\textwidth]{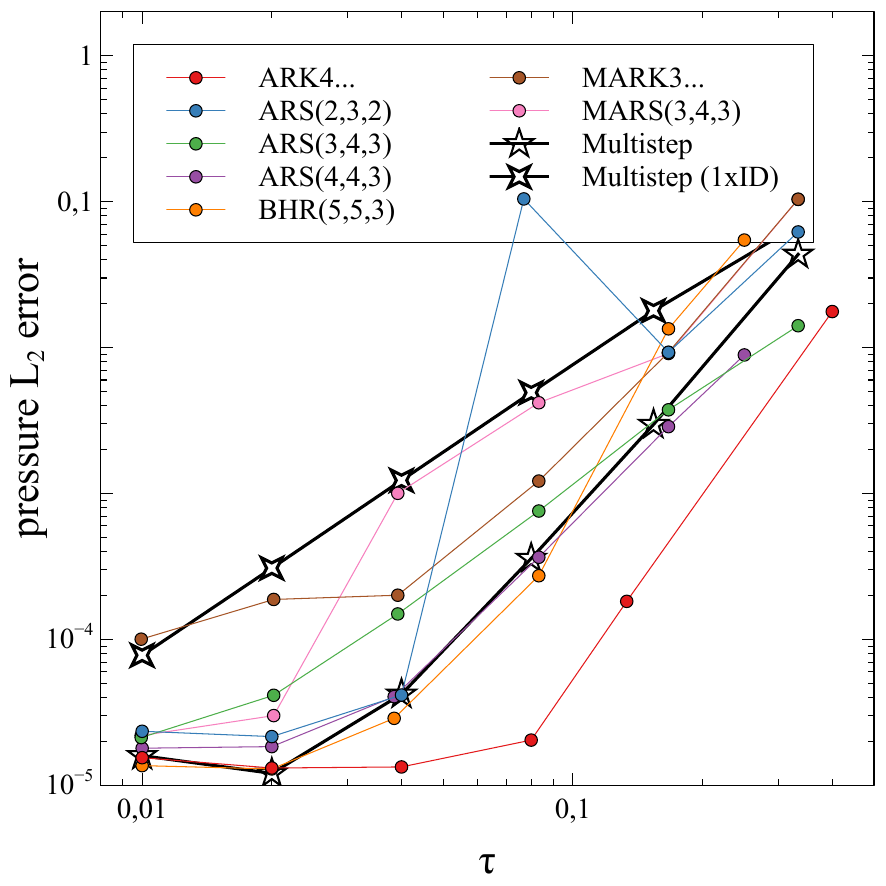}
\caption{Velocity and pressure $L_2$ error for the Taylor~-- Green vortex with $\nu = 0.5$ on an unstructured mesh. $r_{\sigma} = 0$}
\label{fig:tgv2dviscunstr_r0}
~\\~\\
\includegraphics[width=0.45\textwidth]{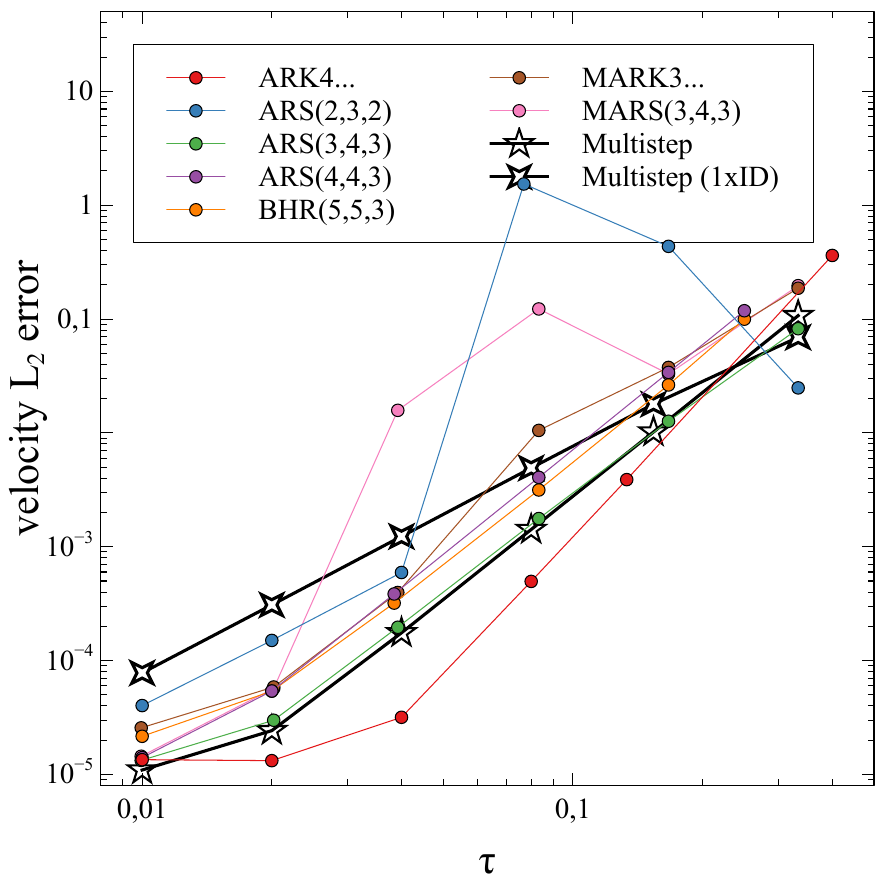}
\includegraphics[width=0.45\textwidth]{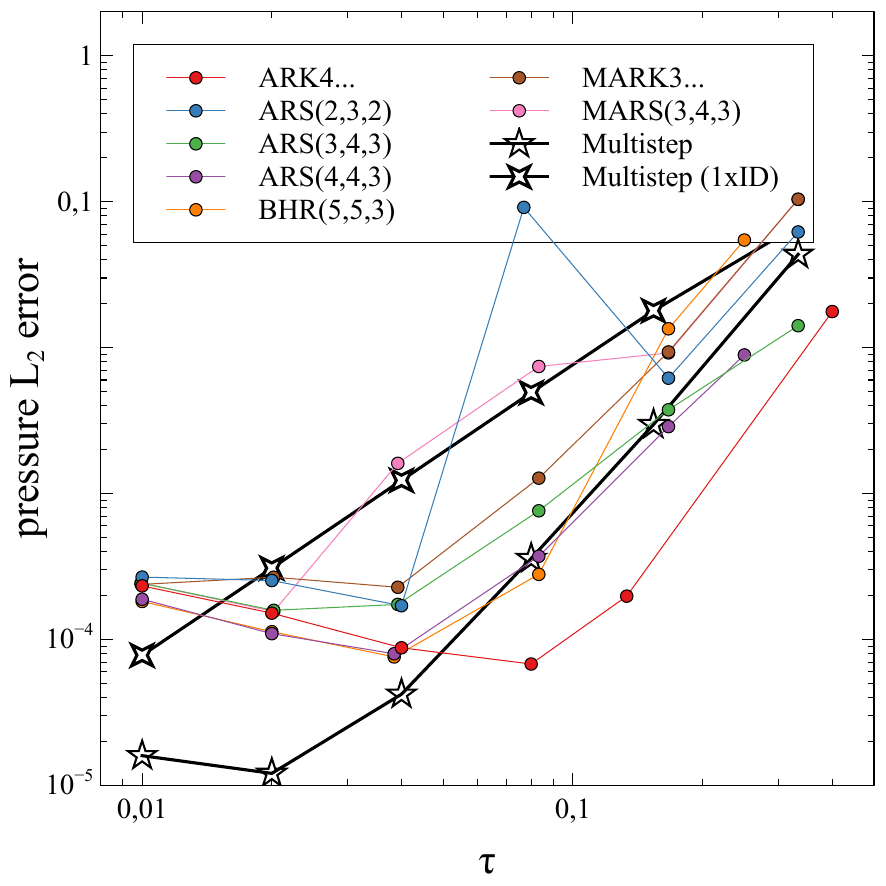}
\caption{Velocity and pressure $L_2$ error for the Taylor~-- Green vortex with $\nu = 0.5$ on an unstructured mesh. $r_{\sigma} = 1$}\label{fig:tgv2dviscunstr_r1}
\end{figure}


The pressure error tends to different values as the timestep goes to zero. The results of the multistep method has the same limit as of SRK schemes with $r_{\sigma} = 0$. This may be explained by the fact that the pressure equation of the first-order multistep method has the same form as of the first-order SRK method.

Two schemes, ARS(2,3,2) and MARS(3,4,3), show instability in a range of timesteps. A possible reason is that the timestep exceeds a CFL-like condition for the explicit part. For some schemes, the domination of the implicit viscous terms mitigates the timestep restriction for the explicit term (see the discussion in \cite{Ascher1997}), but not for these schemes.

The only SRK scheme that outperforms the third order multistep method in this case is ARK4(3)6L[2]SA. The scheme ARS(3,4,3) shows the same velocity error but bigger pressure error. The scheme BHR(5,5,3) shows approximately the same pressure error (in the case $r_{\sigma} = 0$) but slightly bigger velocity error.

Now consider the zero-viscosity case. We use the same mesh, take the polynomial order $p = 6$, the maximal integration time $t_{\max} = 20$, and the timestep $\tau \approx  0.005 \cdot \sigma_{p}$ (corrected so that the number of timestep is integer).

The results for the velocity $L_2$ error, pressure $L_2$ error, and the kinetic energy loss $e_k = 1 - K(t_{\max})/K(0)$ are collected in Table~\ref{table:nonvisc}.

\begin{table}[t]
\caption{\label{table:nonvisc}Numerical results for the travelling vortex using the Galerkin method}
\begin{center}
\begin{tabular}{|c|c|c|c|c|}
\hline 
scheme & $\tau/\sigma_{p}$ & vel. err. & pres. err. & $e_k$ \\
\hline
\multirow{3}{*}{Multistep} & 0.0025 & $1.93 \cdot 10^{-5}$ & $8.92 \cdot 10^{-6}$ & $1.56 \cdot 10^{-7}$ \\
& 0.005 & $2.25 \cdot 10^{-5}$ & $1.23 \cdot 10^{-5}$ & $1.25 \cdot 10^{-6}$ \\
 & 0.01 & $7.11 \cdot 10^{-5}$ & $5.11 \cdot 10^{-5}$ & $9.99 \cdot 10^{-6}$ \\
 & 0.015 & +inf & +inf & +inf \\
\hline
\multirow{4}{*}{ARK4(3)6L[2]SA; $r_{\sigma}=0$} & 0.0025 & $2.19 \cdot 10^{-5}$ & $1.03 \cdot 10^{-5}$ & $5.56 \cdot 10^{-10}$ \\
& 0.005 & $2.65 \cdot 10^{-5}$ & $1.29 \cdot 10^{-5}$ & $1.18 \cdot 10^{-9}$ \\
& 0.01 & $3.50 \cdot 10^{-5}$ & $1.74 \cdot 10^{-5}$ & $4.40 \cdot 10^{-9}$ \\
& 0.015 & +inf & +inf & +inf \\
\hline
\multirow{2}{*}{ARK4(3)6L[2]SA; $r_{\sigma}=1$} & 0.005 & $2.67 \cdot 10^{-5}$ & $1.21 \cdot 10^{-5}$ & $7.11 \cdot 10^{-11}$ \\
& 0.01 & +inf & +inf & +inf \\
\hline
\multirow{3}{*}{ARS(3,4,3); $r_{\sigma}=0$} & 0.005 & $3.10 \cdot 10^{-5}$ & $1.09 \cdot 10^{-4}$ & $1.70 \cdot 10^{-6}$ \\
& 0.01 & $1.27 \cdot 10^{-4}$ & $4.74 \cdot 10^{-4}$ & $1.35 \cdot 10^{-5}$ \\
& 0.015 & +inf & +inf & +inf \\
\hline
\multirow{3}{*}{ARS(3,4,3); $r_{\sigma}=1$} & 0.005 & $3.15 \cdot 10^{-5}$ & $1.09 \cdot 10^{-4}$ & $1.70 \cdot 10^{-6}$ \\
& 0.01 & $1.27 \cdot 10^{-4}$ & $4.74 \cdot 10^{-4}$ & $1.35 \cdot 10^{-5}$ \\
& 0.015 & +inf & +inf & +inf \\
\hline
\multirow{2}{*}{BHR(5,5,3); $r_{\sigma}=0$} & 0.005 & $3.00 \cdot 10^{-5}$ & $1.41 \cdot 10^{-5}$ & $1.74 \cdot 10^{-9}$ \\
& 0.01 & +inf & +inf & +inf \\
\hline
\multirow{2}{*}{BHR(5,5,3); $r_{\sigma}=1$} & 0.005 & $2.55 \cdot 10^{-5}$ & $9.82 \cdot 10^{-6}$ & $1.09 \cdot 10^{-9}$ \\
& 0.01 & +inf & +inf & +inf \\
\hline
\end{tabular}
\end{center}
\end{table}

For the multistep scheme and ARK4(3)6L[2]SA, reducing the timestep does not lead to better results. So we can conclude the the values $2 \cdot 10^{-5}$ and $10^{-5}$ are close to the velocity and pressure error, correspondingly, for the semidiscrete problem. Depending on which fields or quantities are considered impotant, one may assess the results differently.

\begin{itemize}
\item Velocity field. For all schemes considered here, the timestep $\tau = 0.005 \sigma_{p}$ allows to obtain a reasonable accuracy. 
\item Pressure field. The multistep scheme, ARK4(3)6L[2]SA, and BHR(5,5,3) perform well, and ARS(3,4,3) does not.
\item Integral kinetic energy. Here ARK4(3)6L[2]SA, and BHR(5,5,3) perform well.  ARS(3,4,3) and the multistep scheme yields much bigger error. This complies with the results on structured meshes in Section~\ref{sect:TGV2D_inviscid}.
\end{itemize}

\section{Another group of IMEX methods}
\label{sect:IMEXv2}

In this section, we consider SRK schemes basing on a different group of IMEX RK methods. For these SRK schemes, Proposition~\ref{th:continuity} does not hold, i. e. the exact divergence constraint is not preserved. We study how this affects the numerical results.

\subsection{General form}

IMEX RK methods in \cite{Boscarino2021, Boscarino2024} have the following form. 
Each stage $j = 1, \ldots, s$ begins with the evaluation of the explicit term:
\begin{equation}
y_{j}^{(E)} = y^{n-1} + \tau \sum\limits_{k=1}^{j-1} \hat{a}_{jk} \mathcal{H}_k,
\quad
\hat{K}_j = \xi(\hat{t}_j, y_{j}^{(E)}).
\label{eq_IMEX2_1}
\end{equation}
Here $\hat{t}_j = t^{n-1} + \hat{c}_j \tau$. Then the implicit term is defined as the solution of
\begin{equation}
y_j = y_{j,*} + \tau a_{jj} K_j, \quad K_j = \eta(t_j, y_j),
\label{eq_IMEX2_2}
\end{equation}
with $t_j = t^{n-1} + c_j \tau$ and
\begin{equation}
y_{j,*} = y^{n-1} + \tau a_{jj} \hat{K}_j + \tau \sum\limits_{k=1}^{j-1} a_{jk} \mathcal{H}_k.
\label{eq_IMEX2_3}
\end{equation}
Then put $\mathcal{H}_k = \hat{K}_j + K_j$. Finally, the value at $t^{n}$ is
\begin{equation}
y^n = y^{n-1} + \tau \sum\limits_{k=1}^s b_k \mathcal{H}_k.
\label{eq_IMEX2_4}
\end{equation}

Among the family \eqref{eq_IMEX2_1}--\eqref{eq_IMEX2_4}, we consider the methods of type A only (i.~e. assume that matrix $A = \{a_{jk}\}$ is non-degenerate). Otherwise, we would have to take to the explicit term the value $K_1 = \eta(t^{n-1}, y^{(n-1)})$, which is unbounded in the stiff limit.

\subsection{SRK schemes based on \eqref{eq_IMEX2_1}--\eqref{eq_IMEX2_4}}

An IMEX SRK method for \eqref{eq_momentun_base}, \eqref{eq_continuity_base} may be defined by \eqref{eq_IMEX2_1}--\eqref{eq_IMEX2_4} with the substitution \eqref{eq_xi_eta_for_IMEX}. An implicit stage \eqref{eq_IMEX2_2} has the same form as \eqref{eq_IMEX_1}, and it was discussed in Section~\ref{sect:stab}. So no additional work is needed to write a final form of the segregated Runge~-- Kutta scheme.


1. For each $j = 1, \ldots, s$:

1.1. Put $t_j = t^{n-1} + \tau c_j$, $\hat{t}_j = t^{n-1} + \tau \hat{c}_j$, $c_j = \sum_{k=1}^j a_{jk}$, $\hat{c}_j = \sum_{k=1}^j \hat{a}_{jk}$.

1.2. Explicit fluxes:
$$
u_{j}^{(E)} = u^{n-1} + \tau \sum\limits_{k=1}^{j-1} \hat{a}_{jk} \mathcal{H}_k^{(u)},
\quad
p_{j}^{(E)} = p^{n-1} + \tau \sum\limits_{k=1}^{j-1} \hat{a}_{jk} q_k,
$$
$$
\hat{K}_j = \B{C}(\hat{t}_j, u_{j}^{(E)}) - G p_j^{(E)}.
$$

1.3. Implicit velocity step: define
$$
u_{j,*} = u^{n-1} + \tau a_{jj} \hat{K}_j + \tau \sum\limits_{k=1}^{j-1} a_{jk} \mathcal{H}_k^{(u)}
$$
and solve
$$
u_j = u_{j,*} + \breve{\tau} \B{D}(t_j, u_j).
$$
Put $\mathcal{H}_j^{(u)} = \breve{\tau}^{-1} (u_j - u_{j,*}) + \hat{K}_j$.

1.4. Pressure update. Define
$$
p_{j,*} = p^{n-1} + \tau \sum\limits_{k=1}^{j-1} a_{jk} q_k.
$$
For $r_{\sigma}=0$, put $q_{j,*} = 0$, and for $r_{\sigma}=1$, put
$$
q_{j,*} = q^{n-1} + \tau \sum\limits_{k=1}^{j-1} a_{jk} \dot{q}_k.
$$
Put $w^{n-1} = p^{n-1}$ for $r_{\sigma}=0$ and $w^{n-1} = \breve{\tau} q^{n-1}$ for $r_{\sigma}=1$. 
Put
$$
\tilde{p}_j = p_{j,*} + \breve{\tau} q_{j,*}  - \alpha \breve{\tau} w^{n-1}.
$$
Solve
\begin{equation*}
\begin{gathered}
L(p_j - \tilde{p}_j) 
= D\B{R}(t_j, u_j, \tilde{p}_j) - \dot{H}(t_j) + \alpha  (Du^{n-1} - H(t^{n-1})),
\end{gathered}
\end{equation*}
Put $q_j = \breve{\tau}^{-1}(p_j - p_{j,*})$ and, for $r_{\sigma}=1$, $\dot{q}_j = \breve{\tau}^{-1}(q_j - q_{j,*})$.

Note: to evaluate $\B{R}(\ldots)$, one may reuse the implicit fluxes. The explicit fluxes should be recalculated if $c_j \ne \hat{c}_j$; otherwise it is optional.

3. Put
$$
u^n = u^{n-1} + \tau \sum\limits_{k=1}^{s} b_{k} \mathcal{H}_k^{(u)},
$$
$p^n = p_s$, and for $r_{\sigma}=1$ also $q^n = q_s$.

\subsection{Some properties}

In an IMEX method of the form \eqref{eq_IMEX_1}--\eqref{eq_IMEX_3}, each stage is concluded by the evaluation of explicit fluxes, which include the pressure gradient. This corresponds to the correction step of of projection methods. In constrast, if we use the form \eqref{eq_IMEX2_1}--\eqref{eq_IMEX2_4}, then the final value of pressure is not used in the evaluation of explicit terms. Thus, we have no means to satisfy the discrete continuity equation. Being not good itself, this raises suspicion about the use of Baumgarte method ($\alpha > 0$).

By the numerical study of the eigenvalues of the amplification matrices (see details in Section~\ref{sect:spectral2}), we obtain the stability range in $\alpha \tau$ for $r_{\sigma} = 0$ and $r_{\sigma} = 1$. This data is collected in  Table~\ref{table:IMEX2}.

\begin{table}[t]
\caption{\label{table:IMEX2}Stability range in $\alpha$ for SRK methods based on IMEX methods of the form \eqref{eq_IMEX2_1}--\eqref{eq_IMEX2_4}}
\begin{center}
\begin{tabular}{|c|c|c|c|c|c|c|c|c|c|}
\hline 
Class & Notation & Ref. & \multicolumn{2}{|c|}{$(\alpha\tau)_{\max}$} & \multicolumn{2}{|c|}{$\mathrm{CFL}_{\max}$} \\
& & & $r_{\sigma}=0$ & $r_{\sigma}=1$ & & $/ \sigma_{p}$ \\
\hline
A & SSP2(3,2,2) & \cite{Pareschi2003} & 0.66 & -- & -- & -- \\
A & SI-IMEX(3,3,2) & \cite{Boscarino2021} & 0.82 & -- & 1.73 & 0.57 \\
A & SI-IMEX(4,4,3) & \cite{Boscarino2021} & 1.09 & 1.08 & 1.74 & 0.43 \\
A & SI-IMEX(4,3,3) & \cite{Boscarino2024} & 1.09 & 1.08 & 2.82 & 0.7 \\
\hline
\end{tabular}
\end{center}
\end{table}

\subsection{Numerical results}

We put $\alpha \tau = 0.5 (\alpha \tau)_{max}$ where $(\alpha \tau)_{max}$ is defined in Table~\ref{table:IMEX2}.

For the case with a manufactured solution, the results are shown in Figure~\ref{fig:results_ctest1_IMEX2}. We see the drastic improvement in accuracy compared to the IMEX methods of the form \eqref{eq_IMEX_1}--\eqref{eq_IMEX_3} of type A. 

\begin{figure}
\centering
\includegraphics[width=0.3\textwidth]{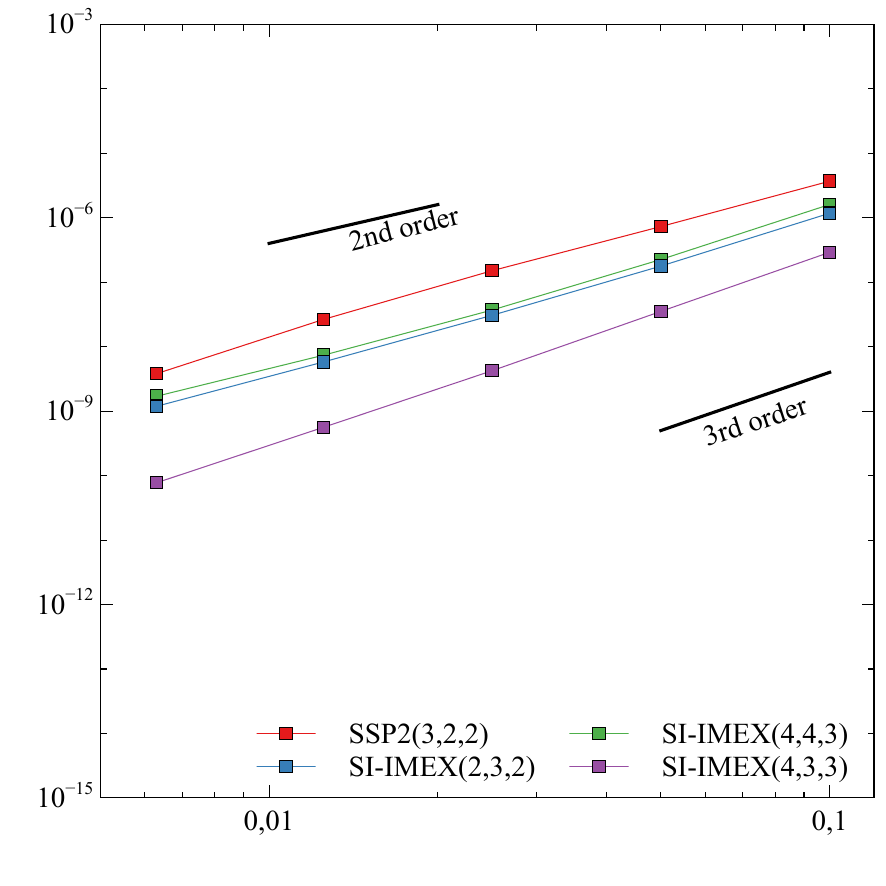}
\includegraphics[width=0.3\textwidth]{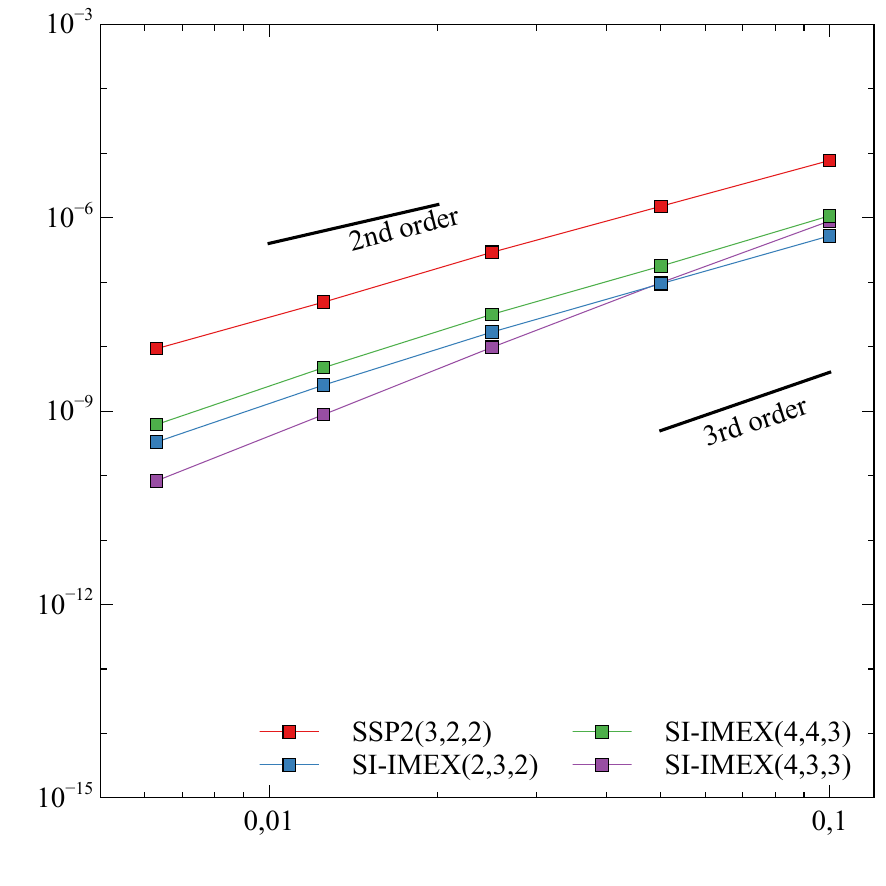}
\includegraphics[width=0.3\textwidth]{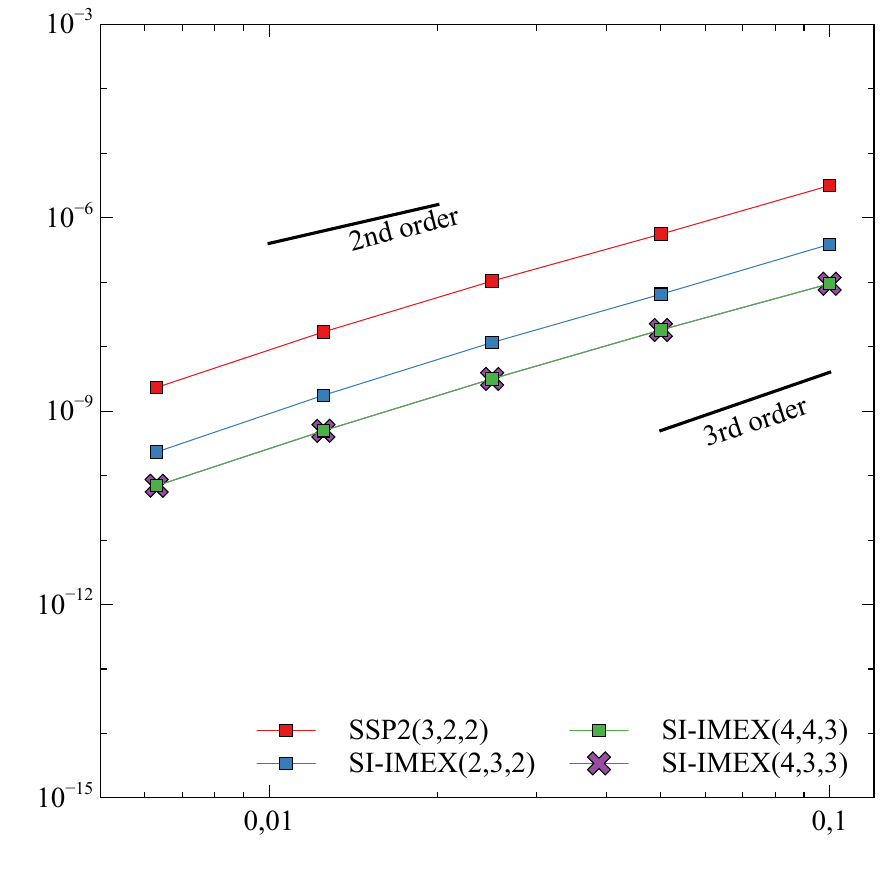}
\includegraphics[width=0.3\textwidth]{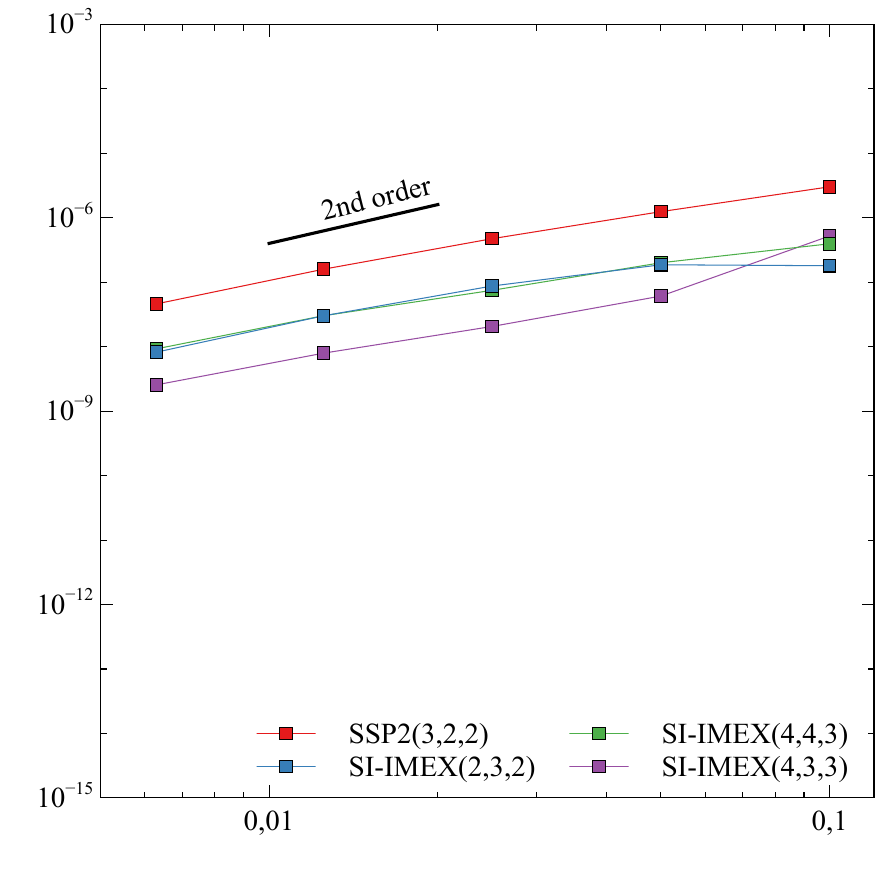}
\includegraphics[width=0.3\textwidth]{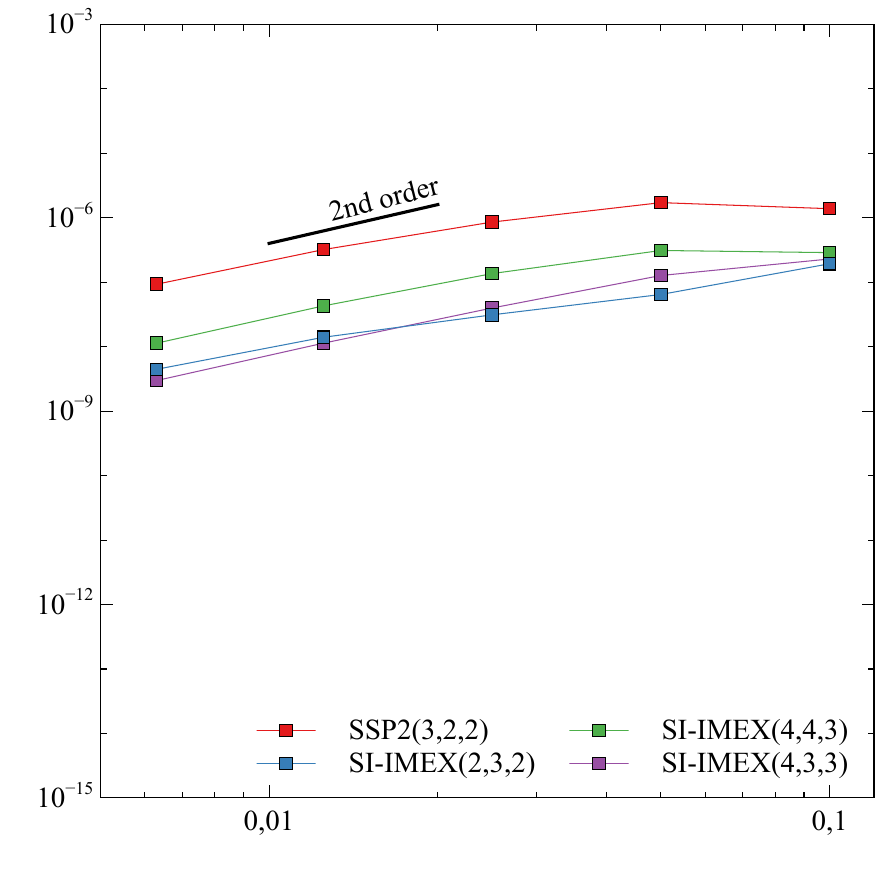}
\includegraphics[width=0.3\textwidth]{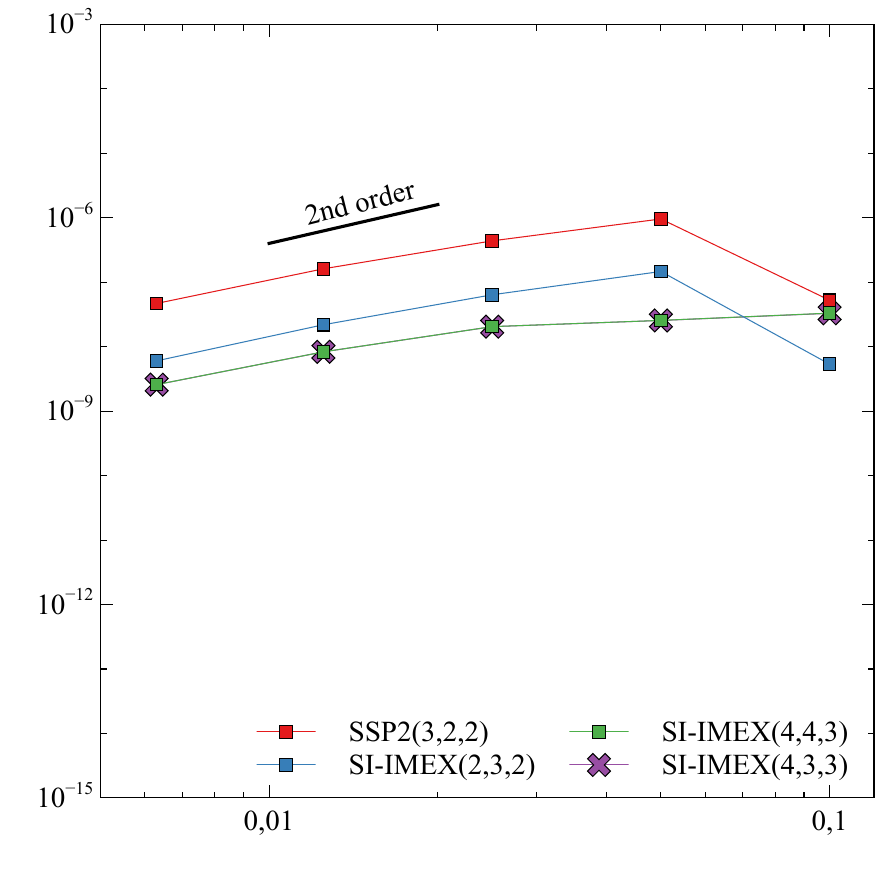}
\caption{Solution error versus timestep for the test with a manufactured solution. $r_{\sigma}=0$. Left to right: explicit SRK schemes, IMEX SRK schemes, implicit SRK schemes. Top: $e_u$; bottom: $e_p$}\label{fig:results_ctest1_IMEX2}
\end{figure}

We skip other simple verification tests and move straight to the inviscid 3D Taylor~-- Green vortex. The case is described in Section~\ref{sect:TGV_inviscid}. To be short, we present the results for the method SI-IMEX(4,3,3) only.

The results for the kinetic energy, the norm of the residual of the continuity equiation, and the norm of the velocity divergence are plotted in Fig.~\ref{fig:IMEXv2_TGV}.

\begin{figure}
\centering
\includegraphics[width=0.3\textwidth]{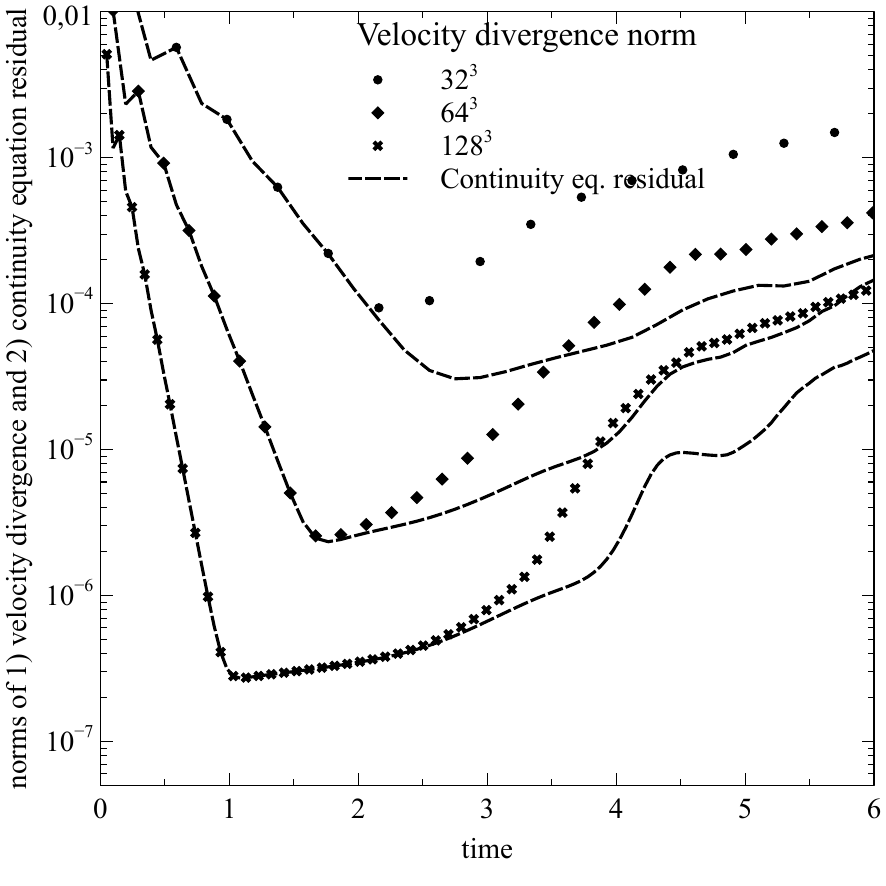}
\includegraphics[width=0.3\textwidth]{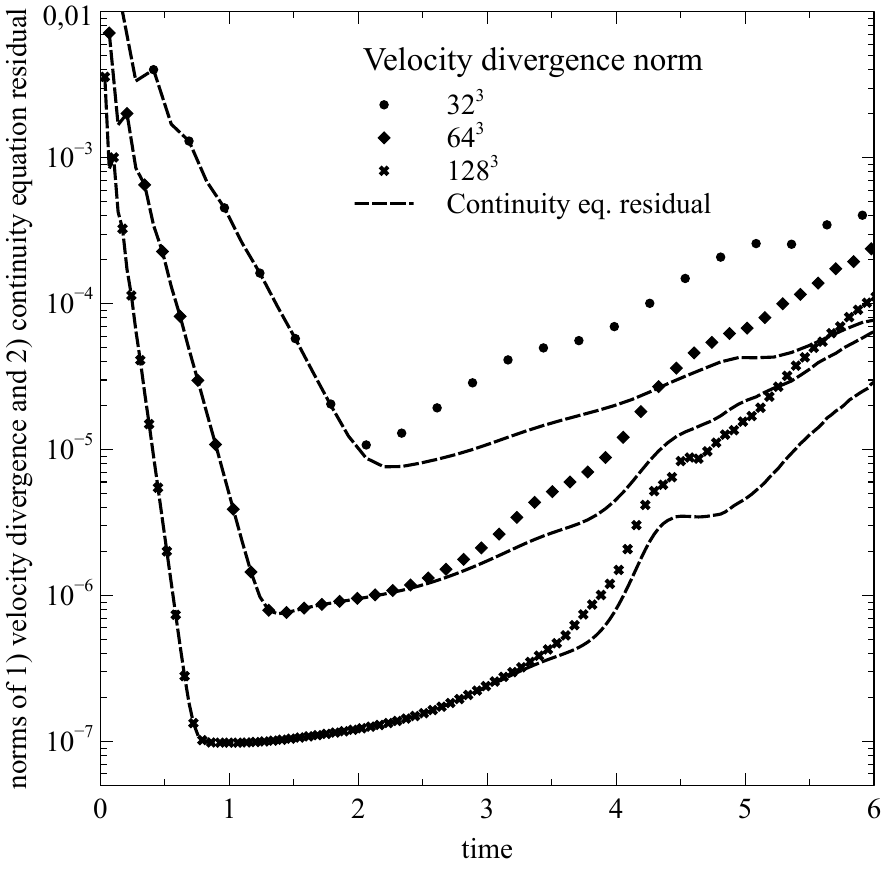}
\includegraphics[width=0.3\textwidth]{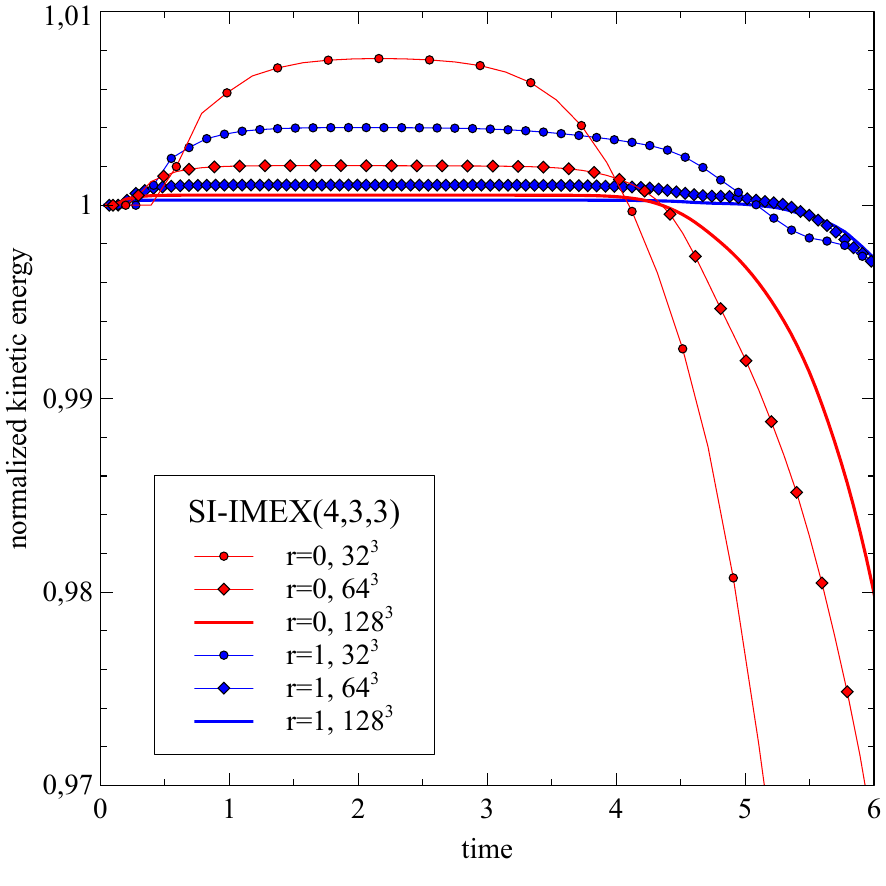}
\caption{Results for SI-IMEX(4,3,3). Left: velocity divergence and continuity equation residual norm for $r_{\sigma} = 0$. Middle: same for $r_{\sigma} = 1$. Right: kinetic energy}\label{fig:IMEXv2_TGV}
\end{figure}

The results are different to what we saw in Section~\ref{sect:TGV_inviscid}. The continuity equation has essentially nonzero residual. For a large $t$, the norm of this residual is smaller than the norm of the velocity divergence. For the kinetic energy, SI-IMEX(4,3,3) yields a non-monotone behavior. Although the kinetic energy and the velocity divergence norm of the divergence do converge to correct values as the mesh is refined, we believe that the schemes of type CK considered in the previous section are preferable.


\section{Conclusion}

In this paper, we presented a new family of segregated Runge~-- Kutta (SRK) methods for the time integration of differential-algebraic equations that result from spatial discretizations the incompressible Navier~-- Stokes equation. Original SRK methods \cite{Colomes2015} were not applicable if a pressure stabilization term is present in the continuity equation. Our methods circumvent this issue.

SRK schemes are based on implicit-explicit Runge~-- Kutta methods, which may be either of type A or of type CK. For IMEX methods of type A, the use of the form \eqref{eq_IMEX_1}--\eqref{eq_IMEX_3} results in a large solution error. The use of the form \eqref{eq_IMEX2_1}--\eqref{eq_IMEX2_4} improves the convergence rate, but the continuity equation is not exactly satisfied, and a nonmonotone behavior of the kinetic energy was observed. Method of type CK with $b \ne \hat{b}$ suffer from the same issue. Therefore, for the construction of SRK schemes, it is better to use IMEX RK methods of type CK that satisfy $b = \hat{b}$. Recall that type ARS is a sub-type of type CK.

Basing on the numerical results, we selected three IMEX methods for the use in SRK schemes:

\begin{itemize}

\item ARS(3,4,3) \cite{Ascher1997}. This is the only scheme of type ARS that shows good behavior for most of the tests we carried out. Type ARS means that the terms that should be taken implicitly, are indeed always taken implicitly (for other methods of type CK, the right-hand side of the momentum equation at the beginning of a timestep is taken explicitly). Although ARS(3,4,3) shows the second order convergence in pressure for all tests and this can be improved by switching to MARS(3,4,3), the latter yields an SRK scheme with worse stability properties, as shown in Fig.~\ref{fig:tgv2dviscunstr_r0} and Fig.~\ref{fig:tgv2dviscunstr_r1}.

\item BHR(5,5,3) \cite{Boscarino2009}. Most of IMEX methods suffer from the so-called order reduction phenomenon. SRK methods inherit this property. This means that the convergence of pressure is usually limited by the second order, except for some model cases. Among the methods we considered, only BHR(5,5,3)  does not suffer from the order reduction. This method is worth considering when solving problems with time-dependent boundary conditions. However, it has a slightly worse restriction on the Courant number then other methods (0.56 per implicit stage versus 0.94 for ARS(3,4,3) and 0.8 for ARK4(3)6L[2]SA).

\item ARK4(3)6L[2]SA \cite{Kennedy2003}. For tests with periodic conditions, it gives the best results in terms of the solution error, both in velocity and pressure. This is the only method that outperforms the 3-rd order multistep method for all tests we used for comparison. Despite being a monstrous construction with 6 explicit fluxes evaluations and 5 pressure solver calls per timestep, this is compensated by a good stability range ($\mathrm{CFL}_{\max} = 4$ for the linear, fully explicit case). 

\end{itemize}

The use of $r_{\sigma} = 0$ yields a better stability in practice, while the use of $r_{\sigma} = 1$ yields a better kinetic energy preservation. For scale-resolving simulations, we need to take $r_{\sigma}$ either equal or very close to 1.

\appendix
\section*{Appendix}
\section{Discussion on the initial data}
\label{sect:ID}

In this section we discuss the initial data for \eqref{eq_momentun_base}, \eqref{eq_continuity_base}. For brevity, put $f(t,u) = -\B{F}_{conv}(u) + \B{F}_{diff}(u) + \B{F}_{source}(t)$.

Recall that $S$ is a self-adjoint operator by assumption.  Let $S'\ :\ \mathrm{Im}\,S \to \mathrm{Im}\,S$ be the restriction of $S$ to $\mathrm{Im}\,S$. Let $L'\ :\ \mathrm{Ker}\,S \to \mathrm{Ker}\,S$ be the restriction of $\Pi_{\mathrm{Ker}} L \Pi_{\mathrm{Ker}}$ to $\Pi_{\mathrm{Ker}} S$. Both operators are invertible.

First consider the case $\sigma_0 > 0$, $\sigma_1 = 0$. Then formally, pressure is a unique function of velocity, as stated by the following lemma. 

\begin{lemma}
Let $u(t), p(t)$ be a solution of \eqref{eq_momentun_base}, \eqref{eq_continuity_base} with $\sigma_1 = 0$. Let $\Pi_{\mathrm{Im}}$ and $\Pi_{\mathrm{Ker}}$ be the orthogonal projections of $Q_h$ onto $\mathrm{Im}\,S$ and $\mathrm{Ker}\,S$, respectively. Then 
\begin{equation}
p(t) = (L')^{-1}\Pi_{\mathrm{Ker}} D f(t,u(t)) + ((L')^{-1} \Pi_{\mathrm{Ker}} L - I) (\sigma_0 S')^{-1} \Pi_{\mathrm{Im}} D u(t)
\label{eq_gen_pressure}
\end{equation}
where $I$ is the identity operator.
\end{lemma}
\begin{proof}
Denote $S_0 =  \sigma_0 S$, $S_0' = \sigma_0 S'$. Applying $\Pi_{\mathrm{Im}}$ to \eqref{eq_continuity_base} we get
$$
\Pi_{\mathrm{Im}} D u + S_0' \Pi_{\mathrm{Im}} p = 0,
$$
which gives $\Pi_{\mathrm{Im}} p = -(S_0')^{-1} \Pi_{\mathrm{Im}} D u$. Applying $\Pi_{\mathrm{Ker}} d/dt$ to \eqref{eq_continuity_base} we get
$$
\Pi_{\mathrm{Ker}} D \frac{du}{dt} = 0,
$$
which expands to $\Pi_{\mathrm{Ker}} D(Gp - f) = 0$.
Since $\Pi_{\mathrm{Ker}} DG = \Pi_{\mathrm{Ker}} L$, then
$$
\Pi_{\mathrm{Ker}} L p = \Pi_{\mathrm{Ker}} D f.
$$
Now represent $p = \Pi_{\mathrm{Im}} p + \Pi_{\mathrm{Ker}} p$. Then
\begin{equation}
L' \Pi_{\mathrm{Ker}} p = - \Pi_{\mathrm{Ker}} L \Pi_{\mathrm{Im}} p + \Pi_{\mathrm{Ker}} D f.
\label{eq_A1b}
\end{equation}
Thus, $\Pi_{\mathrm{Ker}} p = (L')^{-1}(\Pi_{\mathrm{Ker}} L (S_0')^{-1} \Pi_{\mathrm{Im}} D u + \Pi_{\mathrm{Ker}} D f)$.
\end{proof}

Consider two opposite cases. If $S=0$, then $p = L^{-1} D f(t,u)$, which complies with the pressure definition at the continuous level. However, if $S$ is invertible, we have $p = -(\sigma_0 S)^{-1} D u$. Since $S$ is a difference of two Laplace discretizations, its eigenvalues corresponding to long waves are very small. Therefore, the formula $p = -(\sigma_0 S)^{-1} D u$ is practically meaningless, and so is \eqref{eq_gen_pressure} in the case of a general $S$.

Now consider the case $\sigma_0 = 0$, $\sigma_1 > 0$. It is formally enough to specify $u$ and $\Pi_{\mathrm{Im}}p$. Then $\Pi_{\mathrm{Ker}}p$ may be recovered from \eqref{eq_A1b}. For the time derivatives, we have $\Pi_{\mathrm{Im}}(dp/dt) = -(\sigma_1 S')^{-1} \Pi_{\mathrm{Im}} Du$, and $\Pi_{\mathrm{Ker}}(dp/dt)$ may be found from the time derivative of \eqref{eq_A1b}. But these expressions are practically as useless as \eqref{eq_gen_pressure}. 

In all cases considered in this paper, we prescribe the initial data for $u$, $p$, and $dp/dt$. Hoewever, if $\sigma_1 = 0$, then the value of $dp/dt$ is not used, and even if $\sigma_1 \ne 0$, the use $dp/dt=0$ as the initial value does not significantly affect the results.

\end{document}